\documentclass[11pt, a4paper]{scrartcl}
\pdfoutput=1

\usepackage[utf8]{inputenc}
\usepackage[T1]{fontenc}
\usepackage{lmodern}
\usepackage[a4paper, left=2cm, right=2cm, top=2.5cm, bottom=2.5cm]{geometry}
\usepackage{enumitem}
\usepackage{amsmath}
\usepackage{amsthm}
\usepackage{tikz-cd}
\usepackage{url}
\usepackage{accents}

\usepackage{mathtools}
\usepackage{amssymb}
\usepackage{stmaryrd}
\usepackage{enumerate}
\usepackage{xcolor}
\definecolor{Myblue}{rgb}{0,0,0.6}  
\usepackage[colorlinks,allcolors=Myblue,breaklinks=true]{hyperref}

\usepackage[shortcuts]{extdash}

\newtheorem{Theorem}{Theorem}[subsection]
\newtheorem{Lemma}[Theorem]{Lemma}

\newtheorem{Proposition}[Theorem]{Proposition}

\theoremstyle{definition}
\newtheorem{Definition}[Theorem]{Definition}
\newtheorem{Construction}[Theorem]{Construction}
\theoremstyle{remark}
\newtheorem{Remark}[Theorem]{Remark}

\def\id{\mathrm{id}}
\def\Id{\mathrm{Id}}

\def\UModMonCat{\mathrm{UModMonCat}}
\def\ModMonCat{\mathrm{ModMonCat}}
\def\Pairs{\mathrm{Pairs}}
\def\BrTens{\mathrm{BrTens}}

\def\SC{\mathcal{C}}
\def\SD{\mathcal{D}}

\def\SM{\mathcal{M}}

\def\ST{\mathcal{T}}
\def\SV{\mathcal{V}}

\newcommand{\alp}[1]{\alpha^{(#1)}}

\newcommand{\alpd}[1]{\alpha^{(#1)'}}
\newcommand{\alpdd}[1]{\alpha^{(#1)''}}

\newcommand{\com}[1]{\textcolor{black}{\tiny\boxed{#1}}}
\newcommand{\nat}{\text{nat.\,}}
\newcommand{\dfn}{\text{def.\,}}
\newcommand{\funct}{\text{funct.\,}}

\def\1{{\mathbf1}}

\title{Module Monoidal Categories as\\
Categorification of Associative Algebras}
\author{Sebastian Heinrich\\[0.5cm]
\normalsize{\texttt{\href{mailto:sebastian.heinrich@uni-hamburg.de}{sebastian.heinrich@uni-hamburg.de}}}\\[0.1cm]
{\normalsize\slshape Fachbereich Mathematik, Universit\"at Hamburg,}\\
{\normalsize\slshape Bundesstra\ss e 55, 20146 Hamburg, Germany}\\[-0.1cm]
}
\date{}

\begin{document}
\maketitle

\begin{abstract}
In \cite{HPT15}, the notion of a module tensor category was introduced as a braided monoidal central functor $F\colon \SV\longrightarrow \ST$ from a braided monoidal category $\SV$ to a monoidal category $\ST$, which is a monoidal functor $F\colon \SV\longrightarrow\ST$ together with a braided monoidal lift $F^Z\colon \SV\longrightarrow Z(\ST)$ to the Drinfeld center of $\ST$. This is a categorification of a unital associative algebra $A$ over a commutative ring $R$ via a ring homomorphism $f\colon R\longrightarrow Z(A)$ into the center of $A$.
In this paper, we want to categorify the characterization of an associative algebra as a (not necessarily unital) ring $A$ together with an $R$\==module structure over a commutative ring $R$, such that multiplication in $A$ and action of
$R$ on $A$ are compatible. In doing so, we introduce the more general notion of \emph{non\--unital module monoidal categories} and obtain 2\==categories of non\--unital and unital module monoidal categories, their functors and natural transformations. We will show that in the unital case the latter definition is equivalent to the definition in \cite{HPT15} by explicitly writing down an equivalence of 2\==categories.
\end{abstract}
\newpage

\tableofcontents\newpage

\section{Introduction and Summary}\label{secIntroduction}
In this paper, we formulate a categorification for non\--unital and unital  associative algebras over unital commutative rings. Recall that a non\--unital associative algebra over a commutative ring $R$ is a non\--unital ring $A$ carrying the structure of an $R$\==module, such that the multiplication in $A$ and the $R$\==action on $A$ are compatible in the sense that $r.(ab)=(r.a)b=a(r.b)$ for all $r\in R$ and $a,b\in A$.

If $A$ happens to be unital as a ring, then the algebra is called unital. In this case, it can equivalently be described as a ring homomorphism $f\colon R\longrightarrow Z(A)$ from $R$ into the center of $A$. Note, that these characterizations fail to be equivalent if the algebra is non\--unital: Let $k$ be a field and consider the ideal $Xk[X]$ in the polynomial ring $k[X]$. It inherits the structure of a (non\--unital) ring as well as a compatible $k$\==module structure from $k[X]$, 
turning it into a $k$\==algebra. However, the unique ring homomorphism $f\colon k\longrightarrow Xk[X]$ is trivial.

Now, to categorify the second characterization of an unital associative algebra from above, i.e.\ that via a ring homomorphism $f\colon R\longrightarrow Z(A)$, one makes the usual approach to take a monoidal category $\ST$ for the ring $A$, the Drinfeld center $Z(\ST)$ for its center and a braided monoidal category $\SV$ for the commutative ring $R$. For the ring homomorphism $f\colon R\longrightarrow Z(A)$, one needs the notion of a braided central monoidal functor $F\colon\SV\longrightarrow\ST$ \cite{Bezrukavnikov11,DMNO11}, which is a monoidal functor $F\colon \SV\longrightarrow\ST$ together with a braided monoidal lift $F^Z\colon \SV\longrightarrow Z(\ST)$ to the Drinfeld center of $\ST$, i.e.\ $F=R\circ F^Z$ with $R$ the forgetful functor. This was used as a definition of a module tensor category in \cite[Sec.\,3.2]{HPT15}. Namely, a module tensor category over $\SV$ is a pair $(\ST,F)$ consisting of a monoidal category $\ST$ and a braided central monoidal functor $F\colon \SV\longrightarrow \ST$. We refer to this definition of a module tensor category as \emph{pair} in this work. 

From this, assuming more structure, one can prove an equivalence between pointed pivotal module tensor categories and anchored planar algebras \cite{HPT15,HPT16} and define unitarity for both notions \cite{HPT23}. Pairs appeared as 1\==morphisms in the 4\==category $\BrTens$ of braided monoidal tensor categories \cite{BJS21} as well as a tool for showing results on Landau\--Ginzburg/\linebreak[0]Conformal Field Theory correspondence \cite{CW22} and for describing boundary conditions in 3d TFTs \cite{FSV13}.

In this work, we want to categorify the first characterization of associative algebras from above, i.e.\ that via a compatible $R$\==module structure, as \emph{(non\--unital) module monoidal categories\footnote{Since the term \emph{tensor category} is often meant to imply linearity, rigidity and finiteness conditions \cite[Def.\,4.1.1]{EGNO}, here the term \emph{monoidal category} is used.}}. A non\--unital module monoidal category $\SM$ over $\SV$ is a tuple $(\SM,\otimes,a,\rhd,m,l^\rhd,\alp1,\alp2)$ with $(\SM,\otimes,a)$ a monoidal category, $(\SM,\rhd,m,l^\rhd)$ a module category over a braided monoidal category $\SV$ and $\alp1$, $\alp2$ coherence isomorphisms reflecting the equations $r.(ab)=(r.a)b$ and $r.(ab)=a(r.b)$, respectively. Naturally, these data have to make several coherence diagrams commute. For unital module monoidal categories, one requires the monoidal category $\SM$ to be unital.

Note that, besides the viewpoints of pairs and unital module monoidal categories, these objects can under certain assumptions equivalently be described as monoidal categories enriched in braided monoidal categories \cite{MP17,MPP18,JMPP19,KYZZ21,Dell21,DHP22}.

Similar to module monoidal categories, but not the same, is the notion of augmented monoidal categories studied in \cite{Laugwitz19}. These are a categorification of an algebra over a commutative ring which is simultaneously augmented in the base ring.
%SH: Neu
Another similar notion is that of a $\SV$\==central monoidal category used in \cite{LW} to define relative Drinfeld centers as M\"uger centralizers.

The first main result of this paper is the construction of the 2\==categories $\ModMonCat$ (cf.\ Def.\,\ref{def2CategoryModMonCat}) and $\UModMonCat$ (cf.\ Def.\,\ref{def2CategoryUModMonCat}) of non\--unital (resp.\ unital) module monoidal categories, their functors and natural transformations.

The second main result is the equivalence of $\UModMonCat$ to the 2\==category $\Pairs$ (cf.\ Def.\,\ref{def2CategoryPairs}) consisting of pairs $(\ST,F)$ as above, their morphisms and their 2\==morphisms, stated in Theorem \ref{thmEquivalenceUnitalModuleMonoidalCategoriesAndPairs}. An overview of both 2\==categories is provided in Table~\ref{tab:overview}.

\bigskip
\textbf{Theorem \ref{thmEquivalenceUnitalModuleMonoidalCategoriesAndPairs}. }
Let $\SV$ be a braided monoidal category. There is an equivalence of 2\==categories
\[\UModMonCat\cong \Pairs.\]

\begin{table}[t]
    \centering
    \begin{tabular}{c|c|c}
    &$\UModMonCat$ & $\Pairs$  \\\hline
    objects & $\begin{array}{c}\text{unital module monoidal categories}\\ (\SM,\otimes,\1,a,l,r,\rhd,m,l^\rhd,\alp1,\alp2)\\(\text{Def.\,}\ref{defUnitalModuleMonoidalCategory})\end{array}$ & $\begin{array}{c}\text{pairs}\\ (\ST,F)\\(\text{Def.\,}\ref{defPair})\end{array}$\\\hline
    1\==morphisms & $\begin{array}{c}\text{unital module monoidal functors}\\ (F,\varphi_2,\varphi_0,s)\colon \SM\longrightarrow\SM'\\(\text{Def.\,}\ref{defUnitalModuleMonoidalFunctor})\end{array}$ & $\begin{array}{c}\text{morphisms of pairs}\\ (G,\varphi_2,\varphi_0,\gamma)\colon (\ST,F)\longrightarrow(\ST',F')\\(\text{Def.\,}\ref{defMorphismOfPairs})\end{array}$ \\\hline
    2\==morphisms & $\begin{array}{c}\text{unital module monoidal}\\\text{natural transformations}\\(\text{Def.\,}\ref{defUnitalModuleMonoidalNaturalTransformation})\end{array}$ & $\begin{array}{c}\text{2\==morphisms of pairs}\\(\text{Def.\,}\ref{def2MorphismOfPairs})\end{array}$
    \end{tabular}
    \caption{The 2\==categories $\UModMonCat$ and $\Pairs$.}
    \label{tab:overview}
\end{table}

The advantage of the categorification of associative algebras presented in this paper is the more general non\--unital case of module monoidal categories, which is interesting because, for example, they arise from so\--called orbifold data \cite{CRS20,Mulevicius22}, see \cite{HR}.

We finish the introduction with a short overview of the content of this paper. We start by briefly recalling some basics on braided monoidal and module category theory in Section \ref{secPreliminaries} to state the notation and conventions used in this paper. From that, we define the notions of non\--unital module monoidal category, their functors and their 2\==morphisms as well as their unital versions in Section \ref{secDefinitions}. In Section \ref{secPairs}, we recall the definitions of pairs, morphisms of pairs and 2\==morphisms of pairs. Finally, in Section \ref{secEquivalence}, we state and prove our main theorem \ref{thmEquivalenceUnitalModuleMonoidalCategoriesAndPairs} from above.

\subsection*{Acknowledgements}
I would like to thank Ingo Runkel for many prosperous discussions and comments as well as Vincentas Mulevi\v{c}ius for comments on the draft.

\section{Preliminaries}\label{secPreliminaries}

We assume that the reader is used to the notion of a (braided) monoidal category. However, for the sake of notation and better referability later on, we quickly recall some important definitions. In this process, commutative diagrams are labelled by abbreviations related to the conditions they represent, rather than equation numbers, to make their later use in proofs easier to follow. All abbreviations are collected in Appendix \ref{secAppendixGlossary}.
\subsection{Non--unital and unital monoidal categories}

\begin{Definition}[{\cite[Def.\,2.2.8, Rem.\,2.2.9]{EGNO}}]\label{defMonoidalCategory}
A \emph{non\--unital monoidal category} (or \emph{semigroup category}) is a triple $(\SV, \otimes, a)$ where $\SV$ is a category, $\otimes\colon \SV\times\SV\longrightarrow\SV$ is the \emph{tensor functor} or \emph{tensor product} and $a\colon(-\otimes -)\otimes - \Longrightarrow -\otimes (-\otimes -)$ is the \emph{associator} natural isomorphism. These data have to satisfy the pentagon (or associativity) axiom, i.e.\ the diagram
\begin{equation}\tag{$\text{pent.\,}a$}\label{pentA}
\begin{tikzcd}
 & ((u\otimes v)\otimes w)\otimes x \arrow[rd, "{a_{u\otimes v,w,x}}"] \arrow[ld, "{a_{u,v,w}\,\otimes\,\id_x}"'] &\\
(u\otimes(v\otimes w))\otimes x \arrow[d, "{a_{u,v\otimes w,x}}"'] && (u\otimes v)\otimes (w\otimes x) \arrow[d, "{a_{u,v,w\otimes x}}"] \\
u\otimes((v\otimes w)\otimes x) \arrow[rr, "{\id_u\,\otimes\,a_{v,w,x}}"'] && u\otimes(v\otimes (w\otimes x)) 
\end{tikzcd}
\end{equation}
commutes for all $u,v,w,x\in\SV$. A non\--unital monoidal category is called \emph{strict} if the associator is the identity.

A \emph{unital monoidal category} is a sextuple $(\SV, \otimes, \1_\SV, a, l, r)$ where $(\SV,\otimes,a)$ is a non\--unital monoidal category, $\1_\SV$ is the \emph{unit object} of $\SV$ and $l\colon \1_\SV\otimes\strut -\Longrightarrow \Id_\SV$ and $r\colon -\strut\otimes \1_\SV\Longrightarrow \Id_\SV$ are the \emph{left} and \emph{right unit constraint} natural isomorphisms. These data have to satisfy the triangle (or unitality) axiom, i.e.\ the diagram
\begin{equation}\tag{$\text{tri.\,}a$}\label{triangleA}
\begin{tikzcd}
(v\otimes \1_\SV)\otimes w \arrow[rr, "{a_{v,\1_\SV,w}}"] \arrow[rd, "{r_v\,\otimes\,\id_w}"'] && v\otimes (\1_\SV\otimes w) \arrow[ld, "{\id_v\,\otimes\,l_w}"] \\
 & v\otimes w & 
\end{tikzcd}
\end{equation}
commutes for all $v,w\in\SV$. A unital monoidal category is called \emph{strict} if the natural isomorphisms $a$, $l$ and $r$ are identities.
\end{Definition}

In any unital monoidal category, two diagrams similar to \eqref{triangleA}, namely
\begin{equation}\tag{$a\,\&\,l$}\label{aL}
\begin{tikzcd}
(\1_\SV\otimes v)\otimes w \arrow[rr, "{a_{\1_\SV,v,w}}"] \arrow[rd, "{l_v\,\otimes\,\id_w}"'] && \1_\SV\otimes (v\otimes w) \arrow[ld, "{l_{v\otimes w}}"] \\
 & v\otimes w & 
\end{tikzcd}
\end{equation}
and
\begin{equation}\tag{$a\,\&\,r$}\label{aR}
\begin{tikzcd}
(v\otimes w)\otimes \1_\SV \arrow[rr, "{a_{v,w,\1_\SV}}"] \arrow[rd, "{r_{v\otimes w}}"'] && v\otimes (w\otimes \1_\SV) \arrow[ld, "{\id_v\,\otimes\,r_w}"] \\
 & v\otimes w & 
\end{tikzcd}
\end{equation}
also commute for all $v,w\in\SV$. A proof can be found in \cite[Prop.\,2.2.4]{EGNO}.

\begin{Definition}[{\cite[Def.\,2.4.5]{EGNO}}]\label{defMonoidalFunctor}
A \emph{non\--unital monoidal functor} from a non\--unital monoidal category $\SV=(\SV,\otimes,a)$ to a non\--unital monoidal category $\SV'=(\SV',\otimes',a')$ is a tuple $(F,\varphi_2)\colon\SV\longrightarrow\SV'$ where $F\colon\SV\longrightarrow\SV'$ is a functor and $\varphi_2\colon F(-)\otimes' F(-)\Longrightarrow F(-\otimes -)$ is the \emph{monoidal structure} natural isomorphism making the hexagon
\begin{equation}\tag{$\text{hex.\,}\varphi_2$}\label{phi2Hexagon}
\begin{tikzcd}[column sep = huge]
(F(u)\otimes' F(v))\otimes' F(w) \arrow[d, "{a'_{F(u),F(v),F(w)}}"'] \arrow[r, "{\varphi_{2,u,v}\,\otimes'\,\id_{F(w)}}"] & F(u\otimes v)\otimes' F(w) \arrow[r, "{\varphi_{2,u\otimes v,w}}"]& F((u\otimes v)\otimes w) \arrow[d, "{F(a_{u,v,w})}"] \\
F(u)\otimes'(F(v)\otimes' F(w)) \arrow[r, "{\id_{F(u)}\,\otimes'\,\varphi_{2,v,w}}"'] & F(u)\otimes' F(v\otimes w) \arrow[r, "{\varphi_{2,u,v\otimes w}}"'] & F(u\otimes (v\otimes w))
\end{tikzcd}
\end{equation}
commute for all $u,v,w\in\SV$.

A \emph{unital monoidal functor} from a unital monoidal category $\SV$ to a unital monoidal category $\SV'$ is a tuple $(F,\varphi_2,\varphi_0)\colon\SV\longrightarrow\SV'$ where $(F,\varphi_2)$ is a non\--unital monoidal functor and $\varphi_0\colon\1_{\SV'}\longrightarrow F(\1_\SV)$ is an isomorphism in $\SV'$ making the squares
\begin{equation}\tag{$\varphi_0\,\&\,l$}\label{phi0l}
\begin{tikzcd}
\1_{\SV'}\otimes'F(v) \arrow[d, "{\varphi_0\,\otimes'\,\id_{F(v)}}"'] \arrow[r, "l'_{F(v)}"] & F(v)\\
F(\1_\SV)\otimes'F(v) \arrow[r, "{\varphi_{2,\1_\SV,v}}"'] & F(\1_\SV\otimes v) \arrow[u, "F(l_v)"']
\end{tikzcd}
\end{equation}
and
\begin{equation}\tag{$\varphi_0\,\&\,r$}\label{phi0r}
\begin{tikzcd}
F(v)\otimes' \1_{\SV'} \arrow[r, "r'_{F(v)}"] \arrow[d, "{\id_{F(v)}\,\otimes'\,\varphi_0}"'] & F(v) \\
F(v)\otimes' F(\1_\SV) \arrow[r, "{\varphi_{2,v,\1_\SV}}"'] & F(v\otimes\1_\SV) \arrow[u, "F(r_v)"']
\end{tikzcd}
\end{equation}
commute for all $v\in\SV$.
\end{Definition}

\begin{Remark}\label{remCompMonoidalFunctors}
Given three non\--unital (resp.\ unital) monoidal categories $\SV$, $\SV'$, $\SV''$ and two non\--unital (resp.\ unital) monoidal functors $F\colon\SV\longrightarrow\SV'$ and $G\colon\SV'\longrightarrow\SV''$, their composition $(H,\varphi_2^{H})$ (resp.\ $(H,\varphi_2^H,\varphi_0^H)$) with $H\coloneqq GF$ and
\begin{align}\tag{compos.\,$\varphi_2$}\label{composPhi2}
\varphi_{2,v,w}^{H}&\coloneqq \left[GF(v)\otimes''GF(w)\xrightarrow[]{\varphi_{2,F(v),F(w)}^G}G(F(v)\otimes'F(w))\xrightarrow[]{G(\varphi_{2,v,w}^F)}GF(v\otimes w)\right],\\\tag{compos.\,$\varphi_0$}\label{composPhi0}
\varphi^H_0&\coloneqq\left[\1_{\SV''}\xrightarrow[]{\varphi_0^G}G(\1_{\SV'})\xrightarrow[]{G(\varphi_0^F)}GF(\1_\SV)\right]
\end{align}
is again a non\--unital (resp.\ unital) monoidal functor (cf.\ \cite[Rem.\,2.4.7]{EGNO}).
\end{Remark}

As it is done here, we usually ease notation by omitting the $\circ$ for the composition of functors. Further, we might give definitions and equations of morphisms in the notation above, which has to be read e.g.\ for the second equation as $\varphi^H_0\coloneqq G(\varphi_0^F)\circ \varphi_0^G$.

\begin{Definition}\label{defMonoidalNaturalTransformation}
Let $\SV$, $\SV'$ be two non\--unital monoidal categories and let $F,G\colon\SV\longrightarrow\SV'$ be two non\--unital monoidal functors between them. A \emph{non\--unital monoidal natural transformation} $\eta\colon F\Longrightarrow G$ is a natural transformation which is compatible with the monoidal structures on $F$ and $G$, i.e.\ the diagram
\begin{equation}\tag{nat.\,tr.\,\&\,$\varphi_2$}\label{monNatTrafoPhi2}
\begin{tikzcd}
F(v)\otimes' F(w) \arrow[d, "\eta_v\,\otimes'\,\eta_w"'] \arrow[r, "{\varphi_{2,v,w}^F}"] & F(v\otimes w) \arrow[d, "\eta_{v\otimes w}"] \\
G(v)\otimes' G(w) \arrow[r, "{\varphi_{2,v,w}^G}"']& G(v\otimes w) 
\end{tikzcd}
\end{equation}
commutes for all $v,w\in\SV$.

For $\SV$, $\SV'$ unital monoidal categories and $F,G$ unital monoidal functors, we define $\eta\colon F\Longrightarrow G$ to be a \emph{unital monoidal natural transformation} if it is a non\--unital monoidal natural transformation compatible with the unital monoidal structures on $F$ and $G$, i.e.\ the diagram
\begin{equation}\tag{nat.\,tr.\,\&\,$\varphi_0$}\label{monNatTrafoPhi0}
\begin{tikzcd}
& F(\1_\SV) \arrow[dd, "\eta_{\1_\SV}"] \\
\1_{\SV'} \arrow[rd, "\varphi_0^G"'] \arrow[ru, "\varphi_0^F"] & \\
& G(\1_\SV)
\end{tikzcd}
\end{equation}
commutes.
\end{Definition}

\begin{Definition}\label{defEquivalenceMonoidalCategories}
Let $\SV$, $\SV'$ be two (non\==)unital monoidal categories and let $F\colon \SV\longrightarrow\SV'$ be a (non\==)unital monoidal functor. We call $F$ an \emph{equivalence of (non\==)unital monoidal categories} if there exist a (non\==)unital monoidal functor $G\colon \SV'\longrightarrow\SV$ and (non\==)unital monoidal natural isomorphisms $\eta\colon \Id_\SV\Longrightarrow GF$ and $\varepsilon\colon F G\Longrightarrow\Id_{\SV'}$.
\end{Definition}

Next, we see a very useful statement on equivalences of (non\==)unital monoidal categories, which highly reduces the amount of data needed to get an equivalence of (non\==)unital monoidal categories.

Before we do this, recall from \cite[Thm\,3.1.5\,(2)]{Borceux} the definition of an \emph{adjunction} $(F,G,\eta,\varepsilon)$ from a category $\SC$ to a category $\SD$, consisting of two functors $F\colon \SC\longrightarrow \SD$, $G\colon\SD\longrightarrow \SC$ and natural transformations $\eta\colon\Id_\SC\Longrightarrow GF$, $\varepsilon\colon FG\Longrightarrow \Id_\SD$ satisfying the zig\--zag\--relations
\begin{equation}\label{adjunctionZigzag}\tag{zig--zag}
\begin{aligned}
\left[F(c)\xrightarrow[]{F(\eta_c)}FGF(c)\xrightarrow[]{\varepsilon_{F(c)}}F(c)\right]&=\left[F(c)\xrightarrow[]{\id_{F(c)}}F(c)\right],\\
\left[G(d)\xrightarrow[]{\eta_{G(d)}}GFG(d)\xrightarrow[]{G(\varepsilon_d)}G(d)\right]&=\left[G(d)\xrightarrow[]{\id_{G(d)}}G(d)\right]
\end{aligned}
\end{equation}
for all $c\in\SC$ and $d\in\SD$. In this situation, $F$ is called the \emph{left adjoint} and $G$ is called the \emph{right adjoint}.

\begin{Definition}[{\cite[IV.4, p.\,93]{MacLane}}]\label{defAdjointEquivalence}
Let $\SC$, $\SD$ be categories. An \emph{adjoint equivalence} between $\SC$ and $\SD$ is an adjunction $(F,G,\eta,\varepsilon)$ such that $\eta$ and $\varepsilon$ are natural isomorphisms.
\end{Definition}

Obviously, any adjoint equivalence of categories is an equivalence of categories. In \cite[Thm.\,IV.4.1]{MacLane} it is proved that ``being adjoint equivalent'' is equivalent to ``being equivalent'' and further that, given an equivalence $F$, it is always possible to find an adjoint equivalence $(F,G,\eta,\varepsilon)$ containing the very same $F$.

\begin{Proposition}\label{propAlternativeEquivalenceMonoidalCategories}
Let $\SV$, $\SV'$ be two (non\==)unital monoidal categories and let $F\colon \SV\longrightarrow\SV'$ be a (non\==)unital monoidal functor. Then the following statements are equivalent:
\begin{itemize}
    \item[(i)] $F$ is an equivalence of (non\==)unital monoidal categories.
    \item[(ii)] $F$ is part of an adjoint equivalence of (non\==)unital monoidal categories.
    \item[(iii)] $F$ is an equivalence of ordinary categories.
\end{itemize}
\end{Proposition}

\begin{proof}
    By \cite[Prop.\,I.4.4.2]{Rivano}, (iii) is equivalent to (ii). Obviously, (ii) implies (i) and (i) implies (iii).
\end{proof}

From now on, by monoidal categories/functors/natural transformations we mean unital monoidal categories/functors/natural transformations. If something is meant to be non\--unital, we will explicitly write so.

\subsection{Braided monoidal categories}\label{secBasicsBraidedCategories}
The definitions of a pair and a module monoidal category will be ``over a braided monoidal category~$\SV$''. Therefore, we quickly recall the notion of a braided (unital) monoidal category and its functors. Note, that there is also an alternative description of braided monoidal categories using the so--called $b$--structure, which also investigates non--unital braided monoidal categories \cite{DR13}.

\begin{Definition}[{\cite[Def.\,8.1.1\,\&\,8.1.2]{EGNO}}]\label{defBraidedCategory}
A \emph{braiding} on a monoidal category $\SV$ is a natural isomorphism $\beta^\SV\colon -\otimes -\Longrightarrow - \otimes -$ with components
\[\beta_{v,w}^\SV\colon v\otimes w\longrightarrow w\otimes v\]
for all $v,w\in\SV$ which makes the two \emph{braiding hexagons}
\begin{equation}\tag{br.\,hex.\,1}\label{braidingHexagon1}
\begin{tikzcd}[column sep = huge]
(u\otimes v)\otimes w \arrow[r, "{a_{u,v,w}}"] \arrow[d, "{\beta^\SV_{u,v}\,\otimes\,\id_w}"'] & u\otimes (v\otimes w) \arrow[r, "{\beta^\SV_{u,v\otimes w}}"] & (v\otimes w)\otimes u \arrow[d, "{a_{v,w,u}}"] \\
(v\otimes u)\otimes w \arrow[r, "{a_{v,u,w}}"']& v\otimes (u\otimes w) \arrow[r, "{\id_v\,\otimes\,\beta_{u,w}^\SV}"'] & v\otimes (w\otimes u) 
\end{tikzcd}
\end{equation}
and
\begin{equation}\tag{br.\,hex.\,2}\label{braidingHexagon2}
\begin{tikzcd}[column sep = huge]
u\otimes(v\otimes w) \arrow[r, "{a_{u,v,w}^{-1}}"] \arrow[d, "{\id_u\,\otimes\,\beta_{v,w}^\SV}"'] & (u\otimes v)\otimes w \arrow[r, "{\beta^\SV_{u\otimes v,w}}"] & w\otimes (u\otimes v) \arrow[d, "{a_{w,u,v}^{-1}}"] \\
u\otimes(w\otimes v) \arrow[r, "{a_{u,w,v}^{-1}}"']& (u\otimes w)\otimes v \arrow[r, "{\beta^\SV_{u,w}\,\otimes\,\id_v}"'] & (w\otimes u)\otimes v
\end{tikzcd}
\end{equation}
commute for all $u,v,w\in\SV$. A \emph{braided} monoidal category $(\SV,\beta^\SV)$ is a monoidal category together with a braiding on it.
\end{Definition}
Note, that there are possibly different choices for a braiding on a given monoidal category $\SV$.

It is not necessary to impose the right behavior on the braiding and the unit (constraints) of the monoidal category $\SV$, since it holds for free as the following statement shows.

\begin{Proposition}[{\cite[Prop.\,1.1]{JoyalStreet}}]\label{propPropertiesBraidingAndUnit}
Let $\SV$ be a braided monoidal category. Then the diagrams
\begin{equation}\label{braidingV1}
\begin{tikzcd}
v\otimes\1_\SV \arrow[rr, "{\beta_{v,\1_\SV}^\SV}"] \arrow[rd, "r_v"'] & & \1_\SV\otimes v \arrow[ld, "l_v"] \\
 & v &
\end{tikzcd}
\end{equation}
and
\begin{equation}\label{braiding1V}
\begin{tikzcd}
\1_\SV\otimes v \arrow[rr, "{\beta_{\1_\SV,v}^\SV}"] \arrow[rd, "l_v"'] & & v\otimes\1_\SV \arrow[ld, "r_v"] \\
& v & 
\end{tikzcd}
\end{equation}
commute for all $v\in\SV$.
\end{Proposition}
The reader can find a proof in the reference.

\begin{Definition}[{\cite[Def.\,8.1.7]{EGNO}}]\label{defBraidedMonoidalFunctor}
Let $\SV$ and $\SV'$ be braided monoidal categories. A \emph{braided monoidal functor} $(F,\varphi_2,\varphi_0)\colon \SV\longrightarrow \SV'$ is a monoidal functor making the diagram
\begin{equation}\tag{$\beta\,\&\,\varphi_2$}\label{braidedMonoidalFunctor}
\begin{tikzcd}[column sep = huge]
F(v)\otimes' F(w) \arrow[d, "{\varphi_{2,v,w}}"'] \arrow[r, "{\beta^{\SV'}_{F(v),F(w)}}"] & F(w)\otimes' F(v) \arrow[d, "{\varphi_{2,w,v}}"] \\
F(v\otimes w) \arrow[r, "{F(\beta^\SV_{v,w})}"'] & F(w\otimes v)
\end{tikzcd}
\end{equation}
commute for all $v,w\in\SV$.
\end{Definition}

\begin{Remark}\label{remCompBraidedMonoidalFunctors}
The notion of a ``braided monoidal natural transformation'' does not exist; morphisms between braided monoidal functors are just monoidal natural transformations (cf.\ Def.\,\ref{defMonoidalNaturalTransformation}).

Given three braided monoidal categories $\SV$, $\SV'$, $\SV''$ and two braided monoidal functors $F\colon\SV\longrightarrow\SV'$ and $G\colon\SV'\longrightarrow\SV''$, their composition $(H,\varphi_2^H,\varphi_0^H)$ of monoidal functors as in Remark \ref{remCompMonoidalFunctors} is again a braided monoidal functor.
\end{Remark}

\begin{Definition}\label{defEquivalenceBraidedMonoidalCategories}
Let $\SV$, $\SV'$ be two braided monoidal categories and let $F\colon \SV\longrightarrow\SV'$ be a braided monoidal functor. We call $F$ an \emph{equivalence of braided monoidal categories} if there exist a braided monoidal functor $G\colon \SV'\longrightarrow\SV$ and monoidal natural isomorphisms $\eta\colon \Id_\SM\Longrightarrow GF$ and $\varepsilon\colon FG\Longrightarrow\Id_{\SM'}$.
\end{Definition}

As before for monoidal categories (cf.\ Prop.\,\ref{propAlternativeEquivalenceMonoidalCategories}), we get the following statement on equivalences of braided monoidal categories.

\begin{Proposition}\label{propAlternativeEquivalenceBraidedMonoidalCategories}
Let $\SV$, $\SV'$ be two braided monoidal categories and let $F\colon \SV\longrightarrow\SV'$ be a braided monoidal functor. Then the following statements are equivalent:
\begin{itemize}
    \item[(i)] $F$ is an equivalence of braided monoidal categories.
    \item[(ii)] $F$ is part of an adjoint equivalence of braided monoidal categories.
    \item[(iii)] $F$ is an equivalence of ordinary categories.
\end{itemize}
\end{Proposition}

The proof works similar to that of Proposition \ref{propAlternativeEquivalenceMonoidalCategories} and is omitted at this point.

\subsection{Module categories}\label{secBasicsModuleCategories}
As we want to define module monoidal categories later on, we quickly recall some definitions and statements on module categories. For this purpose, let $(\SV,\otimes,\1_\SV,a,l,r)$ be a monoidal category for the rest of this section.

\begin{Definition}[{\cite[Def.\,7.1.2]{EGNO}}]\label{defModuleCategory}
A \emph{(left) $\SV$\==module category} (or \emph{(left) module category over $\SV$}) is a quadruple $(\SM, \rhd, m, l^\rhd)$ where $\SM$ is a category, $\rhd\colon \SV\times\SM\longrightarrow\SM$ is the \emph{action functor} or \emph{module product}, $m\colon (-\otimes -)\rhd - \Longrightarrow -\rhd\,(- \rhd -)$ is the \emph{module associativity (constraint)} natural isomorphism and $l^\rhd\colon \1_\SV\rhd\strut -\Longrightarrow \Id_\SM$ is the \emph{module unit constraint} natural isomorphism. These data have to satisfy the pentagon (or associativity) and triangle (or unitality) axioms, i.e.\ the diagrams
\begin{equation}\tag{$\text{pent.\,}m$}\label{pentM}
\begin{tikzcd}
& ((u\otimes v)\otimes w)\rhd m \arrow[rd, "{m_{u\otimes v,w,m}}"] \arrow[ld, "{a_{u,v,w}\,\rhd\,\id_m}"'] & \\
(u\otimes (v\otimes w))\rhd m \arrow[d, "{m_{u,v\otimes w,m}}"']&& (u\otimes v)\rhd (w\rhd m) \arrow[d, "{m_{u,v,w\rhd m}}"] \\
u\rhd((v\otimes w)\rhd m) \arrow[rr, "{\id_u\,\rhd\,m_{v,w,m}}"'] && u\rhd(v\rhd(w\rhd m))
\end{tikzcd}
\end{equation}
and
\begin{equation}\tag{$\text{tri.\,}m$}\label{triangleM}
\begin{tikzcd}
(v\otimes\1_\SV)\rhd m \arrow[rr, "{m_{v,\1_\SV,m}}"] \arrow[rd, "{r_v\,\rhd\,\id_m}"'] & & v\rhd(\1_\SV\rhd m) \arrow[ld, "{\id_v\,\rhd\,l^\rhd_m}"] \\
& v\rhd m &
\end{tikzcd}
\end{equation}
have to commute for all $u,v,w\in\SV$ and $m\in\SM$.
\end{Definition}
There is an analogous notion of a \emph{right module category}, but we do not need it in this work. Unless otherwise noted, by a \emph{module category} we always mean a left module category over a fixed but otherwise arbitrary monoidal category $\SV$.

\begin{Remark}\label{tri'}
Let $\SM$ be a module category. Then, the triangle
\begin{equation}\tag{tri.'\,$m$}\label{diagTri'}
\begin{tikzcd}
(\1_\SV\otimes v)\rhd m \arrow[rr, "m_{\1_\SV,v,m}"] \arrow[rd, "{l_v\,\rhd\,\id_m}"'] & & \1_\SV\rhd(v\rhd m) \arrow[ld, "l^{\rhd}_{v\rhd m}"] \\
& v\rhd m & 
\end{tikzcd}
\end{equation}
commutes for all $v\in\SV$ and $m\in\SM$. The statement is proved using the same technique as the proof of \cite[Prop.\,2.2.4]{EGNO}.
\end{Remark}

\begin{Definition}[{\cite[Def.\,7.2.1]{EGNO}}]\label{defModuleFunctor}
A \emph{module functor} from a module category $\SM$ to a module category $\SM'$ is a tuple $(F,s)\colon\SM\longrightarrow\SM'$ where $F\colon\SM\longrightarrow\SM'$ is a functor and $s\colon -\rhd'\,F(-)\Longrightarrow F(-\rhd -)$ is a natural isomorphism compatible with the module associator and module unit constraint, i.e.\ the diagrams
\begin{equation}\tag{$s\,\&\,m$}\label{sPentagon}
\begin{tikzcd}
 & (v\otimes w)\rhd' F(m) \arrow[ld, "{m_{v,w,F(m)}'}"'] \arrow[rd, "{s_{v\otimes w,m}}"] & \\
v\rhd'(w\rhd'F(m)) \arrow[d, "{\id_v\,\rhd'\,s_{w,m}}"'] && F((v\otimes w)\rhd m) \arrow[d, "{F(m_{v,w,m})}"] \\
v\rhd'F(w\rhd m) \arrow[rr, "{s_{v,w\rhd m}}"']&& F(v\rhd(w\rhd m))
\end{tikzcd}
\end{equation}
and
\begin{equation}\tag{$s\,\&\,l^\rhd$}\label{sTriangle}
\begin{tikzcd}
\1_\SV\rhd' F(m) \arrow[rr, "{s_{\1_\SV,m}}"] \arrow[rd, "l^{\rhd'}_{F(m)}"'] && F(\1_\SV\rhd m) \arrow[ld, "F(l^\rhd_m)"] \\
& F(m) &
\end{tikzcd}
\end{equation}
commute for all $v,w\in\SV$ and $m\in\SM$.
\end{Definition}

\begin{Remark}\label{remCompModuleFunctors}
Given three module categories $\SM$, $\SM'$, $\SM''$ and two module functors $F\colon\SM\longrightarrow\SM'$ and $G\colon\SM'\longrightarrow\SM''$, their composition $(H,s^H)$ with $H\coloneqq GF$ and
\begin{equation}\tag{compos.\,$s$}\label{composS}
s^{H}_{v,m}\coloneqq\left[ v\rhd''GF(m)\xrightarrow[]{s^G_{v,F(m)}}G(v\rhd' F(m))\xrightarrow[]{G(s^F_{v,m})}GF(v\rhd m)\right]
\end{equation}
for all $v\in\SV$ and $m\in\SM$, is again a module functor (cf.\ \cite[Exc.\,7.2.3]{EGNO}).
\end{Remark}

\begin{Definition}[{\cite[Def.\,7.2.2]{EGNO}}]\label{defModuleNaturalTransformation}
Let $\SM$, $\SM'$ be two module categories and let $F,G\colon\SM\longrightarrow\SM'$ be two module functors between them. A \emph{module natural transformation} $\eta\colon F\Longrightarrow G$ is a natural transformation which is compatible with the module structure, i.e.\ the diagram
\begin{equation}\tag{nat.\,tr.\,\&\,$s$}\label{monNatTrafoS}
\begin{tikzcd}
v\rhd'F(m) \arrow[r, "{s^F_{v,m}}"] \arrow[d, "{\id_v\,\rhd'\,\eta_m}"'] & F(v\rhd m) \arrow[d, "\eta_{v\rhd m}"] \\
v\rhd'G(m) \arrow[r, "{s^G_{v,m}}"'] & G(v\rhd m)
\end{tikzcd}
\end{equation}
commutes for all $v\in\SV$ and $m\in\SM$.
\end{Definition}

\begin{Definition}\label{defEquivalenceModuleCategories}
Let $\SM$, $\SM'$ be two module categories and let $F\colon \SM\longrightarrow\SM'$ be a module functor. We call $F$ an \emph{equivalence of module categories} if there exist a module functor $G\colon \SM'\longrightarrow\SM$ and module natural isomorphisms $\eta\colon \Id_\SM\Longrightarrow GF$ and $\varepsilon\colon FG\Longrightarrow\Id_{\SM'}$.
\end{Definition}

As before for monoidal categories (cf.\ Prop.\,\ref{propAlternativeEquivalenceMonoidalCategories}) and braided monoidal categories (cf.\ Prop.\,\ref{propAlternativeEquivalenceBraidedMonoidalCategories}), we get the following statement on equivalences of module categories.

\begin{Proposition}\label{propAlternativeEquivalenceModuleCategories}
Let $\SM$, $\SM'$ be two module categories and let $F\colon \SM\longrightarrow\SM'$ be a module functor. Then the following statements are equivalent:
\begin{itemize}
    \item[(i)] $F$ is an equivalence of module categories.
    \item[(ii)] $F$ is part of an adjoint equivalence of module categories.
    \item[(iii)] $F$ is an equivalence of ordinary categories.
\end{itemize}
\end{Proposition}

The proof works similar to that of Proposition \ref{propAlternativeEquivalenceMonoidalCategories} and is omitted at this point.

\section{Module monoidal categories}\label{secDefinitions}
In this section, we define the notion of a module monoidal category $\SM$ over a braided monoidal category $\SV$, its functors and natural transformations. The intuition is to categorify the notion of an associative algebra $A$ over a commutative ring $R$. In this case, either $A$, $R$ or both might or might not be unital. The application we are interested in (constructing module monoidal categories from orbifold data \cite{CRS20, Mulevicius22} and vice versa) is comparable to requiring the base ring $R$ to be unital while $A$ does not have to be unital \cite{HR}. Therefore, we develop the notion of non\--unital module monoidal categories over braided unital monoidal categories. Further, there is a choice whether the module action is from the left or the right side. In this work, we concentrate on left module actions, but one also could have considered right module monoidal categories.

In Section \ref{secDefinitionsNonUnital}, we define non\--unital module monoidal categories and in Section \ref{secDefinitionsUnital} we investigate unital module monoidal categories. Remarkably, the only difference is that the monoidal part has to be unital and there do not arise further coherence conditions with the action from the module part.

As in Section \ref{secPreliminaries}, ``monoidal'' means ``unital monoidal'' and ``non\--unital monoidal'' will always be explicitly written out. For the whole section, let $\SV$ be a braided monoidal category.

\subsection*{Conventions on commutative diagrams}
Recall that in definitions, commutative diagrams are labelled by abbreviations which are collected in Appendix \ref{secAppendixGlossary}.

While proving diagrams to be commutative we often drop the objects from natural transformations since they are implied by the source and target of the corresponding arrow, e.g.\ an arrow $(u\otimes v)\otimes w\xrightarrow[]{a_{u,v,w}} u\otimes (v\otimes w)$ is denoted by $(u\otimes v)\otimes w\overset{a}{\longrightarrow} u\otimes (v\otimes w)$. Further, ``naturality'' and ``functoriality'' are consequently abbreviated as ``nat.'' and ``funct.''.

\subsection{Non--unital module monoidal categories}\label{secDefinitionsNonUnital}
\begin{Definition}\label{defNonUnitalModuleMonoidalCategory}
Let $\SV$ be a braided monoidal category. A \emph{(left) non\--unital module monoidal category (over $\SV$)} is a tuple $(\SM,\otimes, a, \rhd, m, l^\rhd, \alp1, \alp2)$ where
\begin{itemize}
    \item $(\SM,\otimes, a)$ is a non\--unital monoidal category,
    \item $(\SM,\rhd,m,l^\rhd)$ is a $\SV$\==module category,
    \item $\alp1\colon (-\rhd-)\otimes -\Longrightarrow -\rhd(-\otimes-)$ is a natural isomorphism with components
    \[\alp1_{v,m,n}\colon (v\rhd m)\otimes n\longrightarrow v\rhd (m\otimes n)\]
    for all $v\in\SV$ and $m,n\in\SM$ and
    \item $\alp2\colon -\otimes\ (-\rhd-)\Longrightarrow -\rhd(-\otimes-)$ is a natural isomorphism with components
    \[\alp2_{v,m,n}\colon m\otimes (v\rhd n)\longrightarrow v\rhd (m\otimes n)\]
    for all $v\in\SV$ and $m,n\in\SM$.
\end{itemize}
Besides the usual axioms for non\--unital monoidal categories and $\SV$\==module categories, these data have to make, for all $v,w\in\SV$ and $l,m,n\in\SM$, the diagrams

\begin{equation}\tag{$\alp1\,\&\,a$}\label{a1a}
\begin{tikzcd}
& ((v\rhd l)\otimes m)\otimes n \arrow[rd, "{a_{v\rhd l, m, n}}"] \arrow[ld, "{\alp1_{v,l,m}\,\otimes\,\id_n}"'] & \\
(v\rhd (l\otimes m))\otimes n \arrow[d, "{\alp1_{v,l\otimes m,n}}"']&& (v\rhd l)\otimes (m\otimes n) \arrow[d, "{\alp1_{v,l,m\otimes n}}"] \\
v\rhd ((l\otimes m)\otimes n) \arrow[rr, "{\id_v\,\rhd\,a_{l,m,n}}"'] && v\rhd (l\otimes (m\otimes n)),
\end{tikzcd}
\end{equation}

\begin{equation}\tag{$\alp1\,\&\,m$}\label{a1m}
\begin{tikzcd}
 & ((v\otimes w)\rhd m)\otimes n \arrow[ld, "{m_{v,w,m}\,\otimes\,\id_n}"'] \arrow[rd, "{\alp1_{v\otimes w,m,n}}"] & \\
(v\rhd (w\rhd m))\otimes n \arrow[d, "{\alp1_{v,w\rhd m, n}}"']& & (v\otimes w)\rhd (m\otimes n) \arrow[d, "{m_{v,w,m\otimes n}}"] \\
v\rhd((w\rhd m)\otimes n) \arrow[rr, "{\id_v\,\rhd\,\alp1_{w,m,n}}"'] & & v\rhd(w\rhd(m\otimes n)) 
\end{tikzcd}
\end{equation}

\begin{equation}\tag{$\alp2\,\&\,a$}\label{a2a}
\begin{tikzcd}
& (l\otimes m)\otimes (v\rhd n) \arrow[rd, "{\alp2_{v,l\otimes m, n}}"] \arrow[ld, "{a_{l,m,v\rhd n}}"'] &\\
l\otimes (m\otimes (v\rhd n)) \arrow[d, "{\id_l\,\otimes\,\alp2_{v,m,n}}"'] && v\rhd((l\otimes m)\otimes n) \arrow[d, "{\id_v\,\rhd\,a_{l,m,n}}"] \\
l\otimes(v\rhd(m\otimes n)) \arrow[rr, "{\alp2_{v,l,m\otimes n}}"']&& v\rhd (l\otimes (m\otimes n)),
\end{tikzcd}
\end{equation}

\begin{equation}\tag{$\alp2\,\&\,m$}\label{a2m}
\begin{tikzcd}
& m\otimes ((v\otimes w)\rhd n) \arrow[ld, "{\id_m\,\otimes\,m_{v,w,n}}"'] \arrow[rd, "{\alp2_{v\otimes w,m,n}}"] & \\
m\otimes (v\rhd (w\rhd n)) \arrow[d, "{\alp2_{v,m,w\rhd n}}"']& & (v\otimes w)\rhd (m\otimes n) \arrow[d, "{m_{v,w,m\otimes n}}"] \\
v\rhd(m\otimes(w\rhd n)) \arrow[rr, "{\id_v\,\rhd\,\alp2_{w,m,n}}"'] & & v\rhd(w\rhd(m\otimes n)) 
\end{tikzcd}
\end{equation}

\begin{equation}\tag{$\alp1\,\&\,\alp2$}\label{a1a2}
\begin{tikzcd}[column sep = huge]
(l\otimes (v\rhd m))\otimes n \arrow[r, "{\alp2_{v,l,m}\,\otimes\,\id_n}"] \arrow[d, "{a_{l,v\rhd m,n}}"'] & (v\rhd (l\otimes m))\otimes n \arrow[r, "{\alp1_{v,l\otimes m,n}}"] & v\rhd((l\otimes m)\otimes n) \arrow[d, "{\id_v\,\rhd\,a_{l,m,n}}"] \\
l\otimes((v\rhd m)\otimes n) \arrow[r, "{\id_l\,\otimes\,\alp1_{v,m,n}}"'] & l\otimes (v\rhd(m\otimes n)) \arrow[r, "{\alp2_{v,l,m\otimes n}}"'] & v\rhd(l\otimes(m\otimes n),
\end{tikzcd}
\end{equation}

\begin{equation}\tag{$\beta\,\&\,\rhd$}\label{brac}
\begin{tikzcd}
 & (w\rhd m)\otimes (v\rhd n) \arrow[ld, "{\alp2_{v,w\rhd m,n}}"'] \arrow[rd, "{\alp1_{w,m,v\rhd n}}"] &\\
v\rhd((w\rhd m)\otimes n) \arrow[d, "{\id_v\,\rhd\,\alp1_{w,m,n}}"'] & & w\rhd(m\otimes(v\rhd n)) \arrow[d, "{\id_w\,\rhd\,\alp2_{v,m,n}}"] \\
v\rhd(w\rhd(m\otimes n)) \arrow[d, "{m^{-1}_{v,w,m\otimes n}}"']& & w\rhd(v\rhd(m\otimes n)) \arrow[d, "{m^{-1}_{w,v,m\otimes n}}"] \\
(v\otimes w)\rhd (m\otimes n) \arrow[rr, "{\beta^\SV_{v,w}\,\rhd\,\id_{m\otimes n}}"'] & & (w\otimes v)\rhd(m\otimes n)
\end{tikzcd}
\end{equation}
commute.
\end{Definition}
Unless otherwise noted, any non\--unital module monoidal category is always considered to be over a fixed but otherwise arbitrary braided monoidal category $\SV$. To clarify notation, the intuition behind $\alp1$ is ``pulling the action of $v$ out of the first component of $m\otimes n$''. Similarly, $\alp2$ can be read as ``pulling the action of $v$ out of the second component of $m\otimes n$''.

One automatically gets that the natural isomorphisms $\alp1$ and $\alp2$ behave well with the module unit constraint $l^\rhd$. This is shown in the following
\begin{Proposition}\label{propFurtherAxiomsNonUnitalModuleMonoidalCategory}
Let $\SM$ be a non\--unital module monoidal category. Then, for all $m,n\in\SM$, the diagrams
\begin{equation}\tag{$\alp1\,\&\,l^\rhd$}\label{a1lm}
\begin{tikzcd}
(\1_\SV\rhd m)\otimes n \arrow[rr, "{\alp1_{\1_\SV,m,n}}"] \arrow[rd, "{l^\rhd_m\,\otimes\,\id_n}"'] && \1_\SV\rhd(m\otimes n) \arrow[ld, "l^\rhd_{m\otimes n}"] \\
& m\otimes n &
\end{tikzcd}
\end{equation}
and
\begin{equation}\tag{$\alp2\,\&\,l^\rhd$}\label{a2lm}
\begin{tikzcd}
m \otimes(\1_\SV\rhd n) \arrow[rr, "{\alp2_{\1_\SV,m,n}}"] \arrow[rd, "{\id_m\,\otimes\,l^\rhd_n}"'] && \1_\SV\rhd(m\otimes n) \arrow[ld, "l^\rhd_{m\otimes n}"] \\
& m\otimes n &
\end{tikzcd}
\end{equation}
commute.
\end{Proposition}

\begin{proof}
The statements are proved using the same technique as the proof of \cite[Prop.\,2.2.4]{EGNO}. With the diagrams from Definition \ref{defNonUnitalModuleMonoidalCategory}, one has that
\begin{itemize}
    \item \eqref{a1m} yields \eqref{a1lm} and
    \item \eqref{a2m} yields \eqref{a2lm}.\qedhere
\end{itemize}
\end{proof}

\begin{Definition}\label{defNonUnitalModuleMonoidalFunctor}
A \emph{non\--unital module monoidal functor} from a non\--unital module monoidal category $\SM$ to a non\--unital module monoidal category $\SM'$ is a tuple $(F,\varphi_2,s)\colon\SM\longrightarrow\SM'$ such that
\begin{itemize}
    \item $(F,\varphi_2)\colon (\SM,\otimes,a)\longrightarrow(\SM',\otimes',a')$ is a non\--unital monoidal functor and
    \item $(F,s)\colon (\SM,\rhd,m,l^\rhd)\longrightarrow(\SM',\rhd',m',l^{\rhd'})$ is a module functor,
\end{itemize}
compatible with the morphisms $\alp1$, $\alpd1$, $\alp2$ and $\alpd2$, i.e.\ the diagrams
\begin{equation}\tag{comp.\,$\alp1$}\label{diagModuleMonoidalFunctorCompAlpha1}
\begin{tikzcd}[column sep = huge]
(v\rhd'F(m))\otimes'F(n) \arrow[r, "{\alpd1_{v,F(m),F(n)}}"] \arrow[d, "{s_{v,m}\,\otimes'\,\id_{F(n)}}"'] & v\rhd'(F(m)\otimes'F(n)) \arrow[r, "{\id_v\,\rhd'\,\varphi_{2,m,n}}"] & v\rhd'F(m\otimes n) \arrow[d, "{s_{v,m\otimes n}}"] \\
F(v\rhd m)\otimes'F(n) \arrow[r, "{\varphi_{2,v\rhd m, n}}"']& F((v\rhd m)\otimes n) \arrow[r, "{F(\alp1_{v,m,n})}"']& F(v\rhd(m\otimes n)) 
\end{tikzcd}
\end{equation}
and
\begin{equation}\tag{comp.\,$\alp2$}\label{diagModuleMonoidalFunctorCompAlpha2}
\begin{tikzcd}[column sep = huge]
F(m)\otimes' (v\rhd'F(n)) \arrow[r, "{\alpd2_{v,F(m),F(n)}}"] \arrow[d, "{\id_{F(m)}\,\otimes'\,s_{v,n}}"'] & v\rhd'(F(m)\otimes'F(n)) \arrow[r, "{\id_v\,\rhd'\,\varphi_{2,m,n}}"] & v\rhd'F(m\otimes n) \arrow[d, "{s_{v,m\otimes n}}"] \\
F(m)\otimes'F(v\rhd n) \arrow[r, "{\varphi_{2,m,v\rhd n}}"']& F(m\otimes(v\rhd n)) \arrow[r, "{F(\alp2_{v,m,n})}"'] & F(v\rhd(m\otimes n)) 
\end{tikzcd}
\end{equation}
commute for all $v\in\SV$ and $m,n\in\SM$.
\end{Definition}

\begin{Proposition}\label{propCompNonUnitalModuleMonoidalFunctors}
Let $\SM$, $\SM'$, $\SM''$ be non\--unital module monoidal categories and $F\colon \SM\longrightarrow\SM'$, $G\colon \SM'\longrightarrow \SM''$ be non\--unital module monoidal functors. Their composition $(H,\varphi_2^H,s^H)\colon\SM\longrightarrow\SM''$ with $H\coloneqq GF$, $\varphi_2^H$ as in \eqref{composPhi2} and $s^H$ as in \eqref{composS} is also a non\--unital module monoidal functor.
\end{Proposition}

\begin{proof}
We have to check that the tuple $(H,\varphi_2^H,s^H)$ satisfies the conditions from Definition \ref{defNonUnitalModuleMonoidalFunctor}. By Remark \ref{remCompMonoidalFunctors} $(H,\varphi_2^H)$ is a non\--unital monoidal functor and by Remark \ref{remCompModuleFunctors} $(H,s^H)$ is a module functor. Hence, we are left to show that the diagrams \eqref{diagModuleMonoidalFunctorCompAlpha1} and \eqref{diagModuleMonoidalFunctorCompAlpha2} commute. Plugging in \eqref{composS} and \eqref{composPhi2} for $s^H$ and $\varphi_2^H$, \eqref{diagModuleMonoidalFunctorCompAlpha1} becomes 
\begin{equation*}
\begin{tikzcd}
(v\rhd''GF(m))\otimes''GF(n) \arrow[r, "\alpdd1"] \arrow[rddd, shorten >= 2cm, "\com{\eqref{diagModuleMonoidalFunctorCompAlpha1}\text{ for } G}" description, color = white]\arrow[ddd, "{s^G\,\otimes''\,\id}"'] & v\rhd''(GF(m)\otimes''GF(n)) \arrow[d, "{\id\,\rhd''\,\varphi_2^G}"] &\\
 & v\rhd''G(F(m)\otimes'F(n)) \arrow[rd, shorten >= 1cm, "\com{\nat s^G}" description, color = white]\arrow[d, "s^G"] \arrow[r, "{\id\,\rhd''\,G(\varphi_2^F)}"]& v\rhd''GF(m\otimes n) \arrow[d, "s^G"] \\
 & G(v\rhd'(F(m)\otimes'F(n))) \arrow[rddd, shorten <= 2cm, "\com{\eqref{diagModuleMonoidalFunctorCompAlpha1}\text{ for } F}" description, color = white]\arrow[r, "{G(\id\,\rhd'\,\varphi_2^F)}"'] & G(v\rhd'F(m\otimes n)) \arrow[ddd, "G(s^F)"]\\
G(v\rhd'F(m))\otimes''GF(n) \arrow[rd, "\com{\nat \varphi_2^G}" description, color = white]\arrow[d, "{G(s^F)\,\otimes''\,\id}"'] \arrow[r, "\varphi_2^G"] & G((v\rhd'F(m))\otimes'F(n)) \arrow[d, "{G(s^F\,\otimes'\,\id)}"'] \arrow[u, "G(\alpd1)"]&\\
GF(v\rhd m)\otimes''GF(n) \arrow[r, "\varphi_2^G"']& G(F(v\rhd m)\otimes'F(n)) \arrow[d, "G(\varphi_2^F)"'] &\\
 & GF((v\rhd m)\otimes n) \arrow[r, "GF(\alp1)"']& GF(v\rhd(m\otimes n)) 
\end{tikzcd}
\end{equation*}
which commutes, for all $v\in\SV$ and $m,n\in\SM$, due to the reasons indicated in boxes. Additionally, \eqref{diagModuleMonoidalFunctorCompAlpha2} reduces to a very similar diagram.
\end{proof}

\begin{Definition}\label{defNonUnitalModuleMonoidalNaturalTransformation}
Let $\SM$, $\SM'$ be two non\--unital module monoidal categories and let $F,G\colon\SM\longrightarrow\SM'$ be two non\--unital module monoidal functors between them. A \emph{non\--unital module monoidal natural transformation} $\eta\colon F\Longrightarrow G$ is a natural transformation such that
\begin{itemize}
    \item $\eta$ is a non\--unital monoidal natural transformation and
    \item $\eta$ is a module natural transformation.
\end{itemize}
\end{Definition}

\begin{Definition}\label{defEquivalenceNonUnitalModuleMonoidalCategories}
Let $\SM$, $\SM'$ be two non\--unital module monoidal categories and let $F\colon \SM\longrightarrow\SM'$ be a non\--unital module monoidal functor. We call $F$ an \emph{equivalence of non\--unital module monoidal categories} if there exist a non\--unital module monoidal functor $G\colon \SM'\longrightarrow\SM$ and non\--unital module monoidal natural isomorphisms $\eta\colon \Id_\SM\Longrightarrow GF$ and $\varepsilon\colon F G\Longrightarrow\Id_{\SM'}$.
\end{Definition}

Like for monoidal, braided monoidal and module functors (cf.\ Prop.\,\ref{propAlternativeEquivalenceMonoidalCategories}, Prop.\,\ref{propAlternativeEquivalenceBraidedMonoidalCategories} and Prop.\,\ref{propAlternativeEquivalenceModuleCategories}) we get a statement, which extremely eases the verification that a given non\--unital module monoidal functor is an equivalence of non\--unital module monoidal categories.

\begin{Proposition}\label{propAlternativeEquivalenceNonUnitalModuleMonoidalCategories}
Let $\SM$, $\SM'$ be two non\--unital module monoidal categories and let $F\colon \SM\longrightarrow\SM'$ be a non\--unital module monoidal functor. Then the following statements are equivalent:
\begin{itemize}
    \item[(i)] $F$ is an equivalence of non\--unital module monoidal categories.
    \item[(ii)] $F$ is part of an adjoint equivalence of non\--unital module monoidal categories.
    \item[(iii)] $F$ is an equivalence of ordinary categories.
\end{itemize}
\end{Proposition}
\begin{proof}
Obviously, (ii) implies (i) and (i) implies (iii).

To see that (iii) implies (ii) we suppose $F\colon\SM\longrightarrow \SM'$ is a non\--unital module monoidal functor which is further an equivalence of categories.

By \cite[Thm.\,IV.4.1]{MacLane}, we can assume that $F$ is part of an adjoint equivalence $(F,G,\eta,\varepsilon)$ and that $F$ is fully faithful. Thus, for all $v\in\SV$ and $m,n\in\SM'$, we can define
\begin{align*}
\varphi_{2,m,n}^G&\coloneqq F^{-1}\left(\varepsilon_{m\otimes' n}^{-1}\circ (\varepsilon_m\otimes'\varepsilon_n)\circ (\varphi_{2,G(m),G(n)}^F)^{-1}\right),\\
s_{v,m}^G&\coloneqq F^{-1}\left(\varepsilon^{-1}_{v\rhd'm}\circ (\id_v\rhd' \varepsilon_m)\circ (s^F_{v,G(m)})^{-1}\right).
\end{align*}\goodbreak

With these definitions the proofs of Propositions \ref{propAlternativeEquivalenceMonoidalCategories} and \ref{propAlternativeEquivalenceModuleCategories} yield that
\begin{itemize}
    \item $(G,\varphi_2^G)$ is a non\--unital monoidal functor,
    \item $(G,s^G)$ is a module functor,
    \item $\eta$ and $\varepsilon$ are monoidal and module natural isomorphisms, i.e.\ \--- given that $G$ is indeed a non\--unital module monoidal functor \--- they are module monoidal natural isomorphisms and
    \item $(F,G,\eta,\varepsilon)$ is an adjunction.
\end{itemize}

We are left to show that $G$ is indeed a non\--unital module monoidal functor, i.e.\ diagrams \eqref{diagModuleMonoidalFunctorCompAlpha1} and \eqref{diagModuleMonoidalFunctorCompAlpha2} commute, where \eqref{diagModuleMonoidalFunctorCompAlpha1} is satisfied by applying $F^{-1}$ to the border of the diagram
\begin{equation*}
{\begin{tikzcd}
F((v\rhd G(m))\otimes G(n))\arrow[ddddddd, "\com{\nat \varphi_2^F}" description, color = blue, pos = 0.3, shift right = 10, shorten >=1cm, shorten <=8cm]\arrow[rddd, "\com{\eqref{diagModuleMonoidalFunctorCompAlpha1}\text{ for } F}" description, color = white, shorten >= 3cm] \arrow[r, "F(\alp1)"] \arrow[ddddddd, "{F(s^G\,\otimes\,\id)}", bend right=15, shift right=18]& F(v\rhd(G(m)\otimes G(n))) \arrow[rd, "\com{\nat s^F}" description, color = white]\arrow[r, "{F(\id\,\rhd\,\varphi_2^G)}"] & F(v\rhd G(m\otimes'n)) \arrow[ddddddd, "F(s^G)"', bend left=15, shift left=13]\arrow[ddddddd, "\com{\dfn s^G}" description, color = blue, shift left = 10, shorten >= 8cm, shorten <= 1cm, pos = 0.3] \\
& v\rhd' F(G(m)\otimes G(n)) \arrow[u, "s^F"] \arrow[r, "{\id\,\rhd'\,F(\varphi_2^G)}"] & v\rhd'FG(m\otimes' n) \arrow[u, "s^F"] \arrow[ddd, "{\id\,\rhd'\,\varepsilon}"'] \\
& v\rhd'(FG(m)\otimes'FG(n)) \arrow[u, "{\id\,\rhd'\,\varphi_2^F}"] \arrow[rdd,start anchor = east, "{\id\,\rhd'\,(\varepsilon\,\otimes'\,\varepsilon)}", pos=0.1]\arrow[ru,"\com{\dfn \varphi_2^G}" description, color = white] \arrow[dd, "{\id\,\rhd'\,(\varepsilon\,\otimes'\,\id)}"', bend left=40, shift left = 15, pos = 0.1] &\\
F(v\rhd G(m))\otimes'FG(n) \arrow[rddd, "\com{\dfn s^G}" description, color = white, shorten >=3cm]\arrow[ddd, "{F(s^G)\,\otimes'\,\id}"] \arrow[uuu, "\varphi_2^F"']& (v\rhd' FG(m))\otimes'FG(n) \arrow[u, "\alpd1"] \arrow[dd, "{(\id\,\rhd'\,\varepsilon)\,\otimes'\,\id}"', bend right, shift right = 13]\arrow[d, "\com{\nat \alpd1}" description, color = white] \arrow[l, "{s^F\,\otimes'\,\id}"] &\\
& v\rhd'(m\otimes' FG(n)) \arrow[r, pos = 0.7, "\com{\funct \otimes'}" description, color = white, shift left = 8, shorten <= 1cm]\arrow[r, "{\id\,\rhd'\,(\id\,\otimes'\,\varepsilon)}"] & v\rhd'(m\otimes' n)\\
& (v\rhd' m)\otimes'FG(n) \arrow[ru, "\com{\nat \alpd1}" description, color = white]\arrow[u, "\alpd1"'] \arrow[d, "{\id\,\otimes'\,\varepsilon}"']&\\
FG(v\rhd' m)\otimes'FG(n) \arrow[r, pos = 0.8, "\com{\funct \otimes'}" description, shift left = 5, shorten <= 1cm, color = white]\arrow[rd, "\com{\dfn \varphi_2^G}" description, color = white]\arrow[d, "\varphi_2^F"] \arrow[r, "{\varepsilon\,\otimes'\,\varepsilon}"'] \arrow[ru, "{\varepsilon\,\otimes'\,\id}"] & (v\rhd' m)\otimes' n \arrow[ruu, "\alpd1"']\arrow[rd, "\com{\nat \varepsilon}" description, color = white]&\\
F(G(v\rhd' m)\otimes G(n)) \arrow[r, "F(\varphi_2^G)"'] & FG((v\rhd' m)\otimes'n) \arrow[u, "\varepsilon"] \arrow[r, "FG(\alpd1)"'] & FG(v\rhd'(m\otimes' n)) \arrow[uuu, "\varepsilon"]
\end{tikzcd}}
\end{equation*}
which commutes, for all $v\in\SV$ and $m,n\in\SM'$, due to the reasons indicated in boxes.

Replacing $\alp1$, $\alpd1$, $\alpdd1$ and \eqref{diagModuleMonoidalFunctorCompAlpha1} by $\alp2$, $\alpd2$, $\alpdd2$ and \eqref{diagModuleMonoidalFunctorCompAlpha2} in the diagram above, one obtains that $G$ satisfies \eqref{diagModuleMonoidalFunctorCompAlpha2}.

In total, we have that $(F,G,\eta,\varepsilon)$ witnesses that $F$ is an adjoint equivalence of non\--unital module monoidal categories.
\end{proof}

\begin{Definition}\label{def2CategoryModMonCat}
Non\--unital module monoidal categories, non\--unital module monoidal functors and non\--unital module monoidal natural transformations form a strict 2\==category, which we denote by $\ModMonCat$.
\end{Definition}
    
\subsection{Unital module monoidal categories}\label{secDefinitionsUnital}
\begin{Definition}\label{defUnitalModuleMonoidalCategory}
Let $\SV$ be a braided monoidal category. A \emph{unital module monoidal category (over $\SV$)} is a tuple $(\SM,\otimes, \1_\SM, a, l, r, \rhd, m, l^\rhd, \alp1, \alp2)$ where
\begin{itemize}
    \item $(\SM,\otimes, a, \rhd, m, l^\rhd, \alp1, \alp2)$ is a non\--unital module monoidal category and
    \item $(\SM,\otimes, \1_\SM, a, l, r)$ is a unital monoidal category.
\end{itemize}
\end{Definition}

Unless otherwise noted, any unital module monoidal category is always considered to be over a fixed but otherwise arbitrary braided monoidal category $\SV$. One automatically gets that the natural isomorphisms $\alp1$ and $\alp2$ behave well with some of the unit constraints. This is shown in the following
\begin{Proposition}\label{propFurtherAxiomsUnitalModuleMonoidalCategory}
Let $\SM$ be a unital module monoidal category. Then, for all $v\in\SV$ and $m\in\SM$, the diagrams
\begin{equation}\tag{$\alp1\,\&\,r$}\label{a1r}
\begin{tikzcd}
(v\rhd m)\otimes \1_\SM \arrow[rr, "{\alp1_{v,m,\1_\SM}}"] \arrow[rd, "r_{v\rhd m}"'] && v\rhd(m\otimes \1_\SM) \arrow[ld, "{\id_v\,\rhd\, r_m}"] \\
 & v\rhd m &
\end{tikzcd}
\end{equation}
and
\begin{equation}\tag{$\alp2\,\&\,l$}\label{a2l}
\begin{tikzcd}
\1_\SM\otimes(v\rhd m) \arrow[rr, "{\alp2_{v,\1_\SM,m}}"] \arrow[rd, "l_{v\rhd m}"'] && v\rhd(\1_\SM\otimes m) \arrow[ld, "{\id_v\,\rhd\, l_m}"] \\
& v\rhd m &
\end{tikzcd}
\end{equation}
commute.
\end{Proposition}

\begin{proof}
The statements are proved using the same technique as the proof of \cite[Prop.\,2.2.4]{EGNO}. Using the diagrams from Definition \ref{defNonUnitalModuleMonoidalCategory}, one has that
\begin{itemize}
    \item \eqref{a1a} yields \eqref{a1r} and
    \item \eqref{a2a} yields \eqref{a2l}.\qedhere
\end{itemize}
\end{proof}

\begin{Definition}\label{defUnitalModuleMonoidalFunctor}
A \emph{unital module monoidal functor} from a unital module monoidal category $\SM$ to a unital module monoidal category $\SM'$ is a tuple $(F,\varphi_2,\varphi_0,s)\colon\SM\longrightarrow\SM'$ such that
\begin{itemize}
    \item $(F,\varphi_2,s)\colon \SM\longrightarrow\SM'$ is a non\--unital module monoidal functor and 
    \item $(F,\varphi_2,\varphi_0)\colon (\SM,\otimes,\1_\SM,a,l,r)\longrightarrow(\SM',\otimes',\1_{\SM'},a',l',r')$ is a unital monoidal functor.
\end{itemize}
\end{Definition}

Combining Remark \ref{remCompMonoidalFunctors} and Proposition \ref{propCompNonUnitalModuleMonoidalFunctors} one gets that unital module monoidal functors are composable.

\begin{Definition}\label{defUnitalModuleMonoidalNaturalTransformation}
A \emph{unital module monoidal natural transformation} is a non\--unital module monoidal natural transformation, which is further compatible with $\varphi_0$, i.e.\ which is a unital monoidal natural transformation and a module natural transformation at once.
\end{Definition}

\begin{Definition}\label{defEquivalenceUnitalModuleMonoidalCategories}
Let $\SM$, $\SM'$ be two unital module monoidal categories and let $F\colon \SM\longrightarrow\SM'$ be a unital module monoidal functor. We call $F$ an \emph{equivalence of unital module monoidal categories} if there exist a unital module monoidal functor $G\colon \SM'\longrightarrow\SM$ and unital module monoidal natural isomorphisms $\eta\colon \Id_\SM\Longrightarrow GF$ and $\varepsilon\colon F G\Longrightarrow\Id_{\SM'}$.
\end{Definition}

Combining the proofs of Propositions \ref{propAlternativeEquivalenceMonoidalCategories} and \ref{propAlternativeEquivalenceNonUnitalModuleMonoidalCategories}, one gets the following statement, which extremely eases the verification, that a given unital module monoidal functor is an equivalence of unital module monoidal categories.

\begin{Proposition}\label{propAlternativeEquivalenceUnitalModuleMonoidalCategories}
Let $\SM$, $\SM'$ be two unital module monoidal categories and let $F\colon \SM\longrightarrow\SM'$ be a unital module monoidal functor. Then the following statements are equivalent:
\begin{itemize}
    \item[(i)] $F$ is an equivalence of unital module monoidal categories.
    \item[(ii)] $F$ is part of an adjoint equivalence of unital module monoidal categories.
    \item[(iii)] $F$ is an equivalence of ordinary categories.
\end{itemize}
\end{Proposition}

\begin{Definition}\label{def2CategoryUModMonCat}
Unital module monoidal categories, unital module monoidal functors and unital module monoidal natural transformations form a strict 2\==category, which we denote by $\UModMonCat$.
\end{Definition}

From now on, by module monoidal categories/functors/natural transformations we mean unital module monoidal categories/functors/natural transformations. If something is meant to be non\--unital, we will explicitly write so.

\section{Pairs $(\ST,F)$}\label{secPairs}

In \cite[Def.\,3.6]{HPT15}, the datum of a module tensor category is introduced as a tuple $(\ST,F)$ consisting of a unital monoidal category $\ST$ and a braided central unital monoidal functor $(F,F^Z)\colon\SV\longrightarrow\ST$, which is a unital monoidal functor $F\colon\SV\longrightarrow\ST$ together with a braided unital monoidal lift $F^Z\colon\SV\longrightarrow Z(\ST)$ from a braided unital monoidal category $\SV$ to the Drinfeld center of $\ST$. We refer to a datum $(\ST,F)$ as a \emph{pair} and later prove that this notion is equivalent to our notion of a unital module monoidal category (cf.\ Sec.\,\ref{secDefinitionsUnital}). 

To fix notation, we briefly recall the notion of the Drinfeld center of a monoidal category in Section \ref{secPairsDrinfeldCenterAsBraidedCategory} and introduce braided central monoidal functors and the notion of a pair as well as its morphisms, 2\==morphisms and equivalences in Section \ref{secPairsBraidedCentralFunctors}.

\subsection{The Drinfeld center as a braided category}\label{secPairsDrinfeldCenterAsBraidedCategory}
Note, that there are two conventions for the half\--braiding on an object of the Drinfeld center. In this section, we briefly introduce our convention and see that the Drinfeld center is a braided category.
\begin{Definition}[{\cite[Def.\,7.13.1]{EGNO}}]\label{defDrinfeldCenter}
Let $\ST$ be a monoidal category. The \emph{Drinfeld center} of $\ST$ is the category $Z(\ST)$ given by
\begin{itemize}
    \item objects $(t, \beta_{t,-})$, where $t\in\ST$ and $\beta_{t,-}\colon t\,\otimes - \Longrightarrow -\otimes\,t$ is a natural isomorphism, such that the diagram
    \begin{equation}\tag{half\--br.\,hex.}\label{defDrinfeldHalfBraiding}
    \begin{tikzcd}[column sep = huge]
    (t\otimes x)\otimes y \arrow[r, "{a_{t,x,y}}"] \arrow[d, "{\beta_{t,x}\,\otimes\,\id_y}"'] & t\otimes (x\otimes y) \arrow[r, "{\beta_{t,x\otimes y}}"] & (x\otimes y)\otimes t \arrow[d, "{a_{x,y,t}}"] \\
    (x\otimes t)\otimes y \arrow[r, "{a_{x,t,y}}"']& x\otimes (t\otimes y) \arrow[r, "{\id_x\,\otimes\,\beta_{t,y}}"'] & x\otimes (y\otimes t) 
    \end{tikzcd}
    \end{equation}
    commutes for all $x,y\in\ST$, and
    \item morphisms $f\colon (t, \beta_{t,-})\longrightarrow(t', \beta_{t',-})$, where $f$ is a morphism in $\ST$ from $t$ to $t'$ such that the diagram
    \begin{equation}\label{defDrinfeldMorphism}
    \begin{tikzcd}
    t\otimes x \arrow[r, "{\beta_{t,x}}"] \arrow[d, "{f\,\otimes\,\id_x}"'] & x\otimes t \arrow[d, "{\id_x\,\otimes\,f}"] \\
    t'\otimes x \arrow[r, "{\beta_{t',x}}"']& x\otimes t'
    \end{tikzcd}
    \end{equation}
    commutes for all $x\in\ST$.
\end{itemize}
\end{Definition}
Note that, for better distinction, the half\--braiding is consequently denoted as $\beta$ whereas the braiding of a braided monoidal category $\SV$ is always denoted as $\beta^\SV$.

For the Drinfeld center to be a braided category it necessarily has to be a monoidal category. This is ensured by the following statement.
\begin{Proposition}[{\cite[Def.\,7.13.1]{EGNO}}]\label{propDrinfeldCenterIsMonoidal}
Let $\ST$ be a monoidal category. The Drinfeld center $Z(\ST)$ becomes a monoidal category with
\begin{itemize}
    \item $\left(t, \beta_{t,-}\right)\otimes\left(t', \beta_{t',-}\right)\coloneqq \left(t\otimes t', \beta_{t\otimes t',-}\right)$, where
    \begin{equation*}
        \beta_{t\otimes t',x}\coloneqq a_{x,t,t'}\circ \left(\beta_{t,x}\otimes \id_{t'}\right)\circ a_{t,x,t'}^{-1}\circ \left(\id_{t}\otimes\beta_{t',x}\right)\circ a_{t,t',x}
    \end{equation*}
    for all $x\in\ST$, which can be reformulated to equivalently read as a commutative diagram
    \begin{equation}\tag{half\--br.\,on\,$\otimes$}\label{defDrinfeldTensor}
    \begin{tikzcd}[column sep = huge]
    t\otimes(t'\otimes x) \arrow[r, "{a_{t,t',x}^{-1}}"] \arrow[d, "{\id_t\,\otimes\,\beta_{t',x}}"'] & (t\otimes t')\otimes x \arrow[r, "{\beta_{t\otimes t',x}}"] & x\otimes (t\otimes t') \arrow[d, "{a_{x,t,t'}^{-1}}"] \\
    t\otimes(x\otimes t') \arrow[r, "{a_{t,x,t'}^{-1}}"'] & (t\otimes x)\otimes t' \arrow[r, "{\beta_{t,x}\,\otimes\,\id_{t'}}"'] & (x\otimes t)\otimes t' 
    \end{tikzcd}
    \end{equation}
    \item unit object $(\1_\ST, \beta_{\1_\ST,-})$, where 
    \begin{equation}\tag{half\--br.\,on\,$\1$}\label{defDrinfeldUnitMorphism}
        \beta_{\1_\ST,x}\coloneqq r_x^{-1}\circ l_x
    \end{equation}
    for all $x\in\ST$ and
    \item associator $a^{Z(\ST)}\coloneqq a$ and unit constraints $l^{Z(\ST)}\coloneqq l$ and $r^{Z(\ST)}\coloneqq r$ given by the corresponding morphisms in $\ST$.
\end{itemize}
\end{Proposition}

The proof of this statement involves rather large commutative diagrams which still only use basic axioms like \eqref{pentA} and \eqref{triangleA} as well as the diagrams from Definition \ref{defDrinfeldCenter} several times. Therefore, it is omitted at this point.

Finally, the next proposition shows that $Z(\ST)$ is indeed a braided monoidal category.
\begin{Proposition}[{\cite[Rem.\,7.13.3]{EGNO}}]\label{propDrinfeldCenterIsBraided}
Let $\ST$ be a monoidal category. The Drinfeld center $Z(\ST)$ becomes a braided monoidal category using the half\--braiding, i.e.
\begin{equation}\label{defDrinfeldBraiding}
    \beta^{Z(\ST)}_{(t, \beta_{t,-}),(t',\beta_{t',-})}\coloneqq \beta_{t,t'}.
\end{equation}
\end{Proposition}

\begin{proof}
We have to show that the braiding given by Equation \eqref{defDrinfeldBraiding} satisfies the two braiding axioms \eqref{braidingHexagon1} and \eqref{braidingHexagon2} from Definition \ref{defBraidedCategory}. Since the associator of $Z(\ST)$ is given by the associator of $\ST$, these are exactly diagrams \eqref{defDrinfeldHalfBraiding} from Definition \ref{defDrinfeldCenter} and \eqref{defDrinfeldTensor} from Definition \ref{propDrinfeldCenterIsMonoidal}.
\end{proof}

\subsection{Braided central monoidal functors and pairs $(\ST,F)$}\label{secPairsBraidedCentralFunctors}
We want to introduce the notions of a pair, its morphisms, 2\==morphisms and equivalences. Note, that in the reference the datum of a pair is called a module tensor category. Further, the authors of the reference assumed pivotal structures both on $\SV$ and $\ST$. Since we do not need these, the definitions below differ slightly from the ones in the reference.

For the entire section, let $\SV$ be a braided monoidal category.
\begin{Definition}[{\cite[Def.\,3.5]{HPT15}}]\label{defBraidedCentralFunctor}
Let $\ST$ be a monoidal category.
\begin{itemize}
    \item A \emph{central} functor $(F,F^Z)\colon\SV\longrightarrow \ST$ is a functor $F\colon\SV\longrightarrow\ST$ together with a lift $F^Z\colon\SV\longrightarrow Z(\ST)$, i.e.\ $F=R F^Z$, where $R$ denotes the forgetful functor. Often, the datum of a lift for a central functor will be suppressed and left implicit.
    \item A $\emph{central monoidal}$ functor is a monoidal functor which is central with a monoidal lift.
    \item A \emph{braided central monoidal} functor is a monoidal functor which is central with a braided monoidal lift.
\end{itemize}
\end{Definition}

\begin{Definition}[{\cite[Def.\,3.1]{HPT16}}]\label{defPair}
Let $\SV$ be a braided monoidal category. For the context of this paper, we define a \emph{pair (over $\SV$)} to be a tuple $(\ST,F)$ where $\ST$ is a monoidal category and $F\colon\SV\longrightarrow \ST$ is a braided central monoidal functor.
\end{Definition}
Unless otherwise noted, any pair is always considered to be over a fixed but otherwise arbitrary braided monoidal category $\SV$.

\begin{Definition}[{\cite[Def.\,3.2]{HPT16}}]\label{defMorphismOfPairs}
A \emph{morphism of pairs} from a pair $(\ST,F)$ to a pair $(\ST',F')$ is a tuple $(G,\varphi_2,\varphi_0,\gamma)$ where
\begin{itemize}
    \item $(G,\varphi_2,\varphi_0)\colon\ST\longrightarrow\ST'$ is a monoidal functor,
    \item $\gamma\colon F'\Longrightarrow G F$ is a monoidal natural isomorphism called the \emph{action coherence} with components
    \[\gamma_v\colon F'(v)\longrightarrow GF(v)\]
    for all $v\in\SV$.
\end{itemize}
Further, the action coherence has to be compatible with the half\--braidings on $Z(\ST)$ and $Z(\ST')$. This is assured by the commutativity of the diagram
\begin{equation}\tag{$\gamma$\,\&\,$\beta$}\label{actionCoherenceAndBraiding}
\begin{tikzcd}[column sep = huge]
F'(v)\otimes' G(t) \arrow[r, "{\gamma_v\,\otimes'\,\id_{G(t)}}"] \arrow[d, "{\beta_{F'(v),G(t)}^{Z(\ST')}}"'] & GF(v)\otimes' G(t) \arrow[r, "{\varphi_{2,F(v),t}^G}"] & G(F(v)\otimes t) \arrow[d, "{G(\beta_{F(v),t}^{Z(\ST)})}"] \\
G(t)\otimes' F'(v) \arrow[r, "{\id_{G(t)}\,\otimes'\,\gamma_v}"'] & G(t)\otimes'GF(v) \arrow[r, "{\varphi_{2,t,F(v)}^G}"'] & G(t\otimes F(v))
\end{tikzcd}
\end{equation}
for all $v\in\SV$ and $t\in\ST$, where $\beta^{Z(\ST)}$ (resp.\ $\beta^{Z(\ST')}$) is the half\--braiding (and not the braiding) on $Z(\ST)$ (resp.\ $Z(\ST')$).
\end{Definition}

\begin{Remark}\label{remCompMorphismsOfPairs}
Given three pairs $(\ST, F)$, $(\ST',F')$, $(\ST'',F'')$ and two morphisms $G\colon(\ST,F)\longrightarrow(\ST,F')$ and $G'\colon(\ST',F')\longrightarrow(\ST'',F'')$ of pairs, the composition $(H,\varphi_2^H,\varphi_0^H,\gamma^H)$ with $(H,\varphi_2^H,\varphi_0^H)\coloneqq G'G$ as monoidal functors (cf.\ Rem.\,\ref{remCompMonoidalFunctors}) and
\begin{equation}\tag{compos.\,$\gamma$}\label{composGamma}
    \gamma_v^H\coloneqq\left[ F''(v)\xrightarrow[]{\gamma_v^{G'}}G'F'(v)\xrightarrow[]{G'(\gamma_v^G)}G'GF(v)\right]
\end{equation}
is again a morphism of pairs.
\end{Remark}

\begin{Definition}[{\cite[Def.\,3.3]{HPT16}}]\label{def2MorphismOfPairs}
Let $(\ST,F)$, $(\ST',F')$ be two pairs and let $G,G'\colon(\ST,F)\longrightarrow(\ST',F')$ be two morphisms of pairs between them. A \emph{2\==morphism of pairs} $\eta\colon G\Longrightarrow G'$ is a monoidal natural transformation which is compatible with the action coherence, i.e.\ the diagram
\begin{equation}\tag{nat.\,tr.\,\&\,$\gamma$}\label{condition2MorphismOfPairs}
\begin{tikzcd}
& GF(v) \arrow[dd, "\eta_{F(v)}"] \\
F'(v) \arrow[rd, "\gamma_v^{G'}"'] \arrow[ru, "\gamma_v^G"] & \\
& G'F(v) 
\end{tikzcd}
\end{equation}
commutes for all $v\in\SV$.
\end{Definition}

For later purposes, we need to understand equivalences of pairs.

\begin{Definition}\label{defEquivalencePairs}
Let $(\ST,F)$, $(\ST',F')$ be two pairs and let $G\colon (\ST,F)\longrightarrow(\ST',F')$ be a morphism of pairs. We call $G$ an \emph{equivalence of pairs} if there exist a morphism of pairs $G'\colon (\ST',F')\longrightarrow(\ST,F)$ and 2\==isomorphisms of pairs $\eta\colon \Id_{(\ST,F)}\Longrightarrow G'G$ and $\varepsilon\colon G G'\Longrightarrow\Id_{(\ST',F')}$.
\end{Definition}

As before, we get 

\begin{Proposition}\label{propAlternativeEquivalencePairs}
Let $(\ST,F),(\ST',F')$ be two pairs and let $G\colon (\ST,F)\longrightarrow(\ST',F')$ be a morphism of pairs. Then the following statements are equivalent:
\begin{itemize}
    \item[(i)] $G$ is an equivalence of pairs.
    \item[(ii)] $G$ is part of an adjoint equivalence of pairs.
    \item[(iii)] $G$ is an equivalence of ordinary categories.
\end{itemize}
\end{Proposition}

The proof works similar to that of Proposition \ref{propAlternativeEquivalenceMonoidalCategories} and is omitted at this point.

\begin{Definition}\label{def2CategoryPairs}
Pairs, morphisms of pairs and 2\==morphisms of pairs form a strict 2\==category, which we denote by $\Pairs$.
\end{Definition}

\section{Equivalence between unital module monoidal categories and pairs}\label{secEquivalence}

This section aims to prove the main theorem.

\begin{Theorem}\label{thmEquivalenceUnitalModuleMonoidalCategoriesAndPairs}
Let $\SV$ be a braided monoidal category. There is an equivalence \[\Pairs\cong \UModMonCat\]
between the 2\==category of pairs and the 2\==category of unital module monoidal categories.
\end{Theorem}

In order to prove this, we explicitly write down a 2\==functor $\Psi\colon\Pairs\longrightarrow \UModMonCat$ and show that it is
\begin{itemize}
    \item essentially surjective on objects and
    \item the contained functors
    \[\Psi_{(\SM,F),(\SM',F')}\colon \Pairs((\SM,F),(\SM',F'))\longrightarrow \UModMonCat(\Psi(\SM,F),\Psi(\SM',F'))\]
    between Hom\==categories are
    \begin{itemize}
        \item bijective\footnote{Replacing the bijectivity condition on 1\==morphisms by the usual requirement that $\Psi_{(\SM,F),(\SM',F')}$ is essentially surjective (and hence an equivalence) only gives that $\Psi$ is an equivalence of bicategories (i.e.\ the coherence data of the bifunctors contained in the equivalence might be non\--trivial) \cite[Thm.\,7.4.1]{JY} rather than an equivalence of 2\==categories (i.e.\ the coherence data of the bifunctors in the equivalence are identities, turning them into 2\==functors).} on 1\==morphisms and
        \item bijective on 2\==morphisms, i.e.\ $\Psi_{(\SM,F),(\SM',F')}$ is fully faithful,
    \end{itemize}
\end{itemize}
which then implies that $\Psi$ is an equivalence of 2\==categories \cite[Thm.\,7.5.8]{JY}.

Since the whole section treats unital monoidal structures rather than non\--unital monoidal ones, we abbreviate such structures as monoidal.

This section is structured as follows. In Section \ref{secEquivalenceConstructionPsi} we construct the 2\==functor $\Psi$, in Section \ref{secEquivalenceObjects} we show that it is essentially surjective on objects and, finally, in Section \ref{secEquivalenceHomCategories} we address the second bullet point from above.

\subsection{The construction of the 2--functor $\Psi$}\label{secEquivalenceConstructionPsi}
A 2\==functor is a strict bifunctor between 2\==categories. In this section, we set up the definition of $\Psi$ as a bifunctor (cf.\ Const.\,\ref{conPsi} and Prop.\,\ref{propPsiIsWelldefined}), which turns out to be strict (cf.\ Prop.\,\ref{propPsiIsStrict}) and hence is indeed a 2\==functor.
\begin{Construction}\label{conPsi}
We define a bifunctor $\Psi\colon\Pairs\longrightarrow\UModMonCat$ as follows.

On objects, we set $\Psi(\SM,F)$ to be the module monoidal category obtained by extending the monoidal category $(\SM,\otimes,\1_\SM,a,l,r)$ to $(\SM,\otimes, \1_\SM, a, l, r, \rhd, m, l^\rhd, \alp1, \alp2)$ where we define
\begin{itemize}
    \item the action as
    \begin{align*}
        \rhd \colon \SV\times \SM&\longrightarrow\SM,\\
        (v,m)&\longmapsto v\rhd m\coloneqq F(v)\otimes m,\\
        \left(\left[v\overset{f}{\longrightarrow}v'\right],\left[m\overset{g}{\longrightarrow}m'\right]\right)&\longmapsto \left[v\rhd m \xrightarrow[]{f\,\rhd\, g}v'\rhd m'\right]\coloneqq \left[F(v)\otimes m\xrightarrow[]{F(f)\,\otimes\, g} F(v')\otimes m'\right],
    \end{align*}
    \item the module associativity constraint
    \[m_{v,w,m}\colon \begin{array}{l}(v\otimes w)\rhd m\\= F(v\otimes w)\otimes m\end{array}\longrightarrow\begin{array}{l}v\rhd(w\rhd m)\\=F(v)\otimes(F(w)\otimes m)\end{array}\]
    as
    \[m_{v,w,m}\coloneqq a_{F(v),F(w),m}\circ (\varphi^{-1}_{2,v,w}\otimes \id_m)\]
    for all $v,w\in\SV$ and $m\in\SM$,
    \item the module unit constraint
    \[l_m^\rhd\colon \begin{array}{l}\1_\SV\rhd m\\=F(\1_\SV)\otimes m\end{array}\longrightarrow m\]
    as
    \[l_m^\rhd\coloneqq l_m\circ (\varphi_0^{-1}\otimes \id_m)\]
    for all $m\in\SM$,
    \item the natural isomorphism
    \[\alp1_{v,m,n}\colon \begin{array}{l}(v\rhd m)\otimes n\\=(F(v)\otimes m)\otimes n\end{array}\longrightarrow \begin{array}{l}v\rhd (m\otimes n)\\=F(v)\otimes (m\otimes n)\end{array}\]
    as
    \[\alp1_{v,m,n}\coloneqq a_{F(v),m,n}\]
    for all $v\in\SV$ and $m,n\in\SM$ and
    \item the natural isomorphism
    \[\alp2_{v,m,n}\colon \begin{array}{l} m\otimes (v\rhd n)\\=m\otimes (F(v)\otimes n)\end{array}\longrightarrow\begin{array}{l} v\rhd (m\otimes n)\\=F(v)\otimes (m\otimes n)\end{array}\]
    as
    \[\alp2_{v,m,n}\coloneqq  a_{F(v),m,n}\circ (\beta_{F(v),m}^{-1}\otimes \id_n)\circ a_{m,F(v),n}^{-1}\]
    for all $v\in\SV$ and $m,n\in\SM$, where $\beta_{F(v),-}$ is the half\--braiding on $F^Z(v)=(F(v),\beta_{F(v),-})\in Z(\SM)$, which is by definition an isomorphism in $\SM$.
\end{itemize}
On 1\==morphisms, we set
\[\Psi\left[(\SM,F)\xrightarrow[]{\left(G,\varphi_2^G,\varphi_0^G,\gamma^G\right)}(\SM',F')\right]\coloneqq \left[\Psi(\SM,F)\xrightarrow[]{\left(\Psi(G),\varphi_2^{\Psi(G)},\varphi_0^{\Psi(G)},s^{\Psi(G)}\right)}\Psi(\SM',F')\right]\]
where we define
\begin{itemize}
    \item $\Psi$ to be the identity on the monoidal part of $G$,
    \[\left(\Psi(G),\varphi_2^{\Psi(G)},\varphi_0^{\Psi(G)}\right)\coloneqq\left(G,\varphi_2^G,\varphi_0^G\right),\]
    \item and the module functor constraint
    \[s^{\Psi(G)}\colon \begin{array}{l}v\rhd' G(m)\\=F'(v)\otimes' G(m)\end{array}\longrightarrow \begin{array}{l}G(v\rhd m)\\=G(F(v)\otimes m)\end{array}\]
    as
    \[s^{\Psi(G)}_{v,m}\coloneqq \varphi_{2,F(v),m}^G\circ (\gamma_v^G\otimes'\id_{G(m)})\]
    for all $v\in\SV$ and $m\in\SM$.
\end{itemize}
Finally, on 2\==morphisms we set $\Psi(\eta)=\eta$ for any 2\==morphism of pairs $\eta\colon G\Longrightarrow G'$ between morphisms of pairs $G,G'\colon (\SM,F)\longrightarrow(\SM',F')$.
\end{Construction}

\begin{Proposition}\label{propPsiIsWelldefined}
The bifunctor $\Psi\colon\Pairs\longrightarrow\UModMonCat$ from Construction \ref{conPsi} is well\--defined, i.e.\ we have to check that
\begin{itemize}
\item the data $\Psi(\SM,F)$ form a module monoidal category,
\item the data $\Psi(G,\varphi_2^G,\varphi_0^G,\gamma^G)$ form a module monoidal functor and
\item the datum $\Psi(\eta)$ is a module monoidal natural transformation.
\end{itemize}
\end{Proposition}

\begin{proof}
First, we show that $\Psi(\SM,F)$ is a module monoidal category.

Note, that $(\SM,\otimes,\1_\SM,a,l,r)$ is a monoidal category by assumption. For $(\SM,\rhd,m,l^\rhd)$ to be a module category, we have to check that
\begin{itemize}
    \item $\rhd\colon\SV\times\SM\longrightarrow\SM$ is a functor, which follows from the functoriality of $\otimes$,
    \item $m$ and $l^\rhd$ are natural isomorphisms, which, as compositions of natural isomorphisms, they clearly are,
    \item the diagrams \eqref{pentM} and \eqref{triangleM} from Definition \ref{defModuleCategory} commute. After resolving the definitions from Construction \ref{conPsi}, \eqref{pentM} becomes
    \begin{equation*}
    {\footnotesize\begin{tikzcd}[column sep = huge]
    F((u\otimes v)\otimes w)\otimes m \arrow[d, "F(a)\,\otimes\,\id"']  &  (F(u\otimes v)\otimes F(w))\otimes m\arrow[rd, "\com{\nat a}" description, color = white]\arrow[r, "{a}"] \arrow[l, "{\varphi_2\,\otimes\,\id}"']& F(u\otimes v)\otimes (F(w)\otimes m)\\
    F(u\otimes(v\otimes w))\otimes m\arrow[rd, "\com{\eqref{phi2Hexagon}}" description, color = white]& ((F(u)\otimes F(v))\otimes F(w))\otimes m\arrow[r, "a"'] \arrow[u, "{(\varphi_2\,\otimes\,\id)\,\otimes\,\id}"]\arrow[d, "{a\,\otimes\,\id}"'] & (F(u)\otimes F(v))\otimes (F(w)\otimes m) \arrow[dd, "a"] \arrow[u, "{\varphi_2\,\otimes\,\id}"']\\
    (F(u)\otimes F(v\otimes w))\otimes m \arrow[d, "a"'] \arrow[u, "{\varphi_2\,\otimes\,\id}"]\arrow[rd, "\com{\nat a}" description, color = white]&(F(u)\otimes (F(v)\otimes F(w)))\otimes m   \arrow[d, "a"] \arrow[l, "{(\id\,\otimes\,\varphi_2)\,\otimes\,\id}"]\arrow[ru, "\com{\eqref{pentA}}" description, color = white] &\\
    F(u)\otimes(F(v\otimes w)\otimes m)  &F(u)\otimes((F(v)\otimes F(w))\otimes m)\arrow[r, "\id\ \otimes\,a"'] \arrow[l, "{\id\,\otimes\,(\varphi_2\,\otimes\,\id)}"]& F(u)\otimes
    (F(v)\otimes(F(w)\otimes m))
    \end{tikzcd}}
    \end{equation*}
    which commutes, for all $u,v,w\in\SV$ and $m\in\SM$, due to the reasons indicated in boxes, and \eqref{triangleM} becomes
    \begin{equation*}
    \begin{tikzcd}
    F(v\otimes\1_\SV)\otimes m \arrow[rdd, "{F(r)\,\otimes\,\id}"']& (F(v)\otimes F(\1_\SV))\otimes m \arrow[r, "a"] \arrow[l, "{\varphi_2\,\otimes\,\id}"']\arrow[rd, "\com{\nat a}" description, color = white]\arrow[d, "\com{\eqref{phi0r}}" description, color = white, shift right = 10] & F(v)\otimes (F(\1_\SV)\otimes m) \\
    & (F(v)\otimes \1_\SM)\otimes m \arrow[u, "{(\id\,\otimes\,\varphi_0)\,\otimes\,\id}"'] \arrow[r, "a"'] \arrow[d, "{r\,\otimes\,\id}"']\arrow[d, shift left = 10, shorten >= 0.5cm, pos = 0.4, "\com{\eqref{triangleA}}" description, color = white] & F(v)\otimes(\1_\SM\otimes m) \arrow[u, "{\id\,\otimes\,(\varphi_0\,\otimes\,\id)}"'] \arrow[ld, "{\id\,\otimes\,l}"]\\
    & F(v)\otimes m &
    \end{tikzcd}
    \end{equation*}
    which commutes, for all $v\in\SV$ and $m\in\SM$, due to the reasons indicated in boxes.
\end{itemize}

Further, we have that $\alp1$ and $\alp2$ are, as compositions of natural isomorphisms, again natural isomorphisms.

We are left to check that the remaining conditions in form of commutative diagrams from Definition \ref{defUnitalModuleMonoidalCategory} are satisfied. One can see that after resolving the definitions from Construction \ref{conPsi},
\begin{itemize}
    \item diagram \eqref{a1a} is precisely \eqref{pentA},
    \item diagram \eqref{a1m} is a combination of \eqref{pentA} and naturality of $a$,
    \item diagram \eqref{a2a} becomes
    \begin{equation*}
    \begin{tikzcd}
    & (l\otimes m)\otimes (F(v)\otimes n)\arrow[ld, "a"'] & \\
    l\otimes (m\otimes (F(v)\otimes n))\arrow[dr, "\com{\eqref{pentA}}" description, color = white] & ((l\otimes m)\otimes F(v))\otimes n \arrow[u, "a"'] \arrow[d, "{a\,\otimes\,\id}"]& F(v)\otimes((l\otimes m)\otimes n) \arrow[ddddd, "{\id\,\otimes\,a}", bend left, shift left=14] \\
    l\otimes ((m\otimes F(v))\otimes n) \arrow[rd, "\com{\nat a}" description, color = white]\arrow[u, "{\id\,\otimes\,a}"]& (l\otimes (m\otimes F(v)))\otimes n \arrow[l, "a"]& (F(v)\otimes(l\otimes m))\otimes n \arrow[dd, shift left = 10, "\com{\eqref{pentA}}" description, color = white]\arrow[u, "a"] \arrow[lu, "{\beta\,\otimes\,\id}"] \\
    l\otimes((F(v)\otimes m)\otimes n) \arrow[dd, "\com{\eqref{pentA}}" description, shift left =  25, color = white]\arrow[u, "{\id\,\otimes\,(\beta\,\otimes\,\id)}"] \arrow[dd, "{\id\,\otimes\,a}"']& (l\otimes(F(v)\otimes m))\otimes n \arrow[ru, "\com{\eqref{defDrinfeldHalfBraiding}}" description, color = white]\arrow[u, "{(\id\,\otimes\,\beta)\,\otimes\,\id}"] \arrow[l, "a"] & \\
    & ((l\otimes F(v))\otimes m)\otimes n \arrow[rd, "\com{\nat a}" description, color = white] \arrow[u, "{a\,\otimes\,\id}"'] \arrow[d, "a"]& ((F(v)\otimes l)\otimes m)\otimes n \arrow[uu, "{a\,\otimes\,\id}"] \arrow[d, "a"'] \arrow[l, "{(\beta\,\otimes\,\id)\,\otimes\,\id}"]\\
    l\otimes(F(v)\otimes(m\otimes n)) & (l\otimes F(v))\otimes (m\otimes n) \arrow[l, "a"] & (F(v)\otimes l)\otimes (m\otimes n) \arrow[d, "a"'] \arrow[l, "{\beta\,\otimes\,(\id\,\otimes\,\id)}"]\\
    & & F(v)\otimes(l\otimes(m\otimes n))
    \end{tikzcd}
    \end{equation*}
    which commutes, for all $v\in\SV$ and $l,m,n\in\SM$, due to the reasons indicated in boxes,
    \item diagram \eqref{a2m} becomes
    \begin{equation*}
    {\footnotesize\begin{tikzcd}
    m\otimes(F(v\otimes w)\otimes n)\arrow[rd, "\com{\nat a}" description, color = white]& (m\otimes F(v\otimes w))\otimes n \arrow[l, "a"] \arrow[d, shift left = 20, "\com{\nat \beta}" description, color = white] & (F(v\otimes w)\otimes m)\otimes n \arrow[d, "a"] \arrow[l, "{\beta\,\otimes\,\id}"']\arrow[d, "\com{\nat a}" description, color = white, shift right = 10]\\
    m\otimes ((F(v)\otimes F(w))\otimes n) \arrow[d, "{\id\,\otimes\,a}"'] \arrow[u, "{\id\,\otimes\,(\varphi_2\,\otimes\,\id)}"]\arrow[dd, "\com{\eqref{pentA}}" description, shift left = 33, color = white] & (m\otimes (F(v)\otimes F(w)))\otimes n \arrow[l, "a"] \arrow[u, "{(\id\,\otimes\,\varphi_2)\,\otimes\,\id}"] & F(v\otimes w)\otimes (m\otimes n) \\
    m\otimes(F(v)\otimes(F(w)\otimes n))&& (F(v)\otimes F(w))\otimes (m\otimes n) \arrow[u, "{\varphi_2\,\otimes\,\id}"'] \arrow[dddd, "a", bend left, shift left=17]\\
    (m\otimes F(v))\otimes (F(w)\otimes n) \arrow[u, "a"] \arrow[rd, "\com{\nat a}" description, color = white] & ((m\otimes F(v))\otimes F(w))\otimes n \arrow[uu, "{a\,\otimes\,\id}"] \arrow[l, "a"] \arrow[rd, "\com{\eqref{defDrinfeldTensor}}" description, color = white] & ((F(v)\otimes F(w))\otimes m)\otimes n \arrow[d, "{a\,\otimes\,\id}"] \arrow[luu, "{\beta\,\otimes\,\id}", start anchor = north west, end anchor = south east] \arrow[u, "a"'] \arrow[uuu, "{(\varphi_2\,\otimes\,\id)\,\otimes\,\id}"', bend left, shift left=17] \\
    (F(v)\otimes m)\otimes (F(w)\otimes n) \arrow[u, "{\beta\,\otimes\,\id}"] \arrow[dd, "a"'] \arrow[dd, "\com{\eqref{pentA}}" description, shift left = 29, color = white] & ((F(v)\otimes m)\otimes F(w))\otimes n \arrow[u, "{(\beta\,\otimes\,\id)\,\otimes\,\id}"] \arrow[l, "a"] \arrow[d, "{a\,\otimes\,\id}"'] & (F(v)\otimes(F(w)\otimes m))\otimes n \arrow[ld, "{(\id\,\otimes\,\beta)\,\otimes\,\id}"'] \arrow[d, "a"]\arrow[d, "\com{\eqref{pentA}}" description, shift left = 15, color = white] \\
    & (F(v)\otimes (m\otimes F(w)))\otimes n \arrow[d, "a"'] \arrow[r, "\com{\nat a}" description, color = white]& F(v)\otimes((F(w)\otimes m)\otimes n) \arrow[ld, "{\id\,\otimes\,(\beta\,\otimes\,\id)}"] \arrow[d, "{\id\,\otimes\,a}"'] \\
    F(v)\otimes (m\otimes (F(w)\otimes n))& F(v)\otimes((m\otimes F(w))\otimes n) \arrow[l, "{\id\,\otimes\,a}"] & F(v)\otimes(F(w)\otimes(m\otimes n)) 
    \end{tikzcd}}
    \end{equation*}
    which commutes, for all $v,w\in\SV$ and $m,n\in\SM$, due to the reasons indicated in boxes,
    \item diagram \eqref{a1a2} is a combination of two times \eqref{pentA} and naturality of $a$ and
    \item diagram \eqref{brac} becomes
    \begin{equation*}
    \begin{tikzcd}
    (F(w)\otimes m)\otimes (F(v)\otimes n) \arrow[rrd, "\com{\eqref{pentA}}" description, color = white] \arrow[rr, "a"] && F(w)\otimes(m\otimes(F(v)\otimes n)) \\
    ((F(w)\otimes m)\otimes F(v))\otimes n \arrow[u, "a"]\arrow[rd, "\com{\eqref{defDrinfeldHalfBraiding}}" description, color = white] \arrow[r, "{a\,\otimes\,\id}"] & (F(w)\otimes (m\otimes F(v))\otimes n \arrow[r, "a"]\arrow[rd, "\com{\nat a}" description, color = white] & F(w)\otimes((m\otimes F(v))\otimes n) \arrow[u, "{\id\,\otimes\,a}"']\\
    (F(v)\otimes (F(w)\otimes m))\otimes n \arrow[u, "{\beta\,\otimes\,\id}"] \arrow[d, "a"'] & (F(w)\otimes(F(v)\otimes m))\otimes n \arrow[u, "{(\id\,\otimes\,\beta)\,\otimes\,\id}"] \arrow[r, "a"]& F(w)\otimes((F(v)\otimes m)\otimes n) \arrow[u, "{\id\,\otimes\,(\beta\,\otimes\,\id)}"'] \arrow[d, "{\id\,\otimes\,a}"] \\
    F(v)\otimes((F(w)\otimes m)\otimes n) \arrow[d, shift left = 10, "\com{\eqref{pentA}}" description, color = white]\arrow[d, "{\id\,\otimes\,a}"']& ((F(w)\otimes F(v))\otimes m)\otimes n \arrow[u, "{a\,\otimes\,\id}"] \arrow[rdd, "a"] \arrow[ru, "\com{\eqref{pentA}}" description, color = white]& F(w)\otimes(F(v)\otimes(m\otimes n)) \\
    F(v)\otimes (F(w)\otimes (m\otimes n)) & ((F(v)\otimes F(w))\otimes m)\otimes n \arrow[u, "{(\beta\,\otimes\,\id)\,\otimes\,\id}"] \arrow[luu, pos = 0.4, start anchor = north west, end anchor = south east, "{a\,\otimes\,\id}"] \arrow[ld, "a", start anchor = south west] &\\
    (F(v)\otimes F(w))\otimes (m\otimes n) \arrow[rr, "\com{\nat a}" description, color = white, shorten <= 2cm, shorten >= 2cm, shift left = 8]\arrow[rr, shift right = 10, "\com{\eqref{braidedMonoidalFunctor}}" description, color = white]\arrow[d, "{\varphi_2\,\otimes\,\id}"'] \arrow[u, "a"] \arrow[rr, "{\beta\,\otimes\,\id}"]\arrow[rr,"{=\,\beta^{Z(\SM)}\,\otimes\,\id}"'] && (F(w)\otimes F(v))\otimes (m\otimes n) \arrow[uu, "a"'] \arrow[d, "{\varphi_2\,\otimes\,\id}"] \\
    F(v\otimes w)\otimes (m\otimes n) \arrow[rr, "{F(\beta^\SV)\,\otimes\,\id}"'] && F(w\otimes v)\otimes (m\otimes n) 
    \end{tikzcd}
    \end{equation*}
    which commutes, for all $v,w\in\SV$ and $m,n\in\SM$, due to the reasons indicated in boxes.
\end{itemize}

Second, we show that $\Psi(G,\varphi_2^G,\varphi_0^G,\gamma^G) = (G,\varphi_2^G,\varphi_0^G,s^{\Psi(G)})$ is a module monoidal functor.

Note, that $(G,\varphi_2^G,\varphi_0^G)$ is a monoidal functor by assumption. For $(G,s^{\Psi(G)})$ to be a module functor, we have to check that
\begin{itemize}
    \item $s^{\Psi(G)}$ is a natural isomorphisms, which, as composition of natural isomorphisms, it clearly is,
    \item the pentagon \eqref{sPentagon} and triangle \eqref{sTriangle} from Definition \ref{defModuleFunctor} commute. After resolving the definitions from Construction \ref{conPsi}, \eqref{sPentagon} becomes
    \begin{equation*}
    \begin{tikzcd}
    F'(v\otimes w)\otimes' G(m) \arrow[r, "{\gamma^G\,\otimes'\,\id}"] \arrow[rd, "\com{\eqref{monNatTrafoPhi2}}" description, color = white] & GF(v\otimes w)\otimes' G(m) \arrow[r, "\varphi_2^G"] \arrow[rd, "\com{\nat \varphi_2^G}" description, color = white] & G(F(v\otimes w)\otimes m) \\
    (F'(v)\otimes'F'(w))\otimes'G(m) \arrow[d, "a'"'] \arrow[u, "{\varphi_2^{F'}\,\otimes'\,\id}"] \arrow[rd, "{(\gamma^G\,\otimes'\,\gamma^G)\,\otimes'\,\id}"] \arrow[d, "\com{\nat a'}" description, color = white, shift left = 8]& G(F(v)\otimes F(w))\otimes' G(m) \arrow[r, "\varphi_2^G"'] \arrow[u, "{G(\varphi_2^F)\,\otimes'\,\id}"'] \arrow[rddd, "\com{\eqref{phi2Hexagon}}" description, shorten <= 2cm, shorten >= 2cm, color = white] & G((F(v)\otimes F(w))\otimes m) \arrow[ddd, "G(a)"] \arrow[u, "{G(\varphi_2^F\,\otimes\,\id)}"'] \\
    F'(v)\otimes'(F'(w)\otimes'G(m)) \arrow[d, "{\id\,\otimes'\,(\gamma^G\,\otimes'\,\id)}"'] \arrow[rd, "{\gamma^G\,\otimes'\,(\gamma^G\,\otimes'\,\id)}"]  \arrow[d, "\com{\funct \otimes'}" description, shift left=8, color = white] & (GF(v)\otimes'GF(w))\otimes'G(m) \arrow[d, "a'"] \arrow[u, "{\varphi_2^G\,\otimes'\,\id}"']& \\
    F'(v)\otimes'(GF(w)\otimes'G(m)) \arrow[r, "{\gamma^G\,\otimes'\,\id}"'] \arrow[d, "{\id\,\otimes'\,\varphi_2^G}"'] \arrow[rd, "\com{\funct \otimes'}" description, color = white]& GF(v)\otimes'(GF(w)\otimes'G(m)) \arrow[d, "{\id\,\otimes'\,\varphi_2^G}"] & \\
    F'(v)\otimes'G(F(w)\otimes m) \arrow[r, "{\gamma^G\,\otimes'\,\id}"']& GF(v)\otimes'G(F(w)\otimes m) \arrow[r, "\varphi_2^G"']& G(F(v)\otimes (F(w)\otimes m)) 
    \end{tikzcd}
    \end{equation*}
    which commutes, for all $v,w\in \SV$ and $m\in\SM$, due to the reasons indicated in boxes, and \eqref{sTriangle} becomes
    \begin{equation*}
    \begin{tikzcd}
    F'(\1_\SV)\otimes'G(m) \arrow[r, "{\gamma^G\,\otimes'\,\id}"]\arrow[rd, "\com{\eqref{monNatTrafoPhi0}}" description, color = white]& GF(\1_\SV)\otimes'G(m) \arrow[r, "\varphi_2^G"]\arrow[d, shift left = 30, "\com{\nat \varphi_2^G}" description, color = white] & G(F(\1_\SV)\otimes m)\\
    \1_{\SM'}\otimes'G(m) \arrow[u, "{\varphi_0^{F'}\,\otimes'\,\id}"] \arrow[rd, "l'"'] \arrow[r, "{\varphi_0^G\,\otimes'\,\id}"'] & G(\1_\SM) \otimes' G(m) \arrow[u, "{G(\varphi_0^F)\,\otimes'\,\id}"'] \arrow[r, "\varphi_2^G"'] \arrow[d, "\com{\eqref{phi0l}}" description, color = white]& G(\1_\SM\otimes m) \arrow[ld, "G(l)"] \arrow[u, "{G(\varphi_0^F\,\otimes\,\id)}"'] \\
    & G(m)& 
    \end{tikzcd}
    \end{equation*}
    which commutes, for all $m\in\SM$, due to the reasons indicated in boxes.
\end{itemize}

Further, we have to check that the remaining conditions in form of commutative diagrams from Definition \ref{defUnitalModuleMonoidalFunctor} are satisfied, i.e.\ diagrams \eqref{diagModuleMonoidalFunctorCompAlpha1} and \eqref{diagModuleMonoidalFunctorCompAlpha2} commute. After resolving the definitions from Construction \ref{conPsi}, \eqref{diagModuleMonoidalFunctorCompAlpha1} reduces quite easily to \eqref{phi2Hexagon} via naturality of $a'$ and functoriality of $\otimes'$, while \eqref{diagModuleMonoidalFunctorCompAlpha2} becomes
\begin{equation*}
    {\footnotesize
    \begin{tikzcd}
    G(m)\otimes'(F'(v)\otimes'G(n)) \arrow[ddd, "{\id\,\otimes'\,(\gamma^G\,\otimes'\,\id)}"', bend right, shift right=23] & F'(v)\otimes'(G(m)\otimes'G(n)) \arrow[r, "{\id\,\otimes'\,\varphi_2^G}"] \arrow[rd, "{\gamma^G\,\otimes'\,\id}"]& F'(v)\otimes'G(m\otimes n) \arrow[d, "\com{\funct \otimes'}" description, color = white]\arrow[dd, "{\gamma^G\,\otimes'\,\id}", bend left, shift left=20] \\
    (G(m)\otimes' F'(v))\otimes' G(n) \arrow[d, "\com{\nat a}" description, color = white, shift right = 10] \arrow[u, "a'"'] \arrow[d, "{(\id\,\otimes'\,\gamma^G)\,\otimes'\,\id}"] & (F'(v)\otimes' G(m))\otimes' G(n) \arrow[u, "a'"] \arrow[r, "\com{\nat a}" description, color = white]\arrow[l, "{\beta^{Z(\SM')}\,\otimes'\,\id}"] \arrow[d, "{(\gamma^G\,\otimes'\,\id)\,\otimes'\,\id}"'] & GF(v)\otimes'(G(m)\otimes' G(n)) \arrow[d, "{\id\,\otimes'\,\varphi_2^G}"] \\
    (G(m)\otimes'GF(v))\otimes'G(n) \arrow[d, "a'"] \arrow[dd, "{\varphi_2^G\,\otimes'\,\id}", bend left, shift left=20] & (GF(v)\otimes'G(m))\otimes'G(n)\arrow[dd, shift right = 10, "\com{\eqref{actionCoherenceAndBraiding}}" description, color = white] \arrow[dd, "{\varphi_2^G\,\otimes'\id}"] \arrow[ru, "a'"]& GF(v)\otimes' G(m\otimes n) \arrow[dddd, "\varphi_2^G"]\\
    G(m)\otimes'(GF(v)\otimes'G(n)) \arrow[d, "\com{\eqref{phi2Hexagon}}" description, color = white]\arrow[ddd, "{\id\,\otimes'\,\varphi_2^G}"', bend right, shift right=15] &&\\
    G(m\otimes F(v))\otimes' G(n) \arrow[d, "\varphi_2^G"] & G(F(v)\otimes m)\otimes' G(n) \arrow[d, "\varphi_2^G"] \arrow[l, "{G(\beta^{Z(\SM)})\,\otimes'\,\id}"] \arrow[ruu, "\com{\eqref{phi2Hexagon}}" description, color = white]&\\
    G((m\otimes F(v))\otimes n) \arrow[rd, "G(a)"']\arrow[r, shift left = 5, "\com{\nat \varphi_2^G}" description, color = white]& G((F(v)\otimes m)\otimes n) \arrow[l, "{G(\beta^{Z(\SM)}\,\otimes\,\id)}"] \arrow[rd, "G(a)"]&\\
    G(m)\otimes'G(F(v)\otimes n) \arrow[r, "\varphi_2^G"'] & G(m\otimes(F(v)\otimes n)) & G(F(v)\otimes(m\otimes n))
    \end{tikzcd}}
\end{equation*}
where $\beta^{Z(\SM)}$ (resp.\ $\beta^{Z(\SM')}$) is the half\--braiding (and not the braiding) on $Z(\SM)$ (resp.\ $Z(\SM')$) and which commutes, for all $v\in \SV$ and $m,n\in\SM$, due to the reasons indicated in boxes.

Finally, we show that $\Psi(\eta)$ is a module monoidal natural transformation.

Note, that $\eta$ is a monoidal natural transformation by assumption. Further, after resolving the definitions from Construction \ref{conPsi}, diagram \eqref{monNatTrafoS} becomes
\begin{equation*}
    \begin{tikzcd}
    F'(v)\otimes' G(m) \arrow[dd, "{\id\,\otimes'\,\eta_m}"'] \arrow[r, "{\gamma^G\,\otimes'\,\id}"] \arrow[rd, "{\gamma^{G'}\,\otimes'\,\id}"', bend right = 10] & GF(v)\otimes' G(m) \arrow[d, shift right = 10,"\com{\eqref{condition2MorphismOfPairs}}" description, color = white]\arrow[r, "\varphi_2^G"] \arrow[d, "{\eta_{F(v)}\,\otimes'\,\id}"] \arrow[dd, "{\eta_{F(v)}\,\otimes'\,\eta_m}", bend left, shift left=16] & G(F(v)\otimes m) \arrow[dd, "{\eta_{F(v)\otimes m}}", shift left = 5] \arrow[dd, pos = 0.8, shift right = 8, shorten <=1.5cm, "\com{\eqref{monNatTrafoPhi2}}" description, color = white]\\
    & G'F(v)\otimes' G(m) \arrow[d, shift left = 10, "\com{\funct \otimes'}" description, color = white]\arrow[d, "{\id\,\otimes'\,\eta_m}"'] & \\
    F'(v)\otimes' G'(m) \arrow[ru, "\com{\funct \otimes'}" description, color = white]\arrow[r, "{\gamma^{G'}\,\otimes'\,\id}"'] & G'F(v)\otimes'G'(m) \arrow[r, "\varphi_2^{G'}"']& G'(F(v)\otimes m)
\end{tikzcd}
\end{equation*}
which commutes, for all $v\in\SV$ and $m\in\SM$, due to the reasons indicated in boxes. Hence, $\eta$ is also a module natural transformation and in total a module monoidal natural transformation.
\end{proof}

For now, we have a bifunctor $\Psi\colon \Pairs\longrightarrow\UModMonCat$, but since we want to establish an equivalence of 2\==categories, we need $\Psi$ to be a 2\==functor, i.e. a strict bifunctor. This is ensured by the following statement, which is easy but lengthy to prove and therefore not proven in this paper.

\begin{Proposition}\label{propPsiIsStrict}
    The bifunctor $\Psi\colon\Pairs\longrightarrow\UModMonCat$ from Construction \ref{conPsi} is strict, i.e.\ it is a 2\==functor.
\end{Proposition}

\subsection{Equivalence on objects}\label{secEquivalenceObjects}
As discussed in the introduction of Section \ref{secEquivalence}, for the 2\==functor $\Psi\colon\Pairs\longrightarrow\UModMonCat$ from Section \ref{secEquivalenceConstructionPsi} to be an equivalence we first need that it is essentially surjective on objects, which is implied by the statement that equivalence classes of $\Pairs$ and $\UModMonCat$ are in bijection. In fact, we prove in Theorem \ref{thmEquivalenceUnitalModuleMonoidalCategoriesAndPairsOnObjects} a slightly stronger version of this, namely we see that so\--called $\Id$\==equivalence classes of objects in $\Pairs$ and $\UModMonCat$ are in bijection.

\begin{Definition}\label{defIdEquivalence}
Given an equivalence $F$ of module monoidal categories (resp.\ pairs), we call it an \emph{$\Id$\==equivalence}, if the underlying functor of $F$ is the identity. We call two module monoidal categories (resp.\ pairs) \emph{$\Id$\==equivalent}, if there exists an $\Id$\==equivalence between them.
\end{Definition}
 
\begin{Theorem}\label{thmEquivalenceUnitalModuleMonoidalCategoriesAndPairsOnObjects}
Let $\SV$ be a braided monoidal category. There is a bijection
\[\left\{\begin{array}{l}\text{$\Id$\==equivalence classes}\\\text{of pairs }(\ST,F)\text{ over }\SV\end{array}\right\}\longleftrightarrow\left\{\begin{array}{l}\text{$\Id$\==equivalence classes of unital}\\\text{module monoidal categories }\SM\text{ over }\SV\end{array}\right\}.\]
\end{Theorem}

Before we prove this theorem, we need a concept how to construct a pair out of a module monoidal category.

\begin{Lemma}\label{lemUnitalModuleMonoidalCategoryToPair}
Given a module monoidal category $(\SM,\otimes, \1_\SM, a, l, r, \rhd, m, l^\rhd, \alp1, \alp2)$, we construct a pair $(\ST,F)$ as we set $\ST\coloneqq \SM$ as monoidal categories and define a braided central monoidal functor $(F,F^Z)\colon \SV\longrightarrow \ST$ (cf.\ Def.\,\ref{defBraidedCentralFunctor}) as follows.
\begin{itemize}
    \item The lift $F^Z\colon\SV\longrightarrow Z(\ST)$ is given by
    \begin{align*}
        F^Z\colon \SV & \longrightarrow Z(\ST),\\
        v&\longmapsto (F(v)\coloneqq v\rhd \1_\ST,\beta_{F(v),-}),\\
        \left[v\overset{f}{\longrightarrow} v'\right]&\longmapsto \left[F(v)\xrightarrow[]{F(f)}F(v')\right]\coloneqq\left[v\rhd \1_\ST\xrightarrow[]{f\,\rhd\,\id_{\1_\ST}} v'\rhd \1_\ST\right],
    \end{align*}
    where the half\--braiding $\beta_{F(v),-}$ on the object $F(v)$ with components
    \[\beta_{F(v),t}\colon\begin{array}{l}F(v)\otimes t\\=(v\rhd\1_\ST)\otimes t\end{array}\longrightarrow\begin{array}{l}t\otimes F(v)\\=t\otimes (v\rhd\1_\ST)\end{array}\]
    is given by
    \[\beta_{F(v),t}\coloneqq (\alp2_{v,t,\1_\ST})^{-1}\circ (\id_v\rhd\beta_{\1_\ST,t})\circ\alp1_{v,\1_\ST,t}\]
    for all $t\in\ST$.
    \item The monoidal structure on $F^Z$ is given by
    \[\varphi_2\colon \begin{array}{l}F(v)\otimes F(w)\\=(v\rhd\1_\ST)\otimes( w\rhd\1_\ST)\end{array}\longrightarrow \begin{array}{l}F(v\otimes w)\\=(v\otimes w)\rhd\1_\ST\end{array}\]
    defined as
    \[\varphi_{2,v,w}\coloneqq (\id_{v\otimes w}\rhd r_{\1_\ST})\circ m^{-1}_{v,w,\1_\ST\otimes\1_\ST}\circ (\id_v\rhd\alp2_{\1_\ST,w,\1_\ST})\circ \alp1_{v,\1_\ST,w\rhd\1_\ST}\]
    for all $v,w\in\SV$, and
    \[\varphi_0\colon \1_\ST\longrightarrow \begin{array}{l}F(\1_\SV)\\=\1_\SV\rhd\1_\ST\end{array}\]
    defined as
    \[\varphi_0\coloneqq (l^\rhd_{1_\ST})^{-1}.\]
\end{itemize}
\end{Lemma}
The proof of this lemma is quite lengthy since it involves many large to huge commutative diagrams to be checked. Therefore, we defer its proof to Appendix \ref{secAppendixProofLemmaUnitalModuleMonoidalCategoryToPair}.

\begin{proof}[Proof of Theorem \ref{thmEquivalenceUnitalModuleMonoidalCategoriesAndPairsOnObjects}]
We divide the proof of this theorem into two parts.

Part 1: Given a module monoidal category $(\SM,\otimes,\1_\SM,a,l,r,\rhd,m,l^\rhd,\alp1,\alp2)$, first applying Lemma \ref{lemUnitalModuleMonoidalCategoryToPair} and then applying the 2\==functor $\Psi$ from Section \ref{secEquivalenceConstructionPsi}, we end up with a module monoidal category equivalent to the original one as module monoidal categories in the sense of Definition \ref{defEquivalenceUnitalModuleMonoidalCategories}.

Part 2: Given a pair $(\ST,F)$, first applying the 2\==functor $\Psi$ from Section \ref{secEquivalenceConstructionPsi} and then applying Lemma \ref{lemUnitalModuleMonoidalCategoryToPair}, we end up with a pair equivalent to the original one as pairs in the sense of Definition \ref{defEquivalencePairs}.\bigskip

{\bfseries Part 1}

Given a module monoidal category $(\SM,\otimes,\1_\SM,a,l,r,\rhd,m,l^\rhd,\alp1,\alp2)$, applying Lemma \ref{lemUnitalModuleMonoidalCategoryToPair} yields a pair $(\ST,F)$ and thereafter applying the 2\== functor $\Psi$ from Section \ref{secEquivalenceConstructionPsi}, we end up with another module monoidal category, denoted by $(\SM',\otimes',\1_\SM',a',l',r',\rhd',m',l^{\rhd'},\alpd1,\alpd2)$, where
\begin{itemize}
    \item $\SM=\SM'$ as monoidal categories, i.e.\ $(\SM,\otimes,\1_\SM,a,l,r)=(\SM',\otimes',\1_\SM',a',l',r')$,
    \item the action $\rhd'$ is given by
    \begin{align*}
        v\rhd'm&\overset{\ref{conPsi}}{=}F(v)\otimes m\\
        &\overset{\ref{lemUnitalModuleMonoidalCategoryToPair}}{=}(v\rhd\1_\SM)\otimes m,\\
        \left[v\rhd'm\xrightarrow[]{f\rhd'g}w\rhd'n\right]&\overset{\ref{conPsi}}{=}\left[F(v)\otimes m\xrightarrow[]{F(f)\otimes g}F(w)\otimes n\right]\\
        &\overset{\ref{lemUnitalModuleMonoidalCategoryToPair}}{=}\left[(v\rhd\1_\SM)\otimes m\xrightarrow[]{(f\rhd\id_{\1_\SM})\otimes g}(w\rhd\1_\SM)\otimes n\right]
    \end{align*}
    for all $v,w\in\SV$, $m,n\in\SM$, $f\in\SV(v,w)$ and $g\in\SM(m,n)$,
    \item the module associativity $m'$ has, for all $v,w\in\SV$ and $m\in\SM$, components
    \begin{align*}
    m'_{v,w,m}&=a_{v\rhd\1_\SM,w\rhd\1_\SM,m}\circ ((\alp1_{v,\1_\SM,w\rhd\1_\SM})^{-1}\otimes\id_m) \circ ((\id_v\rhd(\alp2_{w,\1_\SM,\1_\SM})^{-1})\otimes\id_m)\\
    &\phantom{\,=\,}\circ (m_{v,w,\1_\SM\otimes\1_\SM}\otimes\id_m) \circ ((\id_{v\otimes w}\rhd r^{-1}_{\1_\SM})\otimes\id_m) 
    \end{align*}
    since
    \begin{equation*}
    \begin{tikzcd}[column sep = huge]
    \begin{array}{l}(v\otimes w)\rhd'm\\=F(v\otimes w)\otimes m\\=((v\otimes w)\rhd \1_\SM)\otimes m\end{array} \arrow[r, "{m'}"] \arrow[rd, shift left = 5, "\com{1}" description, color = white] & \begin{array}{l}v\rhd'(w\rhd'm)\\=F(v)\otimes(F(w)\otimes m)\\=(v\rhd\1_\SM)\otimes((w\rhd\1_\SM)\otimes m)\end{array} \\
    ((v\otimes w)\rhd(\1_\SM\otimes\1_\SM))\otimes m \arrow[r, "\com{2}" description, color = white] \arrow[u, "{(\id\,\rhd\,r)\,\otimes\,\id}"] \arrow[d, "{m\,\otimes\,\id}"'] & \begin{array}{l}(F(v)\otimes F(w))\otimes m\\=((v\rhd\1_\SM)\otimes(w\rhd\1_\SM))\otimes m\end{array} \arrow[lu, "\varphi_2\,\otimes\,\id"] \arrow[u, "a"'] \arrow[d, "{\alp1\,\otimes\,\id}"'] \\
    (v\rhd(w\rhd(\1_\SM\otimes\1_\SM)))\otimes m   & (v\rhd(\1_\SM\otimes(w\rhd\1_\SM)))\otimes m \arrow[l, "{(\id\,\rhd\,\alp2)\,\otimes\,\id}"]
    \end{tikzcd}
    \end{equation*}
    commutes as \fbox{1} is the definition of $m'$ from Construction \ref{conPsi} and \fbox{2} is the definition of $\varphi_2$ from Lemma \ref{lemUnitalModuleMonoidalCategoryToPair},
    \item the module unit constraint $l^{\rhd'}$ has, for all $m\in\SM$, components
    \[l^{\rhd'}_m=l_m\circ (l^\rhd_{\1_\SM}\otimes\id_m)\]
    since
    \begin{equation*}
    \begin{tikzcd}
    \begin{array}{l}\1_\SV\rhd'm\\=F(\1_\SV)\otimes m\\=(\1_\SV\rhd\1_\SM)\otimes m\end{array} \arrow[rr, shift right = 5, "\com{1}" description, color = white, shorten >= 1cm]\arrow[rr, "l^{\rhd'}"] \arrow[rd, "{l^\rhd\,\otimes\,\id}"', bend right] & & m \\
    & \1_\SM\otimes m \arrow[ru, "l"'] \arrow[lu, shorten <=1cm, pos = 0.8, shift left = 5, "\com{2}" description, color = white] \arrow[lu, "{\varphi_0\,\otimes\,\id}"'] &
    \end{tikzcd}
    \end{equation*}
    commutes as \fbox{1} is the definition of $l^{\rhd'}$ from Construction \ref{conPsi} and \fbox{2} is the definition of $\varphi_0$ from Lemma \ref{lemUnitalModuleMonoidalCategoryToPair},
    \item the natural isomorphism $\alpd1$ has, for all $v\in\SV$ and $m,n\in\SM$, components
    \[\alpd1_{v,m,n}=a_{v\rhd\1_\SM,m,n}\]
    since
    \begin{equation*}
    \begin{tikzcd}
    \begin{array}{l}(v\rhd'm)\otimes n\\=(F(v)\otimes m)\otimes n\\=((v\rhd\1_\SM)\otimes m)\otimes n\end{array} \arrow[r, "{\begin{array}{l}\alp1\\\overset{\ref{conPsi}}{=}\,a\\\overset{\ref{lemUnitalModuleMonoidalCategoryToPair}}{=}\,a\end{array}}"] & \begin{array}{l}v\rhd'(m\otimes n)\\=F(v)\otimes(m\otimes n)\\=(v\rhd\1_\SM)\otimes (m\otimes n)\end{array}
    \end{tikzcd}
    \end{equation*}
    holds and
    \item the natural isomorphism $\alpd2$ has, for all $v\in\SV$ and $m,n\in\SM$, components
    \begin{align*}
    \alpd2_{v,m,n}&=a_{v\rhd\1_\SM,m,n}\circ ((\alp1_{v,\1_\SM,m})^{-1}\otimes \id_n)\circ ((\id_v\rhd\beta^{-1}_{\1_\SM,m})\otimes\id_n)\\
    &\phantom{\,=\,}\circ (\alp2_{v,m,\1_\SM}\otimes\id_n)\circ a_{m,v\rhd\1_\SM,n}^{-1}
    \end{align*}
    since
    \begin{equation*}
    \begin{tikzcd}
    \begin{array}{l}m\otimes(v\rhd' n)\\=m\otimes (F(v)\otimes n)\\=m\otimes ((v\rhd\1_\SM)\otimes n)\end{array} \arrow[rr, "\alpd2"]&& \begin{array}{l}v\rhd'(m\otimes n)\\=F(v)\otimes (m\otimes n)\\=(v\rhd\1_\SM)\otimes (m\otimes n)\end{array}\\
    \begin{array}{l}(m\otimes F(v))\otimes n\\=(m\otimes (v\rhd\1_\SM))\otimes n\end{array} \arrow[u, "a"] \arrow[d, "{\alp2\,\otimes\,\id}"']  \arrow[rr, shift left = 13, "\com{1}" description, color = white]\arrow[rr, shift right = 9, "\com{2}" description, color = white]&& \begin{array}{l}(F(v)\otimes m)\otimes n\\=((v\rhd\1_\SM)\otimes m)\otimes n\end{array} \arrow[u, "a"'] \arrow[d, "{\alp1\,\otimes\,\id}"] \arrow[ll, "{\beta_{F(v),m}\,\otimes\,\id}"'] \\
    (v\rhd(m\otimes\1_\SM))\otimes n && (v\rhd(\1_\SM\otimes m))\otimes n \arrow[ll, "{(\id\,\rhd\,\beta_{\1_\SM,m})\,\otimes\,\id}"] 
    \end{tikzcd}
    \end{equation*}
    commutes as \fbox{1} is the definition of $\alpd2$ from Construction \ref{conPsi} and \fbox{2} is the definition of $\beta_{F(v),m}$ from Lemma \ref{lemUnitalModuleMonoidalCategoryToPair}.
\end{itemize}
Now, we define a module monoidal functor $(G,\varphi_2,\varphi_0,s)\colon\SM\longrightarrow\SM'$ with
\[(G,\varphi_2,\varphi_0)=(\Id,\id,\id)\]
and
\begin{align*}
    s_{v,m}\colon \begin{array}{l}v\rhd'G(m)\\=(v\rhd\1_\SM)\otimes m\end{array}&\longrightarrow \begin{array}{l}G(v\rhd m)\\=v\rhd m\end{array}
\end{align*}
given by
\begin{align*}
    s_{v,m}&\coloneqq (\id_v\rhd l_m)\circ \alp1_{v,\1_\SM,m}.
\end{align*}

Indeed, $G$ is a module monoidal functor, since
\begin{itemize}
    \item $(G,\varphi_2,\varphi_0)=(\Id,\id,\id)\colon\SM\longrightarrow\SM'$ is obviously a monoidal functor and
    \item $(G,s)$ is a module functor, since
    \begin{itemize}
        \item $s$ is as a composition of natural isomorphisms again a natural isomorphism,
        \item the pentagon \eqref{sPentagon}, after resolving the definitions from above, becomes
        \begin{equation*}
        \hspace{-1cm}
        {\footnotesize\begin{tikzcd}[column sep = huge]
        ((v\otimes w)\rhd\1_\SM)\otimes m\arrow[r, "\alp1"] \arrow[rd, "\com{\nat \alp1}" description, color = white]& (v\otimes w)\rhd(\1_\SM\otimes m) \arrow[r, "{\id\,\rhd\,l}"]  \arrow[rdddddddd, bend left = 40, start anchor = south east, end anchor = north east, "m"] \arrow[r, "\com{\nat m}" description, shift right=7, shorten <= 1cm, color = white, pos = 1]& (v\otimes w)\rhd m \arrow[ddddddddd, "m", bend left=15, shift left = 18]\\
        ((v\otimes w)\rhd(\1_\SM\otimes\1_\SM))\otimes m \arrow[u, "{\begin{array}{l}(\id\,\rhd\,l)\,\otimes\,\id\\=\,(\id\,\rhd\,r)\,\otimes\,\id\end{array}}"] \arrow[r, "\alp1"'] \arrow[dd, "{m\,\otimes\,\id}"'] \arrow[rdd, "\com{\eqref{a1m}}" description, color = white] & (v\otimes w)\rhd((\1_\SM\otimes\1_\SM)\otimes m)\arrow[u, "{\id\,\rhd\,(l\,\otimes\,\id)}"] \arrow[d, "m"']\arrow[d, "\com{\nat m}" description, shift left=20, color = white]& \\
        & v\rhd(w\rhd((\1_\SM\otimes\1_\SM)\otimes m)) \arrow[dddddd, "{\id\,\rhd\,(\id\,\rhd\,a)}"', bend left, shift left = 20] \arrow[rdddddd, "\begin{array}{l}\id\,\rhd\,(\id \,\rhd\,(l_{\1_\SM}\,\otimes\,\id))\\=\id\,\rhd\,(\id\,\rhd\,(r_{\1_\SM}\,\otimes\,\id))\end{array}", start anchor = south east, bend right = 10]& \\
        (v\rhd(w\rhd(\1_\SM\otimes\1_\SM)))\otimes m \arrow[r, "\alp1"] \arrow[rd, "\com{\nat \alp1}" description, color = white]& v\rhd((w\rhd(\1_\SM\otimes\1_\SM))\otimes m) \arrow[u, "{\id\,\rhd\,\alp1}"] \arrow[d, "\com{\eqref{a1a2}}" description, shift left=13, color = white]& \\
        (v\rhd(\1_\SM\otimes(w\rhd\1_\SM)))\otimes m\arrow[u, "{(\id\,\rhd\,\alp2)\,\otimes\,\id}"] \arrow[r, "\alp1"'] \arrow[rdd, "\com{\eqref{a1a}}" description, color = white] & v\rhd((\1_\SM\otimes(w\rhd\1_\SM))\otimes m) \arrow[u, "{\id\,\rhd\,(\alp2\,\otimes\,\id)}"] \arrow[dd, "{\id\,\rhd\,a}"']  & \\
        ((v\rhd\1_\SM)\otimes(w\rhd\1_\SM))\otimes m \arrow[u, "{\alp1\,\otimes\,\id}"] \arrow[d, "a"'] & & \\
        (v\rhd\1_\SM)\otimes ((w\rhd \1_\SM)\otimes m) \arrow[r, "\alp1"] \arrow[d, "{\id\,\otimes\,\alp1}"'] \arrow[rd, "\com{\nat \alp1}" description, color = white]& v\rhd(\1_\SM\otimes((w\rhd\1_\SM)\otimes m)) \arrow[d, "{\id\,\rhd\,(\id\,\otimes\,\alp1)}"]& \\
        (v\rhd\1_\SM)\otimes(w\rhd(\1_\SM\otimes m)) \arrow[r, "\alp1"'] \arrow[dd, "{\id\,\otimes\,(\id\,\rhd\,l)}"'] \arrow[dd, "\com{\nat \alp1}" description, shift left = 10, color = white] & v\rhd(\1_\SM\otimes(w\rhd(\1_\SM\otimes m))) \arrow[dd, "{\id\,\rhd\,(\id\,\otimes\,(\id\,\rhd\,l))}"', bend right, shift right = 23] \arrow[d, "{\id\,\rhd\,\alp2}"] & \\
        & v\rhd(w\rhd(\1_\SM\otimes(\1_\SM\otimes m))) \arrow[r, "{\id\,\rhd\,(\id\,\rhd\,(\id\,\otimes\,l))}"] \arrow[d, "\com{\nat \alp2}" description, color = white] \arrow[r, "\com{\eqref{triangleA}}" description, shift left=10, color = white] & v\rhd(w\rhd(\1_\SM\otimes m)) \arrow[d, "{\id\,\rhd\,(\id\,\rhd\,l)}"] \arrow[d] \arrow[d, "\com{\eqref{a2l}}" description, shift right=11, pos = 0.8, color = white] \\
        (v\rhd\1_\SM)\otimes (w\rhd m) \arrow[r, "\alp1"']& v\rhd(\1_\SM\otimes(w\rhd m)) \arrow[r, "{\id\,\rhd\,l}"'] \arrow[ru, "{\id\,\rhd\,\alp2}"] & v\rhd(w\rhd m) 
        \end{tikzcd}}
        \end{equation*}
        which commutes, for all $v,w\in\SV$ and $m\in\SM$, due to the reasons indicated in boxes,
        \item the triangle \eqref{sTriangle} reduces to \eqref{a1lm} for $\SM$ via naturality of $l^\rhd$.
    \end{itemize}
    \item condition \eqref{diagModuleMonoidalFunctorCompAlpha1} reduces quite easily to \eqref{a1a} for $\SM$ via naturality of $\alp1$ and \eqref{aL}, and
    \item condition \eqref{diagModuleMonoidalFunctorCompAlpha2}, after resolving the definitions from above, becomes
    \begin{equation*}
        \begin{tikzcd}
        m\otimes((v\rhd \1_\SM)\otimes n) \arrow[dddd, shift left = 20, "\com{\eqref{a1a2}}" description, color = white]\arrow[dddd, "{\id\,\otimes\,\alp1}"'] & & (v\rhd\1_\SM)\otimes(m\otimes n) \arrow[dddd, "\alp1", bend left, shift left=12]\\
        & (m\otimes(v\rhd\1_\SM))\otimes n \arrow[lu, "a"] \arrow[d, "{\alp2\,\otimes\,\id}"']& ((v\rhd\1_\SM)\otimes m)\otimes n \arrow[d, "{\alp1\,\otimes\,\id}"'] \arrow[u, "a"]\arrow[d, shift left = 10, "\com{\eqref{a1a}}" description, color = white] \\
        & (v\rhd(m\otimes\1_\SM))\otimes n \arrow[d, "\alp1"']\arrow[rd, "\com{\nat \alp1}" description, color = white]& (v\rhd(\1_\SM\otimes m))\otimes n \arrow[l, "{(\id\,\rhd\,\beta)\,\otimes\,\id}"'] \arrow[d, "\alp1"']\\
        & v\rhd((m\otimes\1_\SM)\otimes n) \arrow[rd, shift right = 4, shorten <=1cm, shorten >=1cm, "\com{2}" description, color = white]\arrow[d, shift left = 4, "\com{1}" description, color = white]\arrow[rdd, "{\id\,\rhd\,(r\,\otimes\,\id)}", pos = 0.2, end anchor = north west] \arrow[d, "{\id\,\rhd\,a}"'] & v\rhd((\1_\SM\otimes m)\otimes n) \arrow[l, "{\id\,\rhd\,(\beta\,\otimes\,\id)}"'] \arrow[d, "{\id\,\rhd\,a}"] \arrow[dd, pos = 0.2, "{\id\,\rhd\,(l\,\otimes\,\id)}", bend right=70, shift right=5] \\
        m\otimes(v\rhd(\1_\SM\otimes n)) \arrow[r, "\alp2"] \arrow[d, "{\id\,\otimes\,(\id\,\rhd\,l)}"']& v\rhd(m\otimes(\1_\SM\otimes n)) \arrow[rd, shift right = 2, "{\id\,\rhd\,(\id\,\otimes\,l)}"']& v\rhd(\1_\SM\otimes (m\otimes n)) \arrow[d, "{\id\,\rhd\,l}"] \arrow[d, shift right = 7, "\com{\eqref{aL}}" description, color = white]\\
        m\otimes(v\rhd n) \arrow[rr, "{\alp2}"']\arrow[ru, "\com{\nat \alp2}" description, color = white] & & v\rhd(m\otimes n)
        \end{tikzcd}
    \end{equation*}

    which commutes, for all $v\in\SV$ and $m,n\in\SM$, due to the reasons indicated in boxes, where \fbox{1}~is \eqref{triangleA} and \fbox{2}~is \eqref{defDrinfeldUnitMorphism}.
\end{itemize}

Further, $G=\Id$ is obviously an equivalence of categories and hence it is an equivalence of module monoidal categories by Proposition \ref{propAlternativeEquivalenceUnitalModuleMonoidalCategories}, resulting in $G$ being an $\Id$\==equivalence of module monoidal categories.\bigskip

{\bfseries Part 2}

Given a pair $(\ST,F)$, applying the 2\==functor $\Psi$ from Section \ref{secEquivalenceConstructionPsi} yields a module monoidal category, denoted by $(\SM,\otimes,\1_\SM,a,l,r,\rhd,m,l^\rhd,\alp1,\alp2)$, and then applying Lemma \ref{lemUnitalModuleMonoidalCategoryToPair}, we end up with a pair $(\ST',F')$, where
\begin{itemize}
    \item $\ST=\ST'$ as monoidal categories,
    \item the braided central monoidal functor $F'\colon\SV\longrightarrow\ST$ is given by the lift $(F'^Z,\varphi^{F'}_2,\varphi^{F'}_0)$ with
    \begin{align*}
        F'^Z(v)&\overset{\ref{lemUnitalModuleMonoidalCategoryToPair}}{=}(v\rhd\1_\ST,\beta_{F'(v),-})\\
        &\overset{\ref{conPsi}}{=}F(v)\otimes\1_\ST,\\
        F'^Z\left(v\overset{f}{\longrightarrow}w\right)&\overset{\ref{lemUnitalModuleMonoidalCategoryToPair}}{=}\left[v\rhd\1_\ST\xrightarrow[]{f\,\rhd\,\id_{\1_\ST}}w\rhd\1_\ST\right]\\
        &\overset{\ref{conPsi}}{=}\left[F(v)\otimes\1_\ST\xrightarrow[]{F(f)\,\otimes\,\id_{\1_\ST}}F(w)\otimes\1_\ST\right],
    \end{align*}
    for all $v,w\in\SV$, $t\in\ST$ and $f\in\SV(v,w)$ where
    \begin{itemize}
        \item the half\--braiding $\beta_{F'(v),-}$ has, for all $v\in\SV$ and $t\in\ST$, components
        \[\beta_{F'(v),t}=a_{t,F(v),\1_\ST}\circ (\beta_{F(v),t}\otimes\id_{\1_\ST})\circ a_{F(v),t,\1_\ST}^{-1}\circ (\id_{F(v)}\otimes\beta_{\1_\ST,t})\circ a_{F(v),\1_\ST,t}\]
        since
        \begin{equation*}
        \begin{tikzcd}
        \begin{array}{l}F'(v)\otimes t\\=(v\rhd\1_\ST)\otimes t\\=(F(v)\otimes\1_\ST)\otimes t\end{array} \arrow[r, "{\beta_{F'(v),t}}"] \arrow[d, "{\begin{array}{l}\alp1\\\overset{\ref{conPsi}}{=}\,a\end{array}}"']& \begin{array}{l}t\otimes F'(v)\\=t\otimes(v\rhd\1_\ST)\\=t\otimes(F(v)\otimes\1_\ST)\end{array} \arrow[ldd, "\alp2"] \\
        \begin{array}{l}v\rhd(\1_\ST\otimes t)\\=F(v)\otimes(\1_\ST\otimes t)\end{array} \arrow[ru, "\com{1}" description, color = white] \arrow[d, "{\begin{array}{l}\id\,\rhd\,\beta_{\1_\ST,t}\\\overset{\ref{conPsi}}{=}\,\id\,\otimes\,\beta_{\1_\ST,t}\end{array}}"'] & (t\otimes F(v))\otimes\1_\ST \arrow[u, "a"'] \\
        \begin{array}{l}v\rhd(t\otimes\1_\ST)\\=F(v)\otimes(t\otimes\1_\ST)\end{array}\arrow[ru, "\com{2}" description, color = white]& (F(v)\otimes t)\otimes\1_\ST \arrow[u, "{\beta_{F(v),t}\,\otimes\,\id}"'] \arrow[l, "a"]
        \end{tikzcd}
        \end{equation*}
        commutes as \fbox{1} is the definition of $\beta_{F'(v),t}$ from Lemma \ref{lemUnitalModuleMonoidalCategoryToPair} and \fbox{2} is the definition of $\alp2$ from Construction \ref{conPsi},
        \item the natural isomorphism $\varphi_2^{F'}$ has, for all $v,w\in\SV$, components
        \begin{align*}
        \varphi_2^{F'}&=(\id_{F(v\otimes w)}\otimes r_{\1_\ST})\circ (\varphi_{2,v,w}^F\otimes\id_{\1_\ST\otimes\1_\ST})\circ a_{F(v),F(w),\1_\ST\otimes\1_\ST}^{-1}\\
        &\phantom{\,=\,}\circ(\id_{F(v)}\otimes a_{F(w),\1_\ST,\1_\ST})\circ(\id_{F(v)}\otimes(\beta_{F(w),\1_\ST}^{-1}\otimes\id_{\1_\ST}))\\
        &\phantom{\,=\,}\circ(\id_{F(v)}\otimes a_{\1_\ST,F(w),\1_\ST}^{-1})\circ a_{F(v),\1_\ST,F(w)\otimes\1_\ST}
        \end{align*}
        since
        \begin{equation*}
        \begin{tikzcd}
        \begin{array}{l}F'(v)\otimes F'(w)\\=(v\rhd\1_\ST)\otimes (w\rhd\1_\ST)\\=(F(v)\otimes\1_\ST)\otimes (F(w)\otimes\1_\ST)\end{array} \arrow[d, "{\begin{array}{l}\alp1\\\overset{\ref{conPsi}}{=}\,a\end{array}}"'] \arrow[r, "\varphi_2^{F'}"] & \begin{array}{l}F'(v\otimes w)\\=(v\otimes w)\rhd\1_\ST\\=F(v\otimes w)\otimes \1_\ST\end{array} \\
        \begin{array}{l}v\rhd(\1_\ST\otimes (w\rhd\1_\ST))\\=F(v)\otimes(\1_\ST\otimes(F(w)\otimes\1_\ST))\end{array} \arrow[rdd, "{\begin{array}{l}\id\,\rhd\,\alp2\\\overset{\ref{conPsi}}{=}\,\id\,\otimes\,\alp2\end{array}}"]\arrow[r, "\com{1}" description, color = white] & \begin{array}{l}(v\otimes w)\rhd(\1_\ST\otimes\1_\ST)\\=F(v\otimes w)\otimes(\1_\ST\otimes\1_\ST)\end{array} \arrow[u, "{\id\,\otimes\,r}"']\arrow[dd, "m"', bend right, shift right = 3] \\
        F(v)\otimes((\1_\ST\otimes F(w))\otimes \1_\ST) \arrow[d, shift left = 20, "\com{2}" description, color = white]\arrow[u, "{\id\,\otimes\,a}"]& X \arrow[d, "a"] \arrow[u, "{\varphi_2^F\,\otimes\,\id}"']\\
        F(v)\otimes((F(w)\otimes\1_\ST)\otimes\1_\ST) \arrow[u, "{\id\,\otimes\,(\beta_{F(w),\1_\ST}\,\otimes\,\id)}"] \arrow[r, "{\id\,\otimes\,a}"']& \begin{array}{l}v\rhd(w\rhd(\1_\ST\otimes\1_\ST))\\=F(v)\otimes(F(w)\otimes(\1_\ST\otimes\1_\ST))\end{array}  \arrow[uu, shift left = 5, "\com{3}" description, color = white]
        \end{tikzcd}
        \end{equation*}
        with $X\coloneqq(F(v)\otimes F(w))\otimes(\1_\ST\otimes\1_\ST)$ commutes as \fbox{1} is the definition of $\varphi_2^{F'}$ from Lemma \ref{lemUnitalModuleMonoidalCategoryToPair} and \fbox{2} \& \fbox{3} are the definitions of $\alp2$ and $m$ from Construction \ref{conPsi} and
        \item the isomorphism $\varphi_0^{F'}$ is given by
        \[\varphi_0^{F'}=(\varphi_0^F\otimes \id_{\1_\ST})\circ l_{\1_\ST}^{-1}\]
        since
        \begin{equation*}
        \begin{tikzcd}
        \1_\ST \arrow[rr, "\varphi_0^{F'}", bend left, shift left = 2, start anchor = east, end anchor = west] \arrow[rr, shift right = 10,shorten <= 1cm, shorten >= 1cm, "\com{2}" description, color = white]\arrow[rr, shift left = 3, shorten <= 1cm, shorten >= 1cm, "\com{1}" description, color = white]& & \begin{array}{l}F'(\1_\SV)\\=\1_\SV\rhd\1_\ST\\=F(\1_\SV)\otimes\1_\ST\end{array} \arrow[ll, "l_{\1_\ST}^{\rhd}"] \\
        & \1_\ST\otimes\1_\ST \arrow[ru, "{\varphi_0^F\,\otimes\,\id_{\1_\ST}}"', end anchor = west, shift right = 2, shorten <= 1ex] \arrow[lu, "l_{\1_\ST}", end anchor = east, shift left = 2, shorten <= 1ex] &
        \end{tikzcd}
        \end{equation*}
        commutes as \fbox{1} is the definition of $\varphi_0^{F'}$ from Lemma \ref{lemUnitalModuleMonoidalCategoryToPair} and \fbox{2} is the definition of $l^{\rhd}$ from Construction \ref{conPsi}.
    \end{itemize}
\end{itemize}
Now, we define a morphism of pairs $(G,\varphi_2^G,\varphi_0^G,\gamma)\colon (\ST,F)\longrightarrow (\ST',F')$ with
\[(G,\varphi_2^G,\varphi_0^G)=(\Id,\id,\id)\]
and
\begin{align*}
\gamma_v\colon \begin{array}{l} F'(v)\\=F(v)\otimes\1_\ST\end{array}&\longrightarrow \begin{array}{l}G(F(v))\\=F(v)\end{array}
\end{align*}
as
\[\gamma_v\coloneqq r_{F(v)}.\]
Indeed, $G$ is a morphism of pairs, since
\begin{itemize}
    \item $(G,\varphi^G_2,\varphi^G_0)=(\Id,\id,\id)\colon\ST\longrightarrow\ST'$ is obviously a monoidal functor and
    \item the action coherence $\gamma$ is a monoidal natural isomorphism since
    \begin{itemize}
        \item $\gamma$ is natural since $r$ is,
        \item the square \eqref{monNatTrafoPhi2}, after resolving the definitions from above, becomes
        \begin{equation*}
            {\hspace{-1cm}\footnotesize
            \begin{tikzcd}
            (F(v)\otimes\1_\ST)\otimes(F(w)\otimes\1_\ST) \arrow[d, "{\id\,\otimes\,r}"] \arrow[rd, "a"]\arrow[d, shift right = 10, "\com{\funct \otimes}" description, color = white]\arrow[dddddd, pos = 0.05, "{r\,\otimes\,r}"', bend right, shift right = 20] && F(v\otimes w)\otimes\1_\ST \arrow[dddddd, pos = 0.05,  "{r}", bend left, shift left=21]\arrow[dd, shift left = 10, "\com{\funct \otimes}" description, color = white] \\
            (F(v)\otimes\1_\ST)\otimes F(w)\arrow[r, "\com{\nat a}" description, color = white] \arrow[d, "a"] \arrow[d, shift right = 10, "\com{\eqref{triangleA}}" description, color = white]\arrow[ddddd, "{r\,\otimes\,\id}", bend right, shift right=18] & F(v)\otimes(\1_\ST\otimes(F(w)\otimes\1_\ST)) \arrow[d, shift right = 8, "\com{\eqref{aR}}" description, color = white]\arrow[ld, "{\id\,\otimes\,(\id\,\otimes r)}"] & \\
            F(v)\otimes(\1_\ST\otimes F(w)) \arrow[rdd, shorten >= 2cm, "\com{\nat r}" description, color = white]\arrow[dd, shift right = 10, "\com{1}" description, color = white]\arrow[dd, "{\id\,\otimes\,\beta_{\1_\ST,F(w)}}"] \arrow[dddd, "{\id\,\otimes\,l}", bend right, pos = 0.3, shift right=10] & F(v)\otimes((\1_\ST\otimes F(w))\otimes \1_\ST) \arrow[u, "{\id\,\otimes\,a}"'] \arrow[l, "{\id\,\otimes\,r}"] \arrow[d, "{\id\,\otimes(\beta_{\1_\ST,F(w)}\,\otimes\,\id)}"', bend right]& F(v\otimes w)\otimes(\1_\ST\otimes\1_\ST) \arrow[uu, "{\id\,\otimes\,r}"]\\
            & F(v)\otimes((F(w)\otimes\1_\ST)\otimes\1_\ST)\arrow[d, shift right = 10, "\com{\eqref{aR}}" description, color = white] \arrow[ld, "{\id\,\otimes\,r}"'] \arrow[d, "{\id\,\otimes\,a}"] \arrow[u, "{\id\,\otimes\,(\beta_{F(w),\1_\ST}\,\otimes\,\id)}"', bend right] & \\
            F(v)\otimes(F(w)\otimes\1_\ST) \arrow[dd, shift left = 20, shorten >= 1cm, shorten <=1cm, "\com{\eqref{aR}}" description, color = white]\arrow[dd, "{\id\,\otimes\,r}"]& F(v)\otimes(F(w)\otimes(\1_\ST\otimes\1_\ST)) \arrow[l, "{\id\,\otimes\,(\id\,\otimes\,r)}"'] & (F(v)\otimes F(w))\otimes (\1_\ST\otimes\1_\ST) \arrow[d, "{\id\,\otimes\,r}"] \arrow[uu, "{\varphi_2^F\,\otimes\,\id}"] \arrow[l, "a"]\arrow[d, shift right = 10, "\com{\nat a}" description, color = white]\\
            && (F(v)\otimes F(w))\otimes\1_\ST \arrow[lld, "r"] \arrow[uuuuu, "{\varphi_2^F\,\otimes\,\id}", bend right, shift right=20] \arrow[llu, "a"]\arrow[d, "\com{\nat r}" description, color = white] \\
            F(v)\otimes F(w) \arrow[rr, "\varphi_2^F"']& & F(v\otimes w)
        \end{tikzcd}}
        \end{equation*}
        which commutes, for all $v,w\in\SV$, due to the reasons indicated in boxes, where \fbox{1} is \eqref{defDrinfeldUnitMorphism},       
        \item the triangle \eqref{monNatTrafoPhi0} reduces directly to naturality of $r$ and, finally,
    \end{itemize}
    \item the action coherence $\gamma$ is compatible with the braidings on $Z(\ST)$ and $Z(\ST')$, since \eqref{actionCoherenceAndBraiding}, after resolving the definitions from above, becomes
        \begin{equation*}
            \begin{tikzcd}
            (F(v)\otimes\1_\ST)\otimes t \arrow[d, shift left = 10, "\com{\eqref{triangleA}}" description, color = white] \arrow[d, "a"'] \arrow[rr, "{r\,\otimes\,\id}"] & & F(v)\otimes t \arrow[d, shift right = 6, "\com{\eqref{aR}}" description, color = white]\arrow[ddd, "{\beta_{F(v),t}}", bend left, shift left=10] \\
            F(v)\otimes(\1_\ST\otimes t) \arrow[r, "{\id\,\otimes\,\beta_{\1_\ST,t}}"'] \arrow[rru, "{\id\,\otimes\,l}", bend left=2] \arrow[rr, shift left = 5, shorten <= 2cm, shorten >=2cm, "\com{1}" description, color = white]& F(v)\otimes (t\otimes\1_\ST) \arrow[ru, "{\id\,\otimes\,r}"']& (F(v)\otimes t)\otimes\1_\ST \arrow[u, "r"'] \arrow[d, "{\beta_{F(v),t}\,\otimes\,\id}"'] \arrow[l, "a"]\arrow[d, shift left = 8, "\com{\nat r}" description, color = white] \\
            & & (t\otimes F(v))\otimes\1_\ST \arrow[d, "r"'] \arrow[lld, "a"'] \\
            t\otimes (F(v)\otimes\1_\ST) \arrow[rr, "{\id\,\otimes\,r}"'] & & t\otimes F(v)  \arrow[u, shift left = 10, "\com{\eqref{aR}}" description, color = white]
            \end{tikzcd}
        \end{equation*}
        which commutes, for all $v\in\SV$ and $t\in\ST$, due to the reasons indicated in boxes, where \fbox{1} is \eqref{defDrinfeldUnitMorphism}.
\end{itemize}

Further, $G=\Id$ is obviously an equivalence of categories and hence it is an equivalence of pairs by Proposition \ref{propAlternativeEquivalencePairs}, resulting in $G$ being an $\Id$\==equivalence of pairs.
\end{proof}

\begin{Remark}\label{remOtherModuleStructure}
Note the condition of the equivalence theorem, that the functor in the pair has to be braided central monoidal. As mentioned in \cite[Section 3.2]{HPT15}, given a braided monoidal category $\SV$, a monoidal category $\ST$ and a monoidal functor $F\colon\SV\longrightarrow\ST$, there are a priori two different ways to endow $\ST$ with a $\SV$\==module structure. However, if $F$ is a braided central monoidal functor, these two module structures turn out to be equivalent.
\end{Remark}

\subsection{Equivalence on Hom--categories}\label{secEquivalenceHomCategories}
This section aims to show that the functors
\[\Psi_{(\SM,F),(\SM',F')}\colon \Pairs((\SM,F),(\SM',F'))\longrightarrow \UModMonCat(\Psi(\SM,F),\Psi(\SM',F'))\]
on Hom\==spaces contained in the 2\==functor $\Psi\colon\Pairs\longrightarrow\UModMonCat$ from Section \ref{secEquivalenceConstructionPsi} are
\begin{itemize}
    \item bijective on 1\==morphisms and
    \item bijective on 2\==morphisms, i.e.\ $\Psi_{(\SM,F),(\SM',F')}$ is fully faithful.
\end{itemize}

Fix pairs $(\SM,F)$ and $(\SM',F')$. First, we show injectivity on 1\==morphisms.

Let $(G,\varphi_2^G,\varphi_0^G,\gamma^G),(G',\varphi_2^{G'},\varphi_0^{G'},\gamma^{G'})\colon (\SM,F)\longrightarrow(\SM',F')$ be two morphisms of pairs and assume that they map to the same module monoidal functor under $\Psi$. The definition of $\Psi$ then implies
\begin{align*}
    (G,\varphi_2^G,\varphi_0^G,\varphi_2^G\circ(\gamma^G\otimes'\id))&=(G,\varphi_2^G,\varphi_0^G,s^{\Psi(G)})\\
    &=\Psi(G,\varphi_2^G,\varphi_0^G,\gamma^G)\\
    &=\Psi(G',\varphi_2^{G'},\varphi_0^{G'},\gamma^{G'})\\
    &=(G',\varphi_2^{G'},\varphi_0^{G'},\varphi_2^{G'}\circ(\gamma^{G'}\otimes'\id)),
\end{align*}
and hence $(G,\varphi_2^G,\varphi_0^G)=(G',\varphi_2^{G'},\varphi_0^{G'})$. With this and the fact that the target of $\gamma_v^{G'}$ is $G'F(v)=GF(v)$, we get
\[\varphi_{2,F(v),m}^G\circ(\gamma_v^G\otimes'\id_{G(m)}) = s_{v,m}^{\Psi(G)} = s_{v,m}^{\Psi(G')} = \varphi_{2,F(v),m}^{G'}\circ(\gamma_v^{G'}\otimes'\id_{G'(m)}) = \varphi_{2,F(v),m}^G\circ(\gamma_v^{G'}\otimes'\id_{G(m)})\]
for all $v\in\SV$ and $m\in\SM$ and hence, since $\varphi_2^G$ is an isomorphism, we have
\[\gamma_v^G\otimes'\id_{G(m)} = \gamma_v^{G'}\otimes'\id_{G(m)}\]
for all $v\in\SV$ and $m\in\SM$. Especially, for $m=\1_\SM$, we can compose with isomorphisms $\varphi_0^G$ and $r'$ as in
\[F'(v)\otimes'G(\1_\SM)\xrightarrow[]{\gamma^G\,\otimes'\,\id}GF(v)\otimes' G(\1_\SM)\xleftarrow[]{\id\,\otimes\,\varphi_0^G}GF(v)\otimes'\1_{\SM'}\xrightarrow[]{r'}GF(v),\]
which, by naturality of $\varphi_0^G$ and $r'$ is equal to the composition
\[F'(v)\otimes'G(\1_\SM)\xleftarrow[]{\id\,\otimes\,\varphi_0^G}F'(v)\otimes' \1_{\SM'}\xrightarrow[]{r'}F'(v)\xrightarrow[]{\gamma^G}GF(v).\]
Thus, we have, since $\varphi_0^G$ and $r'$ are isomorphisms, that
\begin{align*}
    &&\gamma_v^G\otimes'\id_{G(m)} &= \gamma_v^{G'}\otimes'\id_{G(m)} \\
    \Leftrightarrow&&\underbrace{r'_{GF(v)}\circ(\id_{GF(v)}\otimes (\varphi_0^G)^{-1})\circ(\gamma_v^G\otimes'\id_{G(m)})}_{=\,\gamma_v^G\,\circ \,r'_{F'(v)}\,\circ\,(\id_{F'(v)}\,\otimes\,(\varphi_0^G)^{-1})} &= \underbrace{r'_{GF(v)}\circ(\id_{GF(v)}\otimes (\varphi_0^G)^{-1})\circ(\gamma_v^{G'}\otimes'\id_{G(m)})}_{=\,\gamma_v^{G'}\,\circ\, r'_{F'(v)}\,\circ\,(\id_{F'(v)}\,\otimes\, (\varphi_0^G)^{-1})} \\
    \Leftrightarrow&&\gamma_v^G &= \gamma_v^{G'}
\end{align*}
for all $v\in\SV$.

Second, we show surjectivity on 1\==morphisms. Let $(G,\varphi_2^G,\varphi_0^G,s^G)\colon \Psi(\SM,F)\longrightarrow\Psi(\SM',F')$ be an arbitrary module monoidal functor. Define $(H,\varphi_2^H,\varphi_0^H,\gamma^H)\colon (\SM,F)\longrightarrow(\SM',F')$ as
\begin{itemize}
    \item $(H,\varphi_2^H,\varphi_0^H)\coloneqq(G,\varphi_2^G,\varphi_0^G)$ and
    \item define $\gamma^H_v\colon F'(v)\longrightarrow \begin{array}{l}HF(v)\\=GF(v)\end{array}$ as
    \[\gamma^H_v\coloneqq r'_{GF(v)}\circ (\id_{GF(v)}\otimes' \varphi_0^{-1})\circ \varphi_{2,F(v),\1_\SM}^{-1}\circ s^G_{v,\1_\SM}\circ (\id_{F'(v)}\otimes' \varphi_0)\circ r'^{-1}_{F'(v)},\]
    where $\varphi_2\coloneqq\varphi_2^G=\varphi_2^H$ and $\varphi_0\coloneqq\varphi_0^G=\varphi_0^H$.
\end{itemize}

We have that $(H,\varphi_2^H,\varphi_0^H,\gamma^H)\colon (\SM,F)\longrightarrow(\SM',F')$ is a morphism of pairs, since
\begin{itemize}
    \item $(H,\varphi_2^H,\varphi_0^H)=(G,\varphi_2^G,\varphi_0^G)$ is a monoidal functor,
    \item $\gamma^H$ is as a composition of natural isomorphisms again a natural isomorphism,
    \item the square \eqref{monNatTrafoPhi2} commutes; to prove that, observe first of all that functoriality of $\otimes'$ and the definition of $\gamma^H$ imply, for all $v,w\in\SV$, that $\gamma_v^H\otimes'\gamma_w^H$ is equal to the composition
    \[{\footnotesize\begin{tikzcd}
    F'(v)\otimes' F'(w)  & F'(v)\otimes'(F'(w)\otimes'\1_{\SM'}) \arrow[l, "{\id\,\otimes'\,r'}"'] \arrow[r, "{\id\,\otimes'\,(\id\,\otimes'\,\varphi_0)}"] & F'(v)\otimes'(F'(w)\otimes'G(\1_\SM)) \arrow[d, "{\id\,\otimes'\,s^G_{w,\1_\SM}}"] \\
    & (F'(v)\otimes'\1_{\SM'})\otimes'G(F(w)\otimes\1_\SM) \arrow[d, "{(\id\,\otimes'\,\varphi_0)\,\otimes'\,\id}"']\arrow[r, "{r'\,\otimes'\,\id}"]  & F'(v)\otimes'G(F(w)\otimes\1_\SM) \arrow[dd, "s^G_{v,F(w)\otimes\1_\SM}", dotted] \\
    &(F'(v)\otimes'G(\1_\SM))\otimes'G(F(w)\otimes\1_\SM) \arrow[d, "{s^G_{v,\1_\SM}\,\otimes'\,\id}"'] &\\
    &G(F(v)\otimes\1_\SM)\otimes'G(F(w)\otimes\1_\SM)\arrow[r, color = white, "\com{*}" description] & G(F(v)\otimes(F(w)\otimes\1_\SM))\\
    &{(GF(v)\otimes'\,G(\1_\SM))\otimes'G(F(w)\otimes\1_\SM)} \arrow[u, "{\varphi_2\,\otimes'\,\id}"]&\\
     & {(GF(v)\otimes'\,\1_{\SM'})\otimes'G(F(w)\otimes\1_\SM)} \arrow[u, "{(\id\,\otimes'\,\varphi_0)\,\otimes'\,\id}"'] \arrow[r, "{r'\,\otimes'\,\id}"] & GF(v)\otimes'G(F(w)\otimes\1_\SM) \arrow[uu, "\varphi_2"', dotted]\\
    (GF(v)\otimes'GF(w)) & GF(v)\otimes'(GF(w)\otimes'\1_{\SM'}) \arrow[l, "{\id\,\otimes'\,r'}"'] \arrow[r, "{\id\,\otimes'\,(\id\,\otimes'\,\varphi_0)}"] & GF(v)\otimes'(GF(w)\otimes'G(\1_\SM)) \arrow[u, "{\id\,\otimes'\,\varphi_2}"']
    \end{tikzcd}}\]
    where \fbox{$*$} commutes, since (with $X\coloneqq F(w)\otimes\1_\SM$)
    \[{\begin{tikzcd}
    (F'(v)\otimes'\1_{\SM'})\otimes'G(X) \arrow[r, "a'"] \arrow[rr, bend left=14, "{r'\,\otimes'\,\id}"] \arrow[d, "{(\id\,\otimes'\,\varphi_0)\,\otimes'\,\id}"'] \arrow[rr, "\com{\eqref{triangleA}}" description, shift left=6, shorten >=2cm, shorten <=2cm, color = white] \arrow[rd, "\com{\nat a'}" description, color = white] & F'(v)\otimes'(\1_{\SM'}\otimes'G(X)) \arrow[d, "{\id\,\otimes'\,(\varphi_0\,\otimes'\,\id)}"]\arrow[r, "{\id\,\otimes'\,l'}"] \arrow[rd, "\com{\eqref{phi0l}}" description, color = white] & F'(v)\otimes'G(X) \arrow[ddd, bend left, shift left = 13, "{s^G_{v,X}}"']\arrow[d, shift left = 8, "\com{\nat s^G}" description, color = white] \\
    (F'(v)\otimes'G(\1_\SM))\otimes'G(X) \arrow[r, "a'"] \arrow[dd, "{s^G_{v,\1_\SM}\,\otimes'\,\id}"'] & F'(v)\otimes'(G(\1_\SM)\otimes'G(X))\arrow[r, "{\id\,\otimes'\,\varphi_2}"] \arrow[dd, "\com{\eqref{diagModuleMonoidalFunctorCompAlpha1}}" description, color = white]& F'(v)\otimes'G(\1_\SM\otimes X)\arrow[d, "{s^G_{v,\1_\SM\otimes X}}"']\arrow[u, "{\id\,\otimes'\,G(l)}"]\\
    &   & G(F(v)\otimes(\1_\SM\otimes X)) \arrow[d, "{G(\id\,\otimes\,l)}"']\\
    G(F(v)\otimes\1_\SM)\otimes'G(X) \arrow[d, "\com{\eqref{phi0r}}" description, shift left = 8, color = white]\arrow[r, "\varphi_2"] \arrow[rrd, "{G(r)\,\otimes'\,\id}"] &G((F(v)\otimes\1_\SM)\otimes X) \arrow[r, "{G(r\,\otimes\,\id)}"']\arrow[r, "\com{\eqref{triangleA}}" description, shift left = 4, color = white] \arrow[ru, "G(a)"]  & G(F(v)\otimes X)\arrow[d, "\com{\nat \varphi_2}" description, shift right = 8, color = white]\\
    {(GF(v)\otimes'\,G(\1_\SM))\otimes'G(X)} \arrow[u, "{\varphi_2\,\otimes'\,\id}"] & {(GF(v)\otimes'\,\1_{\SM'})\otimes'G(X)} \arrow[l, "{(\id\,\otimes'\,\varphi_0)\,\otimes'\,\id}"] \arrow[r, "{r'\,\otimes'\,\id}"']& GF(v)\otimes'G(X) \arrow[u, "\varphi_2"']
    \end{tikzcd}}\]
    commutes, for all $v,w\in\SV$, due to the reasons indicated in boxes; by which the commutativity of \eqref{monNatTrafoPhi2} follows, for all $v,w\in\SV$, from the commutativity of
    \[{\footnotesize
    \begin{tikzcd}
    F'(v)\otimes' F'(w) \arrow[rr, "\varphi_2^{F'}"] \arrow[d, "\com{\eqref{aR}}" description, shift left=9, color = white] \arrow[rr, "\com{\nat r'}" description, shift right=5, shorten <=2cm, color = white]& & F'(v\otimes w) \\
    F'(v)\otimes'(F'(w)\otimes'\1_{\SM'}) \arrow[u, "{\id\,\otimes'\,r'}"] \arrow[d, "{\id\,\otimes'\,(\id\,\otimes'\,\varphi_0)}"'] \arrow[rd, "\com{\nat a'}" description, color = white] & (F'(v)\otimes'F'(w))\otimes'\1_{\SM'} \arrow[lu, "r'"'] \arrow[d, "{\id\,\otimes'\,\varphi_0}"] \arrow[l, "a'"] \arrow[r, "{\varphi_2^{F'}\,\otimes'\,\id}"] \arrow[rd, "\com{\funct \otimes'}" description, color = white] & F'(v\otimes w)\otimes'\1_{\SM'} \arrow[u, "r'"'] \arrow[d, "{\id\,\otimes'\,\varphi_0}"] \\
    F'(v)\otimes'(F'(w)\otimes'G(\1_\SM)) \arrow[d, "{\id\,\otimes'\,s^G_{w,\1_\SM}}"'] & (F'(v)\otimes'F'(w))\otimes'G(\1_\SM) \arrow[l, "a'"] \arrow[r, "{\varphi_2^{F'}\,\otimes'\,\id}"'] \arrow[d, "\com{\eqref{sPentagon}}" description, color = white] & F'(v\otimes w)\otimes' G(\1_\SM) \arrow[ddd, "{s^G_{v\otimes w,\1_\SM}}"]\\
    F'(v)\otimes'G(F(w)\otimes\1_\SM) \arrow[d,dashed]\arrow[d, shift left = 10, color = white, "\com{*}" description]\arrow[r, "{s^G_{v,F(w)\otimes\1_\SM}}"]& G(F(v)\otimes(F(w)\otimes\1_\SM)) &\\
    GF(v)\otimes'G(F(w)\otimes\1_\SM) \arrow[ru, "\varphi_2"]\arrow[rd, "\com{\eqref{phi2Hexagon}}" description, color = white] & &\\
    GF(v)\otimes'(GF(w)\otimes'G(\1_\SM)) \arrow[u, "{\id\,\otimes'\,\varphi_2}"]  \arrow[d, "\com{\nat a'}" description, shift right=10, color = white]& G((F(v)\otimes F(w))\otimes\1_\SM) \arrow[uu, "G(a)"] \arrow[r, "{G(\varphi_2^F\,\otimes\,\id)}"] \arrow[r, shift right = 6, "\com{\nat \varphi_2}" description, color = white]& G(F(v\otimes w)\otimes\1_\SM)\\
    (GF(v)\otimes'GF(w))\otimes'G(\1_\SM) \arrow[u, "a'"'] \arrow[r, "{\varphi_2\,\otimes'\,\id}"] \arrow[r, shift right = 6, "\com{\funct \otimes'}" description, color = white] & G(F(v)\otimes F(w))\otimes'G(\1_\SM) \arrow[u, "\varphi_2"] \arrow[r, "{G(\varphi_2^F)\,\otimes'\,\id}"] \arrow[r, shift right = 6, "\com{\funct \otimes'}" description, color = white] & GF(v\otimes w))\otimes'G(\1_\SM) \arrow[u, "\varphi_2"'] \\
    (GF(v)\otimes'GF(w))\otimes'\1_{\SM'} \arrow[u, "{\id\,\otimes'\,\varphi_0}"'] \arrow[r, "{\varphi_2\,\otimes'\,\id}"] \arrow[dd, "r'", bend left, shift left=20] \arrow[d, "a'"]  \arrow[d, "\com{\eqref{aR}}" description, shift left=10, color = white] & G(F(v)\otimes F(w))\otimes'\1_{\SM'} \arrow[dd, "\com{\nat r'}" description, shift right = 20, color = white]\arrow[u, "{\id\,\otimes'\,\varphi_0}"] \arrow[r, "{G(\varphi_2^F)\,\otimes'\,\id}"] \arrow[dd, "r'"] \arrow[rdd, "\com{\nat r'}" description, color = white]& GF(v\otimes w))\otimes'\1_{\SM'} \arrow[u, "{\id\,\otimes'\,\varphi_0}"'] \arrow[dd, "r'"] \\
    GF(v)\otimes'(GF(w)\otimes'\1_{\SM'}) \arrow[d, "{\id\,\otimes'\,r'}"] \arrow[uuu, "{\id\,\otimes'\,(\id\,\otimes'\,\varphi_0)}"', bend left, shift left=21]& &\\
    (GF(v)\otimes'GF(w)) \arrow[r, "\varphi_2"']& G(F(v)\otimes F(w)) \arrow[r, "G(\varphi_2^F)"']& GF(v\otimes w)) 
    \end{tikzcd}}\]
    due to the reasons indicated in boxes,
    \item plugging in the definition of $\gamma^H$, the triangle \eqref{monNatTrafoPhi0} becomes 
    \begin{equation*}
        \begin{tikzcd}
\1_{\SM'} \arrow[rr, "\varphi_0^{F'}"] \arrow[dddddd, shift right = 10,  "\com{\nat l'}" description, color = white]\arrow[dddddd, shift right = 8, "\varphi_0"', bend right]\arrow[rrd, pos=0.4, "\com{\nat r'}" description, color = white] && F'(\1_\SV)\\
\1_{\SM'}\otimes\1_{\SM'} \arrow[u, "l'", "{=\,r'}"'] \arrow[d, "{\id\,\otimes'\,\varphi_0}"] \arrow[rr, "{\varphi_0^{F'}\,\otimes'\,\id}"]\arrow[rrd, pos = 0.4, "\com{\funct \otimes'}" description, color = white] && F'(\1_\SV)\otimes'\1_{\SM'} \arrow[u, "r'"'] \arrow[d, "{\id\,\otimes\,\varphi_0}"] \\
\1_{\SM'}\otimes'G(\1_\SM) \arrow[rr, "{\varphi_0^{F'}\,\otimes'\,\id}"] \arrow[dddd, "l'"]\arrow[rd,  "\com{\eqref{sTriangle}}" description, color = white]&& F'(\1_\SV)\otimes'G(\1_\SM) \arrow[d, "{s^G_{\1_\SV,\1_\SM}}"]\\
& G(\1_\SM\otimes\1_\SM) \arrow[rd,  pos=0.4, "\com{\nat \varphi_2}" description, color = white]\arrow[r, "{G(\varphi_0^F\,\otimes\,\id)}"] \arrow[lddd, "G(l)"', "{=\,G(r)}", start anchor = west] & G(F(\1_\SV)\otimes\1_\SM) \\
& G(\1_\SM)\otimes'G(\1_\SM) \arrow[d, shift right = 8,  "\com{\eqref{phi0r}}" description, color = white]\arrow[rd, pos = 0.4, "\com{\funct \otimes'}" description, color = white]\arrow[u, "\varphi_2"] \arrow[r, "{G(\varphi_0^F)\,\otimes'\,\id}"] & GF(\1_\SV)\otimes'G(\1_\SM) \arrow[u, "\varphi_2"']\\
& G(\1_\SM)\otimes' \1_{\SM'} \arrow[rd, pos = 0.4, "\com{\nat r'}" description, color = white]\arrow[r, "{G(\varphi_0^F)\,\otimes'\,\id}"] \arrow[u, "{\id\,\otimes\,\varphi_0}"'] \arrow[ld, "r'"] & GF(\1_\SV)\otimes'\1_{\SM'} \arrow[u, "{\id\,\otimes\,\varphi_0}"'] \arrow[d, "r'"]\\
G(\1_\SM) \arrow[rr, "G(\varphi_0^F)"']&& GF(\1_\SV) 
\end{tikzcd}
    \end{equation*}
    which commutes due to the reasons indicated in boxes and
    \item the hexagon \eqref{actionCoherenceAndBraiding} splits into
    \begin{equation*}
        \varphi_{2,F(v),m}^G\circ (\gamma_v^H\otimes' \id_{G(m)})\overset{(I)}{=}s^G_{v,m}\overset{(II)}{=} G(\beta_{F(v),m}^{Z(\SM)})^{-1}\circ \varphi_{2,m,F(v)}^G\circ (\id_{G(m)}\otimes' \gamma_v^H)\circ \beta_{F'(v),G(m)}^{Z(\SM')}
    \end{equation*}
    where (I) holds by the commutativity of
    \begin{equation*}
        {\footnotesize
\begin{tikzcd}
& & F'(v)\otimes' G(m) \arrow[ddddddd, "{s^G_{v,m}}", bend left, shift left=13]\\
(F'(v)\otimes'\1_{\SM'})\otimes'G(m) \arrow[rru, "{r'\,\otimes'\,\id}", bend left=8] \arrow[d, "{(\id\,\otimes'\,\varphi_0)\,\otimes'\,\id}"'] \arrow[r, "a'"] \arrow[rd, "\com{\nat a'}" description, color = white]\arrow[rru, "\com{\eqref{triangleA}}" description, color = white, shorten >= 2cm, shorten <= 2cm] & F'(v)\otimes'(\1_{\SM'}\otimes' G(m)) \arrow[ru, "{\id\,\otimes'\,l'}"] \arrow[d, "{\id\,\otimes'\,(\varphi_0\,\otimes'\,\id)}"] \arrow[rd, "\com{\eqref{phi0l}}" description, color = white] & \\
(F'(v)\otimes'G(\1_{\SM}))\otimes'G(m) \arrow[d, "{s^G_{v,\1_\SM}\,\otimes'\,\id}"'] \arrow[r, "a'"] & F'(v)\otimes' (G(\1_{\SM})\otimes'G(m)) \arrow[r, "{\id\,\otimes'\,\varphi_2}"] \arrow[d, "\com{\eqref{diagModuleMonoidalFunctorCompAlpha1}}" description, color = white] & F'(v)\otimes'G(\1_\SM\otimes m) \arrow[uu, "{\id\,\otimes'\,G(l)}"] \arrow[dddd, "{s^G_{v,\1_\SM\otimes m}}"'] \arrow[dddd, "\com{\nat s^G}" description, shift left=10, color = white] \\
G(F(v)\otimes\1_\SM)\otimes'G(m) \arrow[r, "\varphi_2"] & G((F(v)\otimes\1_\SM)\otimes m) \arrow[rddd, "G(a)"] \arrow[dd, "\com{\eqref{phi2Hexagon}}" description, color = white]& \\
(GF(v)\otimes'G(\1_\SM))\otimes' G(m) \arrow[u, "{\varphi_2\,\otimes'\,\id}"] \arrow[rd, "a'"]& & \\
(GF(v)\otimes'\1_{\SM'})\otimes' G(m)  \arrow[r, "\com{\nat a'}" description, color = white] \arrow[d, "\com{\eqref{triangleA}}" description, shift right = 12, color = white]\arrow[dd, "r'\,\otimes'\,\id"',  bend right, shift right = 20] \arrow[u, "{(\id\,\otimes'\,\varphi_0)\,\otimes'\,\id}"] \arrow[d, "a'"'] &  GF(v)\otimes'(G(\1_\SM)\otimes' G(m)) \arrow[d, "{\id\,\otimes'\,\varphi_2}"] & \\
GF(v)\otimes'(\1_{\SM'}\otimes'G(m)) \arrow[r, "\com{\eqref{phi0l}}" description, color = white]\arrow[ru, "\id\,\otimes'\,(\varphi_0\,\otimes'\,\id)"]\arrow[d, "\id\,\otimes'\,l'"']&  GF(v)\otimes'G(\1_\SM\otimes m) \arrow[r, "\varphi_2"] \arrow[ld, "{\id\,\otimes'\,G(l)}"] \arrow[rd, "\com{\nat \varphi_2}" description, color = white] &  G(F(v)\otimes (\1_\SM\otimes m)) \arrow[d, "{G(\id\,\otimes\,l)}"']\\
GF(v)\otimes'G(m)\arrow[rr, "\varphi_2"']&&G(F(v)\otimes m)
\end{tikzcd}}
    \end{equation*}
    for all $v\in\SV$ and $m\in\SM$, and (II) holds by the commutativity of
    \begin{equation*}
        {\footnotesize
\begin{tikzcd}
G(m)\otimes'(F'(v)\otimes'\1_{\SM'}) \arrow[r, "{\id\,\otimes'\,r'}"] \arrow[ddd, "{\id\,\otimes'\,(\id\,\otimes'\,\varphi_0)}", bend right, shift right=19] & G(m)\otimes'F'(v) \arrow[d, "\com{\nat r'}" description, color = white]& F'(v)\otimes' G(m) \arrow[l, "{\beta^{Z(\SM')}}"'] \arrow[dddddddddd, "{s_{v,m}}"', bend left, shift left=19] \arrow[d, "\com{\eqref{aR}}" description, shift right=9, color = white]\\
(G(m)\otimes'F'(v))\otimes'\1_{\SM'} \arrow[ru, "r'"] \arrow[u, "a'"] \arrow[d, "{\id\,\otimes'\,\varphi_0}"] \arrow[u, "\com{\eqref{aR}}" description, shift right=9, color = white]  \arrow[r, "\com{\funct \otimes'}" description, shift right = 9, shorten <= 1cm, shorten >=1cm, color = white] & (F'(v)\otimes'G(m))\otimes'\1_{\SM'} \arrow[ru, "r'"] \arrow[r, "a'"] \arrow[l, "{\beta^{Z(\SM')}\,\otimes'\,\id}"] \arrow[d, "{\id\,\otimes'\,\varphi_0}"] \arrow[rd, "\com{\nat a'}" description, color = white] & F'(v)\otimes'(G(m)\otimes'\1_{\SM'}) \arrow[u, "{\id\,\otimes'\,r'}"'] \arrow[d, "{\id\,\otimes'\,(\id\,\otimes'\,\varphi_0)}"'] \\
(G(m)\otimes'F'(v))\otimes'G(\1_\SM) \arrow[d, "a'"] \arrow[d, "\com{\nat a'}" description, shift right=10, color = white]\arrow[rrddd, "\com{\eqref{diagModuleMonoidalFunctorCompAlpha2}}" description, color = white] & (F'(v)\otimes'G(m))\otimes'G(\1_\SM) \arrow[l, "{\beta^{Z(\SM')}\,\otimes'\,\id}"] \arrow[r, "a'"']& F'(v)\otimes'(G(m)\otimes'G(\1_\SM)) \arrow[d, "{\id\,\otimes'\,\varphi_2}"']\arrow[d, "\com{\eqref{phi0r}}" description, shift left=10, color = white]\\
G(m)\otimes'(F'(v)\otimes'G(\1_\SM)) \arrow[d, "{\id\,\otimes'\,s_{v,\1_\SM}}"']&& F'(v)\otimes'G(m\otimes\1_\SM) \arrow[uuu, "{\id\,\otimes'\,G(r)}", bend right, shift right=18] \arrow[d, "{s_{v,m\otimes\1_\SM}}"'] \arrow[d, "\com{\nat s}" description, shift left=10, color = white] \\
G(m)\otimes'G(F(v)\otimes\1_\SM) \arrow[d, "\varphi_2"]&& G(F(v)\otimes(m\otimes\1_\SM)) \arrow[dddddd, "{G(\id\,\otimes\,r)}", bend left, pos = 0.1, shift left=17] \arrow[d, "\com{\eqref{aR}}" description, shift left=10, color = white] \\
G(m\otimes(F(v)\otimes\1_\SM)) & G((m\otimes F(v))\otimes\1_\SM) \arrow[rd, "\com{\nat \varphi_2}" description, color = white] \arrow[l, "G(a)"']& G((F(v)\otimes m)\otimes\1_\SM) \arrow[u, "G(a)"] \arrow[l, "{G(\beta^{Z(\SM)}\,\otimes\,\id)}"'] \arrow[ddddd, "G(r)"', bend left, shift left=15] \\
G(m)\otimes'(GF(v)\otimes'G(\1_\SM)) \arrow[uu, "{\id\,\otimes'\,\varphi_2}"', bend left, shift left=15, pos = 0.9] \arrow[u, "\com{\eqref{phi2Hexagon}}" description, color = white]& G(m\otimes F(v))\otimes'G(\1_\SM) \arrow[u, "\varphi_2"] \arrow[dd, shift right=25, shorten <=1cm, shorten >=1cm, "\com{\funct\otimes'}" description, color = white]\arrow[rdd, "\com{\funct \otimes'}" description, color = white] & G(F(v)\otimes m)\otimes'G(\1_\SM) \arrow[l, "{G(\beta^{Z(\SM)})\,\otimes'\,\id}"] \arrow[u, "\varphi_2"] \arrow[dd, "\com{\eqref{phi0r}}" description, shift left=10, color = white]\\
(G(m)\otimes'GF(v))\otimes'G(\1_\SM) \arrow[u, "a'"'] \arrow[u, "\com{\nat a'}" description, shift left=10, color = white] \arrow[ru, "{\varphi_2\,\otimes'\,\id}"]&& \\
(G(m)\otimes'GF(v))\otimes'\1_{\SM'} \arrow[u, "{\id\,\otimes'\,\varphi_0}"'] \arrow[d, "a'"'] \arrow[dd, "r'", bend left, shift left=20] \arrow[r, "{\varphi_2\,\otimes'\,\id}"]  \arrow[d, "\com{\eqref{aR}}" description, shift left=10, color = white]  & G(m\otimes F(v))\otimes'\1_{\SM'} \arrow[dd, "r'"] \arrow[uu, "{\id\,\otimes'\,\varphi_0}"] \arrow[rdd, "\com{\nat r'}" description, color = white] \arrow[dd, "\com{\nat r'}" description, shift right=20, color = white] & G(F(v)\otimes m)\otimes'\1_{\SM'} \arrow[l, "{G(\beta^{Z(\SM)})\,\otimes'\,\id}"] \arrow[dd, "r'"] \arrow[uu, "{\id\,\otimes'\,\varphi_0}"]\\
G(m)\otimes'(GF(v)\otimes'\1_{\SM'}) \arrow[uuu, "{\id\,\otimes'\,(\id\,\otimes'\,\varphi_0)}"', bend left, shift left=20] \arrow[d, "{\id\,\otimes'\,r'}"']&& \\
G(m)\otimes'GF(v) \arrow[r, "\varphi_2"']& G(m\otimes F(v)) & G(F(v)\otimes m) \arrow[l, "{G(\beta^{Z(\SM)})}"] 
\end{tikzcd}}
    \end{equation*}
    for all $v\in\SV$ and $m\in\SM$.
\end{itemize}

If we consider the image of $H$ under $\Psi_{(\SM,F),(\SM',F')}$, we have that the equation $s^{\Psi(H)}=s^G$ is satisfied, since 
\begin{equation*}
    {\footnotesize
\begin{tikzcd}
F'(v)\otimes'G(m) \arrow[rr, "{\gamma^H\,\otimes'\,\id}"]&& GF(v)\otimes'G(m) \arrow[ddddd, "\varphi_2", bend left, shift left=21] \arrow[d, "\com{\eqref{phi0r}}" description, shift left=10, color = white]\\
F'(v)\otimes'(\1_{\SM'}\otimes G(m)) \arrow[u, "{\id\,\otimes\,l'}"] \arrow[u, "\com{\eqref{triangleA}}" description, shift right=9, color = white] \arrow[dd, "{\id\,\otimes'\,(\varphi_0\,\otimes'\,\id)}"] \arrow[rdd, "\com{\nat a'}" description, color = white] \arrow[dd, "\com{\eqref{phi0l}}" description, shift right=10, color = white] & (F'(v)\otimes'\1_{\SM'})\otimes G(m) \arrow[l, "a'"] \arrow[lu, "{r'\,\otimes'\,\id}"] \arrow[dd, "{(\id\,\otimes'\,\varphi_0)\,\otimes'\,\id}"] \arrow[r, "\com{\text{def. }\gamma^H}" description, color = white] & (GF(v)\otimes'\1_{\SM'})\otimes'G(m) \arrow[u, "{r'\,\otimes'\,\id}"] \arrow[d, "{(\id\,\otimes'\,\varphi_0)\,\otimes'\,\id}"']\\
 && (GF(v)\otimes'G(\1_\SM))\otimes' G(m) \arrow[d, "{\varphi_2\,\otimes'\,\id}"'] \\
F'(v)\otimes'(G(\1_\SM)\otimes'G(m)) \arrow[dd, "{\id\,\otimes'\,\varphi_2}"]  & (F'(v)\otimes'G(\1_\SM))\otimes' G(m) \arrow[dd, "\com{\eqref{diagModuleMonoidalFunctorCompAlpha1}}" description, color = white] \arrow[r, "{s_{v,\1_\SM}\,\otimes'\,\id}"] \arrow[l, "a'"]& G(F(v)\otimes \1_\SM)\otimes' G(m) \arrow[d, "\varphi_2"'] \arrow[uuu, "{G(r)\,\otimes'\,\id}", bend right, shift right=20] \arrow[d, "\com{\nat \varphi_2}" description, color = white, shift left = 10] \\
&  & G((F(v)\otimes \1_\SM)\otimes m) \arrow[ld, "G(a)"'] \arrow[d, "{G(r\,\otimes\,\id)}"] \arrow[d, "\com{\eqref{triangleA}}" description, shift right=10, color = white]\\
F'(v)\otimes'G(\1_\SM\otimes m) \arrow[r, "{s_{v,\1_\SM\otimes m}}"'] \arrow[uuuuu, "{\id\,\otimes'\,G(l)}", bend left, shift left=18]& G(F(v)\otimes(\1_\SM\otimes m)) \arrow[r, "{G(\id\,\otimes\,l)}"']& G(F(v)\otimes m) 
\end{tikzcd}}
\end{equation*}
commutes, for all $v\in\SV$ and $m\in\SM$, due to the reasons indicated in boxes, where moving along the top and right border we have the definition of $s^{\Psi(H)}_{v,m}$ and moving along the left and bottom border we have, using naturality of $s$, the morphism $s^G_{v,m}$. It follows that
\begin{align*}
    \Psi(H,\varphi_2^H,\varphi_0^H,\gamma^H)&=(G,\varphi_2^G,\varphi_0^G,s^{\Psi(H)})\\
    &=(G,\varphi_2^G,\varphi_0^G,s^G),
\end{align*}
which proves the surjectivity of $\Psi_{(\SM,F),(\SM',F')}$ on 1\==morphisms.

Next, we show that $\Psi_{(\SM,F),(\SM',F')}$ is fully faithful. Recall that, for $\eta\colon G\Longrightarrow H$ a 2\==morphism of pairs between morphisms of pairs $G,H\colon (\SM,F)\longrightarrow (\SM',F')$, we defined $\Psi(\eta)\coloneqq \eta$. This obviously makes $\Psi_{(\SM,F),(\SM',F')}$ faithful. To prove fullness, let $\eta\colon \Psi(G)\Longrightarrow\Psi(H)$ be a module monoidal natural transformation. We have to show that $\eta\colon G\Longrightarrow H$ is a 2\==morphism of pairs. Since $\Psi$ restricted to the monoidal part is the identity, we have that $\eta\colon G\Longrightarrow H$ is a monoidal natural transformation. Further, the diagram \eqref{condition2MorphismOfPairs} reads
\begin{equation*}
    \begin{tikzcd}
F'(v) \arrow[r, shift right = 5, "\com{\nat r'}" description, color = white]\arrow[rrddddd, "\gamma^H"', shift right = 12, bend right=70] \arrow[rr, "\gamma^G", bend left=15]& GF(v)\otimes'\1_{\SM'} \arrow[r, "r'"]\arrow[r, "\com{\eqref{phi0r}}" description, shift right = 5, color = white]\arrow[d, "\id\,\otimes'\,\varphi_0^G"]& GF(v) \arrow[ddddd, shift left = 10, shorten <=2cm, shorten >= 2cm, "\com{\nat \eta}" description, color = white]\arrow[ddddd, "\eta", bend left, shift left=8] \\
& GF(v)\otimes'G(\1_\SM) \arrow[rddd, "\com{\eqref{monNatTrafoS}}" description, color = white, shorten <=2cm, shorten >=2cm]\arrow[r, "\varphi_2^G"] & G(F(v)\otimes\1_\SM) \arrow[u, "G(r)"] \arrow[ddd, "\eta"'] \\
F'(v)\otimes'\1_{\SM'} \arrow[uu, "r'"] \arrow[ruu, "{\gamma^G\,\otimes'\,\id}"] \arrow[r, "{\id\,\otimes'\,\varphi_0^G}"] \arrow[rd, end anchor = north west, "{\id\,\otimes'\,\varphi_0^H}"'] \arrow[rddd, "{\gamma^H\,\otimes'\,\id}", bend right] & F'(v)\otimes'G(\1_\SM) \arrow[u, pos = 0.6, "\com{\funct \otimes'}" description, shift left = 10, color = white]\arrow[d, pos = 0.23, "\com{\eqref{monNatTrafoPhi0}}" description, shift right = 11, color = white]\arrow[d, "{\id\,\otimes'\,\eta}"] \arrow[u, "{\gamma^G\,\otimes'\,\id}"'] & \\
& F'(v)\otimes'H(\1_\SM) \arrow[d, "\com{\funct \otimes'}" description, shift right = 10, color = white]\arrow[d, "{\gamma^H\,\otimes'\,\id}"] & \\
& HF(v)\otimes'H(\1_\SM) \arrow[r, "\com{\eqref{phi0r}}" description, color = white, shift right = 6]\arrow[r, "\varphi_2^H"] & H(F(v)\otimes\1_\SM) \arrow[d, "H(r)"'] \\
& HF(v)\otimes'\1_{\SM'} \arrow[r, "r'"]\arrow[r, "\com{\nat r'}" description, color = white, shift right = 5] \arrow[u, "{\id\,\otimes'\,\varphi_0^H}"'] & HF(v)
\end{tikzcd}
\end{equation*}
which commutes, for all $v\in\SV$, due to the reasons indicated in boxes. Thus, $\eta\colon G\Longrightarrow H$ is a 2\==morphism of pairs.

\appendix
\section{Proof of Lemma \ref{lemUnitalModuleMonoidalCategoryToPair}}\label{secAppendixProofLemmaUnitalModuleMonoidalCategoryToPair}

We have to check that the definition from the lemma defines a braided central monoidal functor $(F,F^Z)\colon\SV\longrightarrow\ST$. This is indeed the case, since
\begin{itemize}
    \item $F^Z$ is well\--defined, i.e.
    \begin{itemize}
        \item $(F(v),\beta_{F(v),-})$ is an element in $Z(\ST)$, i.e.\ $F(v)\in\ST$ (which is clear by definition) and $\beta_{F(v),-}$ satisfies \eqref{defDrinfeldHalfBraiding}, which, after resolving the definitions from the lemma, becomes
        \begin{equation*}
        \begin{tikzcd}
((v\rhd\1_\ST)\otimes t)\otimes t' \arrow[dd, "{\alp1\,\otimes\,\id}"']\arrow[dd, shift left = 15, "\com{\eqref{a1a}}" description, color = white]  \arrow[r, "a"] & (v\rhd\1_\ST)\otimes (t\otimes t')  \arrow[d, "\alp1"]&& (t\otimes t')\otimes (v\rhd\1_\ST) \arrow[dddddd, "a", bend left, shift left=10]\arrow[dddddd, "\com{\eqref{a2a}}" description, shorten <= 3cm, shorten >= 3cm, pos = 0.4, shift left = 10, color = white] \arrow[d, "\alp2"'] \\
& v\rhd(\1_\ST\otimes(t\otimes t')) \arrow[rr, shift right = 15, "\com{\eqref{defDrinfeldHalfBraiding}}" description, color = white]\arrow[rr, "{\id\,\rhd\,\beta}"] && v\rhd ((t\otimes t')\otimes \1_\ST) \arrow[ddd, "{\id\,\rhd\,a}"']\\
(v\rhd(\1_\ST\otimes t))\otimes t' \arrow[rd, "\com{\nat \alp1}" description, color = white]\arrow[d, "{(\id\,\rhd\,\beta)\,\otimes\,\id}"'] \arrow[r, "\alp1"] & v\rhd ((\1_\ST\otimes t)\otimes t') \arrow[d, "{\id\,\rhd\,(\beta\,\otimes\,\id)}"] \arrow[u, "{\id\,\rhd\,a}"'] && \\
(v\rhd(t\otimes\1_\ST))\otimes t' \arrow[rddd, shorten >= 2cm, "\com{\eqref{a1a2}}" description, color = white]\arrow[r, "\alp1"]& v\rhd ((t\otimes \1_\ST)\otimes t') \arrow[d, "{\id\,\rhd\,a}"]&& \\
& v\rhd (t\otimes (\1_\ST\otimes t')) \arrow[rrd, "\com{\nat\alp2}" description, color = white]\arrow[rr, "{\id\,\rhd\,(\id\,\otimes\,\beta)}"]&& v\rhd (t\otimes (t'\otimes \1_\ST)) \\
& t\otimes (v\rhd(\1_\ST\otimes t')) \arrow[u, "\alp2"'] \arrow[rr, "{\id\,\otimes\,(\id\,\rhd\,\beta)}"']&& t\otimes(v\rhd(t'\otimes \1_\ST)) \arrow[u, "\alp2"]\\
(t\otimes(v\rhd\1_\ST))\otimes t' \arrow[uuu, "{\alp2\,\otimes\,\id}"] \arrow[r, "a"']& t\otimes ((v\rhd\1_\ST)\otimes t')  \arrow[u, "{\id\,\otimes\,\alp1}"'] &&t\otimes (t'\otimes (v\rhd \1_\ST)) \arrow[u, "{\id\,\otimes\,\alp2}"] 
\end{tikzcd}
        \end{equation*}
        which commutes, for all $v\in\SV$ and $t,t'\in\ST$, due to the reasons indicated in boxes,
        \item for a morphism $f\in\SV(v,w)$, we have that $F^Z(f)$ is a morphism in $Z(\ST)$, i.e.\ $F(f)=f\rhd \id_{\1_\ST}$ is a morphism in $\ST(v\rhd\1_\ST,w\rhd\1_\ST)$ (which is clear by the functoriality of the action $\rhd$) and, for all $t\in\ST$, we have that Equation \eqref{defDrinfeldMorphism} is satisfied, which, after resolving the definitions from the lemma, becomes
        \begin{equation*}
        \begin{tikzcd}
(v\rhd \1_\ST)\otimes t  \arrow[r, "\alp1"]\arrow[d, "{(f\,\rhd\,\id)\,\otimes\,\id}"']\arrow[rd,"\com{\nat\alp1}" description, color = white] & v\rhd(\1_\ST\otimes t) \arrow[rd, "\com{\funct\rhd}" description, color = white]\arrow[d, "{f\,\rhd\,\id}"] \arrow[r, "{\id\,\rhd\,\beta}"]& v\rhd(t\otimes\1_\ST) \arrow[d, "{f\,\rhd\,\id}"']\arrow[rd, "\com{\nat\alp2}" description, color = white] & t\otimes (v\rhd\1_\ST)\arrow[l, "\alp2"']\arrow[d, "{\id\,\otimes\,(f\,\rhd\,\id)}"]\\
(w\rhd\1_\ST)\otimes t \arrow[r, "\alp1"']&w\rhd(\1_\ST\otimes t) \arrow[r, "{\id\,\rhd\,\beta}"']& w\rhd(t\otimes\1_\ST) & t\otimes (w\rhd\1_\ST) \arrow[l, "\alp2"]
\end{tikzcd}
        \end{equation*}
        which commutes, for all $v,w\in\SV$ and $t\in\ST$, due to the reasons indicated in boxes,
    \end{itemize}
    \item $F^Z$ is a functor, i.e.\ compatible with composition of morphisms, which follows from the functoriality of the action $\rhd$, since, for all composable morphisms $u\overset{f}{\longrightarrow}v\overset{g}{\longrightarrow}w$, we have
    \[F^Z(g\circ f)=(g\circ f)\rhd\id_{\1_\ST}=(g\rhd \id_{\1_\ST})\circ (f\rhd\id_{\1_\ST})=F^Z(g)\circ F^Z(f),\]
    \item the functor $F^Z$ is monoidal, i.e.
    \begin{itemize}
        \item $\varphi_2$ is a morphism in $Z(\ST)$, i.e.\ Equation \eqref{defDrinfeldMorphism} is satisfied, since
        \begin{equation*}
        \begin{tikzcd}[column sep = huge]
        \begin{array}{l}(F(v)\otimes F(w))\otimes t\\=((v\rhd\1)\otimes(w\rhd\1))\otimes t\end{array} \arrow[r, "{\beta_{F(v)\otimes F(w),t}}"] \arrow[d, "{\varphi_2\,\otimes\,\id}"'] \arrow[rd, "\Delta"]\arrow[d, "\com{\eqref{pf1}}" description, color = white, shift left = 10] & \begin{array}{l}t\otimes (F(v)\otimes F(w))\\=t\otimes ((v\rhd\1)\otimes(w\rhd\1))\end{array} \arrow[d, "\com{\eqref{pf2}}" description, color = white, shift right = 10] \arrow[d, "{\id\,\otimes\,\varphi_2}"] \\
        \begin{array}{l}F(v\otimes w)\otimes t\\=((v\otimes w)\rhd\1)\otimes t\end{array} \arrow[r, "{\beta_{F(v\otimes w),t}}"']& \begin{array}{l}t\otimes F(v\otimes w)\\=t\otimes((v\otimes w)\rhd\1)\end{array}
        \end{tikzcd}
        \end{equation*}
        commutes, for all $v,w\in\SV$ and $t\in\ST$, since \fbox{\eqref{pf1}} commutes due to Diagram \eqref{pf1} and \fbox{\eqref{pf2}} commutes due to Diagram \eqref{pf2}, namely, after resolving the definitions from the lemma and with abbreviations $\1\coloneqq\1_\ST$, $f_1\coloneqq \id\rhd(\id\otimes(\id\rhd\beta_{\1,t}))$ and $f_2\coloneqq\id\rhd(\id\rhd(\id\otimes\beta_{\1,t}))$, the diagram
        \begin{equation}\label{pf1}
        {\footnotesize
        \begin{tikzcd}[column sep = huge]
        ((v\rhd\1)\otimes (w\rhd\1))\otimes t \arrow[r, "a"]\arrow[rdddd, shorten >= 3cm, "\com{\eqref{a1a}}" description, color = white] \arrow[dddd, "{\alp1\,\otimes\,\id}"'] & (v\rhd\1)\otimes((w\rhd\1)\otimes t) \arrow[r, "{\id\,\otimes\,\alp1}"]\arrow[ddd, shorten >= 3cm, shift left = 20, "\com{\nat\alp1}" description, color = white] \arrow[ddd, "\alp1"]& (v\rhd\1)\otimes(w\rhd(\1\otimes t)) \arrow[d, "{\id\,\otimes\,(\id\,\rhd\,\beta)}"] \arrow[ddd, "\alp1"', shift right = 15, bend right] \\
        & & (v\rhd\1)\otimes(w\rhd(t\otimes\1)) \arrow[d,, shift right = 10,  "\com{\nat\alp1}" description, color = white]\arrow[d, "\alp1"]\\
        & & v\rhd(\1\otimes(w\rhd(t\otimes\1))) \arrow[ddd, "{\id\,\rhd\,\alp2}"', bend left, shift left = 17]\arrow[d, shift left = 10, "\com{\nat\alp2}" description, color = white]\\
        & v\rhd(\1\otimes((w\rhd\1)\otimes t)) \arrow[r, "{\id\,\rhd\,(\id\,\otimes\,\alp1)}"]& v\rhd(\1\otimes(w\rhd(\1\otimes t))) \arrow[d, "{\id\,\rhd\,\alp2}"'] \arrow[u, "{f_1}"]\\
        (v\rhd(\1\otimes(w\rhd\1)))\otimes t \arrow[rdd, "\com{\nat\alp1}" description, color = white]\arrow[dd, "{(\id\,\rhd\,\alp2)\,\otimes\,\id}"'] \arrow[r, "\alp1"]& v\rhd((\1\otimes(w\rhd\1))\otimes t) \arrow[ru, "\com{\eqref{a1a2}}" description, color = white]\arrow[dd, "{\id\,\rhd\,(\alp2\,\otimes\,\id)}"] \arrow[u, "{\id\,\rhd\,a}"]& v\rhd(w\rhd(\1\otimes(\1\otimes t))) \arrow[d, "{f_2}"']\\
        & & v\rhd(w\rhd(\1\otimes(t\otimes\1))) \\
        (v\rhd(w\rhd(\1\otimes\1)))\otimes t \arrow[rddd, shorten >= 3cm, "\com{\eqref{a1m}}" description, color = white] \arrow[r, "\alp1"]& v\rhd((w\rhd(\1\otimes\1))\otimes t) \arrow[dd, "{\id\,\rhd\,\alp1}"] & v\rhd(w\rhd((\1\otimes t)\otimes t)) \arrow[u, "{\id\,\rhd\,(\id\,\rhd\,a)}"'] \arrow[d, "{\id\,\rhd\,(\id\,\rhd\,(\beta\,\otimes\,\id))}"] \\
        & & v\rhd(w\rhd((t\otimes\1)\otimes\1)) \arrow[d, "{\id\,\rhd\,(\id\,\rhd\,a)}"]\arrow[d, "\com{\eqref{defDrinfeldTensor}}" description, color = white, shift right = 15]\\
        & v\rhd(w\rhd((\1\otimes\1)\otimes t))  \arrow[rd, "\com{\nat m}" description, color = white]\arrow[r, "{\id\,\rhd\,(\id\,\rhd\,\beta)}"] \arrow[ruuuu, end anchor = south west, "{\id\,\rhd\,(\id\,\rhd\,a)}"', bend right=20, pos = 0.4] & v\rhd(w\rhd(t\otimes(\1\otimes\1))) \\
        ((v\otimes w)\rhd (\1\otimes\1))\otimes t \arrow[uuu, "{m\,\otimes\,\id}"]\arrow[rd, "\com{\nat\alp1}" description, color = white]\arrow[dd, "{(\id\,\rhd\,r)\,\otimes\,\id}"'] \arrow[r, "\alp1"]& (v\otimes w)\rhd((\1\otimes\1)\otimes t) \arrow[u, "m"']\arrow[rd, "\com{\nat\beta}" description, color = white]\arrow[d, "{\id\,\rhd\,(r\,\otimes\,\id)}"] \arrow[r, "{\id\,\rhd\,\beta}"]& (v\otimes w)\rhd(t\otimes(\1\otimes\1)) \arrow[d, "{\id\,\rhd\,(\id\,\otimes\,r)}"]\arrow[u, "m"'] \\
        & (v\otimes w)\rhd (\1\otimes t) \arrow[r, "{\id\,\rhd\,\beta}"'] & (v\otimes w)\rhd(t\otimes \1) \\
        ((v\otimes w)\rhd\1)\otimes t \arrow[ru, "\alp1"']& & t\otimes ((v\otimes w)\rhd \1) \arrow[u, "\alp2"'] 
        \end{tikzcd}}
        \end{equation}
        commutes, for all $v,w\in\SV$ and $t\in\ST$, due to the reasons indicated in boxes, and, after resolving the definitions from the lemma and with abbreviation $\1\coloneqq\1_\ST$, the diagram
        \begin{equation}\label{pf2}
        \hspace{-1cm}
        {\footnotesize
        \begin{tikzcd}
        \begin{array}{l}(F(v)\otimes F(w))\otimes t\\=((v\rhd\1)\otimes(w\rhd\1))\otimes t\end{array}\arrow[rr,shift right = 20, "\com{\eqref{defDrinfeldTensor}}" description, color = white]\arrow[rr,"{\beta}"]\arrow[d,"a"']&&\begin{array}{l}t\otimes(F(v)\otimes F(w))\\=t\otimes((v\rhd\1)\otimes(w\rhd\1))\end{array}\arrow[dddddddddd,"{\id\,\otimes\,\alp1}"]\\
        \begin{array}{l}F(v)\otimes(F(w)\otimes t)\\=(v\rhd\1)\otimes((w\rhd\1)\otimes t)\end{array}\arrow[d,"{\id\,\otimes\,\alp1}"']\arrow[rdd,"{\id\,\otimes\,\beta}"]&&\\
        (v\rhd\1)\otimes(w\rhd(\1\otimes t))\arrow[rd, "\com{\dfn \beta}" description, color = white]\arrow[d,"{\id\,\otimes\,(\id\,\rhd\,\beta)}"']&&\\
        (v\rhd\1)\otimes(w\rhd(t\otimes\1))\arrow[dd, "\com{\nat\alp1}" description, color = white]\arrow[ddd,"\alp1"',bend right, shift right = 15]&\begin{array}{l}F(v)\otimes(t\otimes F(w))\\=(v\rhd\1)\otimes(t\otimes(w\rhd\1))\end{array}\arrow[l,"{\id\,\otimes\,\alp2}"]\arrow[ldd,"\alp1", bend right=10]&\\
        &\begin{array}{l}(F(v)\otimes t)\otimes F(w)\\=((v\rhd\1)\otimes t)\otimes(w\rhd\1)\end{array}\arrow[d,"{\beta\,\otimes\,\id}"]\arrow[u,"a"']\arrow[ddd,"{\alp1\,\otimes\,\id}",bend right, shift right = 20]&\\
        v\rhd(\1\otimes(t\otimes(w\rhd\1)))\arrow[ru, "\com{\eqref{a1a}}" description, color = white]\arrow[d,"{\id\,\rhd\,(\id\,\otimes\,\alp2)}"]&\begin{array}{l}(t\otimes F(v))\otimes F(w)\\=(t\otimes(v\rhd\1))\otimes(w\rhd\1)\end{array}\arrow[d,"{\alp2\,\otimes\,\id}"]\arrow[ruuuuu,"a", start anchor = north east]\arrow[d, "\com{\dfn\beta}" description, shift right = 8, color = white]\arrow[d, "\com{\eqref{a1a2}}" description, color = white, shift left = 40]&\\
        v\rhd(\1\otimes(w\rhd(t\otimes\1)))\arrow[d, "\com{\eqref{a2a}}" description, color = white, shift left = 11]\arrow[d,"{\id\,\rhd\,\alp2}"']&(v\rhd(t\otimes\1))\otimes(w\rhd\1)\arrow[ddd,"\alp1",bend left, shift left = 20]&\\
        v\rhd(w\rhd(\1\otimes(t\otimes\1)))&(v\rhd(\1\otimes t))\otimes(w\rhd\1)\arrow[d, "\com{\nat\alp1}" description, shift left = 17, color = white]\arrow[u,"{(\id\,\rhd\,\beta)\,\otimes\,\id}"']\arrow[d,"\alp1"]&\\
        v\rhd(w\rhd((\1\otimes t)\otimes t))\arrow[rd, "\com{\nat\alp2}" description, color = white]\arrow[u,"{\id\,\rhd\,(\id\,\rhd\,a)}"]\arrow[d,"{\id\,\rhd\,(\id\,\rhd\,(\beta\,\otimes\,\id))}"']&v\rhd((\1\otimes t)\otimes(w\rhd\1))\arrow[d,"{\id\,\rhd\,(\beta\,\otimes\,\id)}"]\arrow[l,"{\id\,\rhd\,\alp2}"]\arrow[luuu,start anchor = west, shorten <= 0.2cm, end anchor = south east, "{\id\,\rhd\,a}"]&\\
        v\rhd(w\rhd((t\otimes\1)\otimes\1))\arrow[rdd, "\com{\eqref{a2a}}" description, color = white]\arrow[dd,"{\id\,\rhd\,(\id\,\rhd\,a)}"']&v\rhd((t\otimes\1)\otimes(w\rhd\1))\arrow[d,"{\id\,\rhd\,a}"]\arrow[l,"{\id\,\rhd\,\alp2}"]&\\
        &v\rhd(t\otimes(\1\otimes(w\rhd\1)))\arrow[rd, "\com{\nat\alp2}" description, color = white]\arrow[d,"{\id\,\rhd\,(\id\,\otimes\,\alp2)}"]&t\otimes(v\rhd(\1\otimes(w\rhd\1)))\arrow[d,"{\id\,\otimes\,(\id\,\rhd\,\alp2)}"]\arrow[l,"\alp2"']\\
        v\rhd(w\rhd(t\otimes(\1\otimes\1)))\arrow[rrd, shorten <= 3cm, "\com{\eqref{a2m}}" description, color = white]&v\rhd(t\otimes(w\rhd(\1\otimes\1)))\arrow[l,"{\id\,\rhd\,\alp2}"]&t\otimes(v\rhd(w\rhd(\1\otimes\1)))\arrow[l,"\alp2"]\\
        (v\otimes w)\rhd(t\otimes(\1\otimes\1))\arrow[u,"m"]\arrow[rrd, "\com{\nat\alp2}" description, color = white]\arrow[d,"{\id\,\rhd\,(\id\,\otimes\,r)}"']&&t\otimes((v\otimes w)\rhd(\1\otimes\1))\arrow[d,"{\id\,\otimes\,(\id\,\rhd\,r)}"]\arrow[ll,"\alp2"]\arrow[u,"{\id\,\otimes\,m}"']\\
        (v\otimes w)\rhd(t\otimes\1)&&t\otimes((v\otimes w)\rhd\1)\arrow[ll,"\alp2"]
        \end{tikzcd}}
        \end{equation}
        commutes, for all $v,w\in\SV$ and $t\in\ST$, due to the reasons indicated in boxes,
        \item $\varphi_2$ satisfies the hexagon \eqref{phi2Hexagon}, since
        \begin{equation*}
        \begin{tikzcd}
        \begin{array}{l}(F(u)\otimes F(v))\otimes F(w)\\=((u\rhd\1)\otimes (v\rhd\1))\otimes (w\rhd\1)\end{array} \arrow[r, "a"] \arrow[d, "{\varphi_2\,\otimes\,\id}"'] \arrow[rdd, "\Delta"] & \begin{array}{l}F(u)\otimes(F(v)\otimes F(w))\\=(u\rhd\1)\otimes((v\rhd\1)\otimes(w\rhd\1))\end{array} \arrow[d, shift right = 10, "\com{\eqref{pf4}}" description, color = white] \arrow[d, "{\id\,\otimes\,\varphi_2}"] \\
        \begin{array}{l}F(u\otimes v)\otimes F(w)\\=((u\otimes v)\rhd\1)\otimes (w\rhd\1)\end{array} \arrow[d, shift left = 10, "\com{\eqref{pf3}}" description, color = white] \arrow[d, "\varphi_2"'] & \begin{array}{l}F(u)\otimes F(v\otimes w)\\=(u\rhd\1)\otimes((v\otimes w)\rhd\1)\end{array} \arrow[d, "\varphi_2"] \\
        \begin{array}{l}F((u\otimes v)\otimes w)\\=((u\otimes v)\otimes w)\rhd\1\end{array} \arrow[r, "{\begin{array}{l}F(a)\\=\,a\,\rhd\,\id\end{array}}"']&\begin{array}{l}F(u\otimes(v\otimes w))\\=(u\otimes(v\otimes w))\rhd\1\end{array} 
        \end{tikzcd}
        \end{equation*}
        commutes, for all $v,w\in\SV$ and $t\in\ST$, since \fbox{\eqref{pf3}} commutes due to Diagram \eqref{pf3} and \fbox{\eqref{pf4}} commutes due to Diagram \eqref{pf4}, namely, after resolving the definitions from the lemma and with abbreviation $\1\coloneqq\1_\ST$, the diagram
        \begin{equation}\label{pf3}
        {\footnotesize
        \hspace{-1cm}
\begin{tikzcd}
((u\rhd\1)\otimes (v\rhd\1))\otimes (w\rhd\1) \arrow[r, "{\alp1\,\otimes\,\id}"] & (u\rhd(\1\otimes(v\rhd\1)))\otimes (w\rhd\1) \arrow[rd, "\com{\nat\alp1}" description, color = white]\arrow[ld, "{(\id\,\rhd\,\alp2)\,\otimes\,\id}"'] \arrow[r, "\alp1"] & u\rhd((\1\otimes(v\rhd\1))\otimes(w\rhd\1)) \arrow[d, "{\id\,\rhd\,(\alp2\,\otimes\,\id)}"] \\
(u\rhd (v\rhd (\1\otimes\1)))\otimes (w\rhd\1) \arrow[rr, shift right = 5, pos = 0.6, shorten <= 4cm, shorten >= 2cm, "\com{\nat\alp1}" description, color = white] \arrow[rd, "{(\id\,\rhd\,(\id\,\rhd\,r))\,\otimes\,\id}"] \arrow[rr, "\alp1"]&& u\rhd((v\rhd(\1\otimes\1))\otimes (w\rhd\1)) \arrow[ddd, "{\id\,\rhd\,\alp1}"] \arrow[ldd, "{\id\,\rhd\,((\id\,\rhd\,r)\,\otimes\,\id)}"] \\
((u\otimes v)\rhd(\1\otimes\1))\otimes (w\rhd\1) \arrow[u, "{m\,\otimes\,\id}"] \arrow[r, "\com{\nat m}" description, color = white]\arrow[d, "{(\id\,\rhd\,r)\,\otimes\,\id}"'] & (u\rhd(v\rhd\1))\otimes(w\rhd\1) \arrow[d, "\alp1"] & \\
((u\otimes v)\rhd\1)\otimes (w\rhd\1) \arrow[rdddd, "\com{\eqref{a1m}}" description, color = white]\arrow[dddd, "\alp1"'] \arrow[ru, "{m\,\otimes\,\id}"'] & u\rhd((v\rhd\1)\otimes(w\rhd\1)) \arrow[dddd, "{\id\,\rhd\,\alp1}"']\arrow[rd, "\com{\nat\alp1}" description, color = white]& \\
 && u\rhd(v\rhd((\1\otimes\1)\otimes (w\rhd\1))) \arrow[d, shift right = 10, "\com{\eqref{triangleA}}" description, color = white]\arrow[d, "{\id\,\rhd\,(\id\,\rhd\,a)}"] \arrow[lddd, start anchor = west, pos = 0.2, "{\id\,\rhd\,(\id\,\rhd\,(r\,\otimes\,\id))}"'] \\
 && u\rhd(v\rhd(\1\otimes(\1\otimes(w\rhd\1)))) \arrow[ldd,start anchor= west, "{\id\,\rhd\,(\id\,\rhd\,(\id\,\otimes\,l))}", pos = 0.7]\arrow[d, shift right = 10, "\com{\nat\alp2}" description, color = white]\\
 && u\rhd(\1\otimes(v\rhd(\1\otimes(w\rhd\1)))) \arrow[u, "{\id\,\rhd\,\alp2}"'] \arrow[d, "{\id\,\rhd\,(\id\,\otimes\,(\id\,\rhd\,l))}"]\\
(u\otimes v)\rhd(\1\otimes (w\rhd\1)) \arrow[rdd, "\com{\nat m}" description, color = white]\arrow[dd, "{\id\,\rhd\,\alp2}"'] \arrow[r, "m"'] & u\rhd(v\rhd(\1\otimes(w\rhd\1))) \arrow[rdd, "\com{\eqref{a2m}}" description, color = white]\arrow[dd, "{\id\,\rhd\,(\id\,\rhd\,\alp2)}"] & u\rhd(\1\otimes(v\rhd(w\rhd\1)))  \arrow[l, "{\id\,\rhd\,\alp2}"]\\
 && u\rhd(\1\otimes((v\otimes w)\rhd\1)) \arrow[d, "{\id\,\rhd\,\alp2}"]\arrow[u, "{\id\,\rhd\,(\id\,\otimes\,m)}"']\\
(u\otimes v)\rhd(w\rhd (\1\otimes\1)) \arrow[rrd, "\com{\eqref{pentM}}" description, color = white]\arrow[r, "m"'] & u\rhd(v\rhd(w\rhd(\1\otimes\1)))  & u\rhd((v\otimes w)\rhd(\1\otimes\1)) \arrow[l, "{\id\,\rhd\,m}"] \\
((u\otimes v)\otimes w)\rhd(\1\otimes\1) \arrow[u, "m"] \arrow[rrd, "\com{\funct \rhd}" description, color = white]\arrow[d, "{\id\,\rhd\,r}"'] \arrow[rr, "{a\,\rhd\,\id}"]&& (u\otimes (v\otimes w))\rhd(\1\otimes\1) \arrow[d, "{\id\,\rhd\,r}"]\arrow[u, "m"']\\
((u\otimes v)\otimes w)\rhd\1 \arrow[rr, "{a\,\rhd\,\id}"']&& (u\otimes(v\otimes w))\rhd\1
\end{tikzcd}}
        \end{equation}
        commutes, for all $u,v,w\in\SV$, due to the reasons indicated in boxes, and, after resolving the definitions from the lemma and with abbreviations $\1\coloneqq\1_\ST$ and
        \begin{align*}f_1&\coloneqq\id \rhd (\id \otimes (\id \rhd (\id \rhd r_\1)))\\
        &=\id \rhd (\id \otimes (\id \rhd (\id \rhd l_\1)),
        \end{align*}
        the diagram
        \begin{equation}\label{pf4}
        {\footnotesize
        \hspace{-1cm}
        \begin{tikzcd}
        ((u\rhd\1)\otimes (v\rhd\1))\otimes (w\rhd\1) \arrow[d, "{\alp1\,\otimes\,\id}"'] \arrow[rr, "a"]\arrow[d, shift left = 20, "\com{\eqref{a1a}}" description, color = white] & & (u\rhd\1)\otimes((v\rhd\1)\otimes(w\rhd\1)) \arrow[dddddd, "{\id\,\otimes\,\alp1}"] \arrow[ldd, "\alp1"] \\
        (u\rhd(\1\otimes(v\rhd\1)))\otimes (w\rhd\1) \arrow[d, "\alp1"']& & \\
        u\rhd((\1\otimes(v\rhd\1))\otimes(w\rhd\1)) \arrow[d, "{\id\,\rhd\,(\alp2\,\otimes\,\id)}"'] \arrow[r, "{\id\,\rhd\,a}"]& u\rhd(\1\otimes((v\rhd\1)\otimes(w\rhd\1))) \arrow[rdddd, "\com{\nat\alp1}" description, color = white]\arrow[ldddd, shorten >= 0.5cm, end anchor = east, "{\id\,\rhd\,(\id\,\otimes\,\alp1)}"] & \\
        u\rhd((v\rhd(\1\otimes\1))\otimes (w\rhd\1))\arrow[d, shift left = 30, "\com{\eqref{a1a2}}" description, color = white] \arrow[d, "{\id\,\rhd\,\alp1}"'] & & \\
        u\rhd(v\rhd((\1\otimes\1)\otimes (w\rhd\1))) \arrow[d, "{\id\,\rhd\,(\id\,\rhd\,a)}"']& & \\
        u\rhd(v\rhd(\1\otimes(\1\otimes(w\rhd\1)))) & & \\
        u\rhd(\1\otimes(v\rhd(\1\otimes(w\rhd\1)))) \arrow[d, shift left = 10, "\com{\eqref{a2l}}" description, color = white]\arrow[u, "{\id\,\rhd\,\alp2}"] \arrow[d, "{\id\,\rhd\,(\id\,\otimes\,(\id\,\rhd\,l))}"'] \arrow[rd, "{\id\,\rhd\,(\id\,\otimes\,(\id\,\rhd\,\alp2))}"] & & (u\rhd\1)\otimes(v\rhd(\1\otimes(w\rhd\1))) \arrow[d, "{\id\,\otimes\,(\id\,\rhd\,\alp2)}"] \arrow[ll, "\alp1"'] \\
        u\rhd(\1\otimes(v\rhd(w\rhd\1))) & u\rhd(\1\otimes(v\rhd(w\rhd(\1\otimes\1)))) \arrow[d, "\com{\nat\alp1}" description, color = white]\arrow[ru, "\com{\nat\alp1}" description, color = white]\arrow[l, "{f_1}"'] & (u\rhd\1)\otimes(v\rhd(w\rhd(\1\otimes\1)))  \arrow[l, "\alp1"] \arrow[ld, "{\id\,\otimes\,(\id\,\rhd\,(\id\,\rhd\,r))}"] \\
        & (u\rhd\1)\otimes(v\rhd(w\rhd(\1)) \arrow[lu, "\alp1"] \arrow[r, "\com{\nat m}" description, color = white] & (u\rhd\1)\otimes((v\otimes w)\rhd(\1\otimes\1)) \arrow[d, "{\id\,\otimes\,(\id\,\rhd\,r)}"]\arrow[u, "{\id\,\otimes\,m}"'] \\
        u\rhd(\1\otimes((v\otimes w)\rhd\1)) \arrow[uu, "{\id\,\rhd\,(\id\,\otimes\,m)}"]\arrow[ru, "\com{\nat\alp1}" description, color = white]\arrow[d, "{\id\,\rhd\,\alp2}"'] & & (u\rhd\1)\otimes((v\rhd\1)\otimes(w\rhd\1))\arrow[lu, "{\id\,\otimes\,m}"'] \arrow[ll, "\alp1"] \\
        u\rhd((v\otimes w)\rhd(\1\otimes\1)) & (u\otimes (v\otimes w))\rhd(\1\otimes\1) \arrow[l, "m"]\arrow[r, "{\id\,\rhd\,r}"'] & (u\otimes(v\otimes w))\rhd\1
        \end{tikzcd}}
        \end{equation}
        commutes, for all $u,v,w\in\SV$, due to the reasons indicated in boxes,
        \item $\varphi_0$ is a morphism in $Z(\ST)$, i.e.\ Equation \eqref{defDrinfeldMorphism} is satisfied, which, after resolving the definitions from the lemma, becomes
        \begin{equation*}
        \begin{tikzcd}
        &\1_\ST\otimes t \arrow[d,, shift right=10, "\com{\eqref{a1lm}}" description, color = white]\arrow[rd, "\com{\nat l^\rhd}" description, color = white]\arrow[r, "{\beta}"]  & t\otimes\1_\ST \arrow[d, shift left = 10, "\com{\eqref{a2lm}}" description, color = white] &\\
        (\1_\SV\rhd\1_\ST)\otimes t \arrow[r, "\alp1"'] \arrow[ru, "{l^\rhd\,\otimes\,\id}", bend left=15]&\1_\SV\rhd(\1_\ST\otimes t) \arrow[u, "l^\rhd"'] \arrow[r, "{\id\,\rhd\,\beta}"']& \1_\SV\rhd(t\otimes\1_\ST) \arrow[u, "l^\rhd"] & t\otimes(\1_\SV\rhd\1_\ST) \arrow[l, "\alp2"] \arrow[lu, "{\id\,\otimes\,l^\rhd}"', bend right=15]
        \end{tikzcd}
        \end{equation*}
        which commutes, for all $t\in\ST$, due to the reasons indicated in boxes,
        \item $\varphi_0$ satisfies the triangle \eqref{phi0l}, which, after resolving the definitions from the lemma, becomes
        \begin{equation*}
        \begin{tikzcd}
\1_\ST\otimes (v\rhd\1_\ST) \arrow[rr, shift left = 6, shorten >= 2cm, shorten <= 2cm, "\com{\eqref{a2l}}" description, color = white]\arrow[d, shift right = 10, "\com{\eqref{a1lm}}" description, color = white]\arrow[rd, "\com{\nat l^\rhd}" description, color = white] \arrow[rr, "l", bend left=15, shift left = 1] \arrow[r, "\alp2"] & v\rhd (\1_\ST\otimes\1_\ST) \arrow[d, shift left = 7, "\com{\eqref{diagTri'}}" description, color = white]\arrow[r, "{\id\,\rhd\,l}"] & v\rhd\1_\ST\arrow[d, "\com{\funct \rhd}" description, color = white]\\
\1_\SV\rhd(\1_\ST\otimes(v\rhd\1_\ST)) \arrow[r, "{\id\,\rhd\,\alp2}"'] \arrow[u, "l^\rhd"'] & \1_\SV\rhd(v\rhd(\1_\ST\otimes\1_\ST)) \arrow[r, "m"'] \arrow[u, "l^\rhd"] & (\1_\SV\otimes v)\rhd(\1_\ST\otimes\1_\ST) \arrow[d, "{\begin{array}{l}\id\,\rhd\,r\\=\,\id\,\rhd\,l\end{array}}"'] \arrow[lu, "{l\,\rhd\,\id}"'] \\
(\1_\SV\rhd \1_\ST)\otimes (v\rhd \1_\ST) \arrow[uu, "{l^\rhd\,\otimes\,\id}"', pos = 0.2, bend left, shift left=18] \arrow[u, "\alp1"'] & & (\1_\SV\otimes v)\rhd \1_\ST \arrow[uu, "{l\,\rhd\,\id}", pos = 0.2, bend right, shift right=16]
\end{tikzcd}
        \end{equation*}
        which commutes, for all $v\in\SV$, due to the reasons indicated in boxes and
        \item $\varphi_0$ satisfies the triangle \eqref{phi0r}, which, after resolving the definitions from the lemma, becomes
        \begin{equation*}
        \begin{tikzcd}
(v\rhd\1_\ST)\otimes\1_\ST \arrow[rr, shift left = 6, shorten >= 2cm, shorten <= 2cm, "\com{\eqref{a1r}}" description, color = white]\arrow[d, "\com{\nat\alp1}" description, color = white] \arrow[rr, "r", bend left=15, shift left = 1] \arrow[r, "\alp1"] & v\rhd (\1_\ST\otimes\1_\ST) \arrow[d, shift right = 9, pos = 0.6, shorten <= 1cm, "\com{\eqref{a2lm}}" description, color = white]\arrow[r, "{\id\,\rhd\,r}"]& v\rhd\1_\ST\arrow[d, "\com{\funct\rhd}" description, color = white]\\
v\rhd(\1_\ST\otimes(\1_\SV\rhd\1_\ST)) \arrow[r, "{\id\,\rhd\,\alp2}"'] \arrow[ru, "{\id\,\rhd\,(\id\,\otimes\,l^\rhd)}"] & v\rhd(\1_\SV\rhd(\1_\ST\otimes\1_\ST)) \arrow[r, "m"'] \arrow[u, "{\id\,\rhd\,l^\rhd}"']\arrow[r, shift left = 6, pos = -0.4,  "\com{\eqref{triangleM}}" description, color = white] & (v\otimes\1_\SV)\rhd(\1_\ST\otimes\1_\ST) \arrow[d, "{\id\,\rhd\,r}"'] \arrow[lu, "{r\,\rhd\,\id}"', shift right = 1] \\
(v\rhd\1_\ST)\otimes(\1_\SV\rhd \1_\ST) \arrow[u, "\alp1"']\arrow[uu, "{\id\,\otimes\,l^\rhd}"', bend left, shift left=15, pos = 0.2] && (v\otimes\1_\SV)\rhd \1_\ST \arrow[uu, "{r\,\rhd\,\id}", pos = 0.2, bend right, shift right=16]
\end{tikzcd}
        \end{equation*}
        which commutes, for all $v\in\SV$, due to the reasons indicated in boxes and
    \end{itemize}
    \item the functor $F^Z$ is braided, since the square \eqref{braidedMonoidalFunctor} is satisfied, which, after resolving the definitions from the lemma, becomes
    \begin{equation*}
    \begin{tikzcd}
(v\rhd\1_\ST)\otimes (w\rhd\1_\ST) \arrow[d, shift left = 6, "\com{\dfn\beta}" description, color = white]\arrow[d, "\alp1"']  \arrow[rr, bend left = 15, "{\beta^{Z(\ST)}}", "=\,\beta"'] & v\rhd((w\rhd\1_\ST)\otimes\1_\ST) \arrow[d, shift right = 11, pos = 0.8, "\com{\eqref{defDrinfeldUnitMorphism}}" description, color = white]\arrow[d, "\com{\eqref{a1r}}" description, pos = 2.3]\arrow[d, "{\id\,\rhd\,r}"] \arrow[ldd, "{\id\,\rhd\,\alp1}", bend left=80, start anchor = south east, end anchor = south east] & (w\rhd\1_\ST)\otimes(v\rhd\1_\ST) \arrow[d, "\alp1"]  \arrow[l, "\alp2"] \\
v\rhd (\1_\ST\otimes (w\rhd\1_\ST)) \arrow[d, shift left = 8, pos = 0.3, "\com{\eqref{a2l}}" description, color = white]\arrow[d, "{\id\,\rhd\,\alp2}"'] \arrow[ru, "{\id\,\rhd\,\beta}"] \arrow[r, "{\id\,\rhd\,l}"'] & v\rhd(w\rhd\1_\ST) & w\rhd(\1_\ST\otimes(v\rhd\1_\ST))\arrow[d, shift right = 10, shorten <= 2cm, "\com{\eqref{brac}}" description, color = white] \arrow[d, "{\id\,\rhd\,\alp2}"]\\
v\rhd(w\rhd(\1_\ST\otimes\1_\ST))  \arrow[ru, "{\begin{array}{l}\id\,\rhd\,(\id\,\rhd\,l)\\=\,\id\,\rhd\,(\id\,\rhd\,r)\end{array}}"']&  & w\rhd(v\rhd(\1_\ST\otimes\1_\ST))  \\
(v\otimes w)\rhd(\1_\ST\otimes\1_\ST) \arrow[rrd, "\com{\funct\rhd}" description, color = white]\arrow[rr, "{\beta^\SV\,\rhd\,\id}"'] \arrow[d, "{\id\,\rhd\,r}"'] \arrow[u, "m"]&& (w\otimes v)\rhd(\1_\ST\otimes\1_\ST) \arrow[d, "{\id\,\rhd\,r}"] \arrow[u, "m"']\\
(v\otimes w)\rhd \1_\ST \arrow[rr, "{\beta^\SV\,\rhd\,\id}"']&& (w\otimes v)\rhd \1_\ST
\end{tikzcd}
    \end{equation*}
    which commutes, for all $v,w\in\SV$, due to the reasons indicated in boxes.
\end{itemize}

\section{Glossary of diagrams}\label{secAppendixGlossary}
\begin{itemize}[noitemsep,topsep=0pt,parsep=0pt,partopsep=0pt, align = parleft, left =0pt..2.5cm]
\item[\eqref{aL}] triangle automatically satisfied by monoidal categories\dotfill\pageref{aL}
\item[\eqref{aR}] triangle automatically satisfied by monoidal categories\dotfill\pageref{aR}
\item[\eqref{a1a}] pentagon for a module monoidal category\dotfill\pageref{a1a} 
\item[\eqref{a1a2}] hexagon for a module monoidal category\dotfill\pageref{a1a2}
\item[\eqref{a1lm}] triangle automatically satisfied by module monoidal categories \dotfill\pageref{a1lm}
\item[\eqref{a1m}] pentagon for a module monoidal category\dotfill\pageref{a1m}
\item[\eqref{a1r}] triangle automatically satisfied by unital module monoidal categories \dotfill\pageref{a1r}
\item[\eqref{a2a}] pentagon for a module monoidal category \dotfill\pageref{a2a}
\item[\eqref{a2l}] triangle automatically satisfied by unital module monoidal categories \dotfill\pageref{a2l}
\item[\eqref{a2lm}] triangle automatically satisfied by module monoidal category \dotfill\pageref{a2lm}
\item[\eqref{a2m}] pentagon for a module monoidal category \dotfill\pageref{a2m}
\item[\eqref{brac}] septagon for a module monoidal category\dotfill\pageref{brac}
\item[\eqref{braidedMonoidalFunctor}] square for a braided monoidal functor\dotfill\pageref{braidedMonoidalFunctor}
\item[\eqref{actionCoherenceAndBraiding}] hexagon for a pair\dotfill\pageref{actionCoherenceAndBraiding}
\item[\eqref{phi0l}] square for a monoidal functor\dotfill\pageref{phi0l}
\item[\eqref{phi0r}] square for a monoidal functor\dotfill\pageref{phi0r} 
\item[\eqref{braidingHexagon1}] hexagon for a braided monoidal category\dotfill\pageref{braidingHexagon1}
\item[\eqref{braidingHexagon2}] hexagon for a braided monoidal category\dotfill\pageref{braidingHexagon2}
\item[\eqref{diagModuleMonoidalFunctorCompAlpha1}] hexagon for a module monoidal functor\dotfill\pageref{diagModuleMonoidalFunctorCompAlpha1}
\item[\eqref{diagModuleMonoidalFunctorCompAlpha2}] hexagon for a module monoidal functor\dotfill\pageref{diagModuleMonoidalFunctorCompAlpha2}
\item[\eqref{composGamma}] composition formula for a morphism of pairs\dotfill\pageref{composGamma}
\item[\eqref{composPhi0}] composition formula for a monoidal functor\dotfill\pageref{composPhi0}
\item[\eqref{composPhi2}] composition formula for a monoidal functor\dotfill\pageref{composPhi2}
\item[\eqref{composS}] composition formula for a module functor\dotfill\pageref{composS}
\item[\eqref{defDrinfeldHalfBraiding}] hexagon for the half\--braiding\dotfill\pageref{defDrinfeldHalfBraiding}
\item[\eqref{defDrinfeldTensor}] hexagon for the half\--braiding on tensor product\dotfill\pageref{defDrinfeldTensor}
\item[\eqref{defDrinfeldUnitMorphism}] triangle for the half\--braiding on monoidal unit\dotfill\pageref{defDrinfeldUnitMorphism}
\item[\eqref{phi2Hexagon}] hexagon for a monoidal functor\dotfill\pageref{phi2Hexagon}
\item[\eqref{condition2MorphismOfPairs}] triangle for a 2\==morphism of pairs\dotfill\pageref{condition2MorphismOfPairs}
\item[\eqref{monNatTrafoPhi0}] triangle for a monoidal natural transformations\dotfill\pageref{monNatTrafoPhi0}
\item[\eqref{monNatTrafoPhi2}] square for a monoidal natural transformations\dotfill\pageref{monNatTrafoPhi2}
\item[\eqref{monNatTrafoS}] square for a module natural transformations\dotfill\pageref{monNatTrafoS}
\item[\eqref{pentA}] pentagon for a monoidal category\dotfill\pageref{pentA}
\item[\eqref{pentM}] pentagon for a module category\dotfill\pageref{pentM}
\item[\eqref{sTriangle}] triangle for a module functor\dotfill\pageref{sTriangle}
\item[\eqref{sPentagon}] pentagon for a module functor\dotfill\pageref{sPentagon}
\item[\eqref{triangleA}] triangle for a monoidal category\dotfill\pageref{triangleA}
\item[\eqref{triangleM}] triangle for a module category\dotfill\pageref{triangleM}
\item[\eqref{diagTri'}] triangle automatically satisfied by module categories\dotfill\pageref{diagTri'}
\item[\eqref{adjunctionZigzag}] zig\--zag\--identities for adjunctions\dotfill\pageref{adjunctionZigzag}
\item[\eqref{braidingV1}] triangle for braiding $\beta^\SV_{v,\1_\SV}$\dotfill\pageref{braidingV1}
\item[\eqref{braiding1V}] triangle for braiding $\beta^\SV_{\1_\SV,v}$\dotfill\pageref{braiding1V}
\item[\eqref{defDrinfeldMorphism}] square for a morphism in the Drinfeld center\dotfill\pageref{defDrinfeldMorphism}
\item[\eqref{defDrinfeldBraiding}] defining equation for braiding on Drinfeld center\dotfill\pageref{defDrinfeldBraiding}
\item[\eqref{pf1}] auxiliary diagram for the proof of Lemma \ref{lemUnitalModuleMonoidalCategoryToPair}\dotfill\pageref{pf1}
\item[\eqref{pf2}] auxiliary diagram for the proof of Lemma \ref{lemUnitalModuleMonoidalCategoryToPair}\dotfill\pageref{pf2}
\item[\eqref{pf3}] auxiliary diagram for the proof of Lemma \ref{lemUnitalModuleMonoidalCategoryToPair}\dotfill\pageref{pf3}
\item[\eqref{pf4}] auxiliary diagram for the proof of Lemma \ref{lemUnitalModuleMonoidalCategoryToPair}\dotfill\pageref{pf4} 
\end{itemize}

\newcommand{\arxiv}[2]{\href{http://arXiv.org/abs/#1}{#2}}
\newcommand{\doi}[2]{\href{http://doi.org/#1}{#2}}

\end{document}